\documentclass[11pt, a4paper]{article}
\title{Commutation relations for two-sided radial SLE}
\usepackage[T1]{fontenc}

\usepackage{mathtools}
\usepackage{lmodern}
\usepackage{amsthm,amsmath,amsfonts,amssymb}
\usepackage{graphicx}
\usepackage{enumerate,enumitem}
\usepackage{authblk}
\usepackage{cite,url}
\usepackage[normalem]{ulem}
\usepackage{comment}
\usepackage[margin=1cm]{caption}

\usepackage{hyperref}
\hypersetup{
    colorlinks=true, 
    linkcolor=black, 
    urlcolor=black, 
    citecolor=black,
    linktoc=all 
}

\allowdisplaybreaks

\setlist[enumerate]{topsep = 1ex, leftmargin=.5cm, itemsep= -3pt}
\setlist[itemize]{topsep = 1ex, leftmargin=.5cm, itemsep= -3pt}

\let\OLDthebibliography\thebibliography
\renewcommand\thebibliography[1]{
  \OLDthebibliography{#1}
  \setlength{\parskip}{1pt}
  \setlength{\itemsep}{2pt}
}

\newtheorem{thm}{Theorem}[section]
\newtheorem{cor}[thm]{Corollary}

\newtheorem{lem}[thm]{Lemma}
\newtheorem{prop}[thm]{Proposition}

\newtheorem{definition}[thm]{Definition}

\theoremstyle{definition} 
\newtheorem{df}[thm]{Definition}

\newtheorem{remark}[thm]{Remark}

\numberwithin{equation}{section}

\usepackage{longtable,booktabs}

\usepackage[font=normal]{caption}
\usepackage{floatrow}

\usepackage{mathtools}

\usepackage[dvipsnames]{xcolor}

\global\long\def\ud{\mathrm{d}}

\global\long\def\ii{\mathfrak{i}}

\newcommand{\abs}[1]{\left\lvert #1 \right \rvert}

\newcommand{\comm}[1]{\left[ #1 \right]}

\newcommand{\mc}[1]{\mathcal{#1}}
\newcommand{\m}[1]{\mathbb{#1}}

\newcommand{\mf}[1]{\mathfrak{#1}}

\def\SLE{\operatorname{SLE}}

\def\a{\alpha}
\def\b{\beta}
\def\g{\gamma}

\def\d{\delta}
\def\D{\Delta}

\def\t{\theta}

\def\l{\lambda}
\def\k{\kappa}

\def\vare{\varepsilon}
\def\HH{{\mathbb H}}
\def\PP{\mathbb{P}}

\def\eps{\varepsilon}

\def\dd{\mathrm{d}}

\newcommand{\ad}[1]{\overline{#1}}

\def\1{\mathbf{1}}

 \newcommand{\splus}{{\scriptstyle +}}

\global\long\def\CR{\mathrm{CR}}
\global\long\def\E{\mathbb{E}}

\global\long\def\cond{\;|\;}

\def\mbt{\mathbf{t}}
\def\capa{\operatorname{cap}}
\def\ee{\mathrm{e}}
\def\LG{\mathcal{G}}
\def\LZ{\mathcal{Z}}
\usepackage[dvipsnames]{xcolor}

\author{Ellen Krusell\thanks{\protect\url{ekrusell@kth.se}
Department of Mathematics, KTH Royal Institute of Technology, 114 28 Stockholm, Sweden}\qquad
Yilin Wang\thanks{\protect\url{yilin.wang@math.ethz.ch} Department of Mathematics, Eidgen\"ossische Technische Hochschule, 8092 Z\"urich, Switzerland}
\qquad
Hao Wu\thanks{\protect\url{hao.wu.proba@gmail.com}
Yau Mathematical Sciences Center, Tsinghua University, Beijing 100084, China}}
\begin{document}

\maketitle

\begin{abstract}
We study the commutation relation for 2-radial SLE in the unit disc starting from two distinct boundary points. 
We follow the framework introduced by Dub\'{e}dat~\cite{Dub_comm}. 
Under an additional requirement of the interchangeability of the two curves, we classify all locally commuting 2-radial SLE$_\kappa$ for $\kappa\in (0,8)$: it is either two-sided radial SLE$_\kappa$ with spiral of constant spiraling rate or chordal SLE$_\kappa$ weighted by a power of the conformal radius of its complement. 
Namely, for fixed $\k$ and starting points, we have exactly two one-parameter families of locally commuting 2-radial SLE curves. 
We also discuss the semiclassical limit of the commutation relation as $\kappa \to 0$. In particular, we show that the limit for the second family with an appropriately chosen power of the conformal radius is a chord that minimizes a modified chordal Loewner energy, which is unique only when the endpoints are not antipodal. \\

\end{abstract}

\section{Introduction}

\subsection{Background on radial SLE}
In 1999, O. Schramm \cite{Schramm2000} introduced the Schramm-Loewner evolution (SLE)   
as the non-self-crossing random curve driven by a multiple of Brownian motion using Loewner's transform.  This definition is motivated by a quest to mathematically describe the random interfaces in two-dimensional critical lattice models, which satisfy \emph{conformal invariance} and the \emph{domain Markov property}.  These properties impose that the chordal Loewner driving function of such a random interface has to be a multiple of Brownian motion, hence justifying the definition of chordal SLE. Indeed, SLE curves are proved to be the scaling limit of interfaces in many conformally invariant statistical mechanics models, e.g., \cite{Schramm:ICM,LSW04LERWUST,SS09GFF,CDHKS, SmirnovPerco, CamiaNewmanPerco}, and play a central role in random conformal geometry. 

However, to characterize the other natural variant --- radial SLE --- one needs an additional condition on the reflection symmetry. Since this article focuses on variants of radial SLE, let us briefly describe its definition and characterization. 
We will describe the radial Loewner chain in $\m D = \{z \in \m C \,|\, |z| < 1\}$ targeting $0$. The radial Loewner chain in another simply connected domain $D \subset \m C$ targeting an interior point $z_0 \in D$ is defined via a conformal map from $\m D$ onto $D$ sending $0$ to $z_0$.

\textbf{Conformal radius and capacity.}
  For any compact subset $K$ (not necessarily connected) of $\ad {\m D}$ such that $\m D \setminus K$ is simply connected and contains $0$,  we let $g_K$ be the unique conformal map $\m D \setminus K  \to \m D$ such that $g_{K} (0) = 0$ and $g'_{K} (0) > 0$ (called \emph{the radial mapping-out function} of $K$).   
  The \emph{conformal radius} of $\m D\setminus K$ is 
  \[\CR(\m D\setminus K):=(g'_K(0))^{-1}.\]
  The \emph{capacity} of $K$ is 
  \[\capa(K)=\log g'_K (0) =-\log\CR(\m D\setminus K).\]

\textbf{Radial Loewner chain.}
  For $\theta\in [0,2\pi)$, suppose $\eta: [0,T]\to \overline{\m D}$ is a continuous non-self-crossing curve such that $\eta_0=\ee^{\ii \theta}$ and $\eta_{(0,T)}\subset\m D\setminus\{0\}$. Let $D_t$ be the connected component of $\m D\setminus\eta_{[0,t]}$ containing the origin. Let $g_t: D_t\to \m D$ be the unique conformal map with $g_t(0)=0$ and $g_t'(0)>0$. We say that the curve is parameterized by capacity if $g_t'(0)=\ee^t$. Then $g_t$ satisfies the radial Loewner equation
\begin{equation}\label{eqn::radialLoewner}
    \partial_t g_t(z)=g_t(z)\frac{\ee^{\ii \xi_t}+g_t(z)}{\ee^{\ii \xi_t}-g_t(z)}, \quad g_0(z)=z,
\end{equation}
where $t\mapsto \xi_t  \in \m R /2\pi \m Z$ is continuous and called the \emph{driving function} of $\eta$. We note that if $z = \ee^{\ii V_0} \in \partial \m D$, then taking a continuous branch $V_t := \arg g_t (\ee^{\ii \t})$, we have
\begin{equation}\label{eq::loewner_boundary}
    \partial_t V_t (z) = \cot ((V_t - \xi_t)/2).
\end{equation}

\textbf{Characterization of radial SLE.} Consider a family $(\m P_\t)_{\t \in \m R/2\pi \m Z}$ of probability distributions on curves $\eta : [0,\infty) \to \ad{\m D}$ with $\eta_0 = \ee^{\ii \t}$ and parametrized by capacity that satisfies the following properties:
\begin{enumerate}
    \item (\emph{Conformal invariance})  For all $a \in \m R / 2\pi \m Z$, let $\rho_a (z)  = \ee^{\ii a} z$ be the rotation map $\m D \to \m D$. For all $a, \t \in \m R/2\pi \m Z$, the pullback measure $\rho_a^* \,\m P_\t = \m P_{\t - a}$. From this, we may extend the definition of $\m P_\t$ to any simply connected domain $D$ with an interior marked point $z_0$ by pulling back using a uniformizing conformal map $D \to \m D$ sending $z_0$ to $0$.
    \item (\emph{Domain Markov property}) For any $t >0$, $\t \in \m R / 2\pi \m Z$, let $\eta \sim \m P_\t$. Conditioned on $\eta|_{[0,t]}$, $g_t (\eta_{[t,\infty)}) \sim \m P_{\xi_t}$ where $\xi_\cdot$ is the driving function of $\eta$. See Figure~\ref{fig::radial11}.
    \item (\emph{Reflection symmetry}) Let $\iota : z\mapsto \ad z$ be the complex conjugation, then $\m P_\t \sim \iota^* \m P_{-\t}$.
\end{enumerate}
Then there exists $\k \ge 0$ such that for all $\t$, the driving function of $\eta \sim \m P_\t$ is $t\mapsto \t + \sqrt \k B_t$ modulo $2\pi \m Z$, where $B_\cdot$ is the standard Brownian motion. This follows from the fact that $(\sqrt \k B)_{\k \ge 0}$ are the only continuous L\'evy processes $W$ that have the same law under the transformation $W \mapsto - W$.
In this case, as $t\to \infty$, $\eta_t \to 0$ almost surely. The distribution $\m P_\t$ is the law of the \emph{radial $\SLE_{\kappa}$} in $\m D$ starting from $\ee^{\ii \t}$. We also call it the radial $\SLE_{\kappa}$ in $(\m D; \ee^{\ii\t}; 0)$ for short. 

\begin{figure}[ht!]
    \begin{center}    \includegraphics[width=0.6\textwidth]{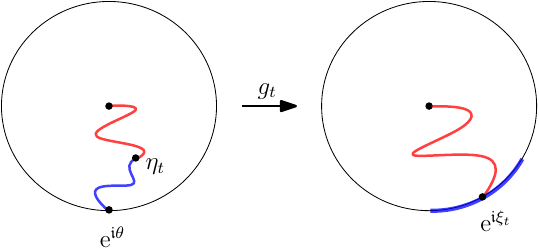}
    \end{center}
    \caption{Domain Markov property of radial SLE.}
    \label{fig::radial11}
\end{figure}

The third assumption on the reflection symmetry is natural for conformally invariant and achiral statistical mechanics models. However, one may also wonder what happens without the reflection symmetry. This means that the law of the driving function is no longer required to be invariant with respect to $W \mapsto - W$. 
From the classification of continuous L\'evy processes, we obtain that there exist $\k \ge 0$ and $\mu \in \m R$ such that for all $\t$, the driving function of $\eta \sim \m P_\t$ is 
$$t\mapsto \t + \sqrt \k B_t + \mu t \qquad \text{ (mod) } 2\pi \m Z.$$ 
The curve generated is called \emph{radial SLE$_\k$ with spiraling rate $\mu$} starting from $\ee^{\ii \t}$ and we denote it as radial $\SLE_\k^\mu$ in $(\m D;\ee^{\ii \t}; 0)$. It was shown in \cite{IG4} that almost surely, $\eta_t \to 0$ as $t \to \infty$ for all $\k > 0$ and $\mu \in \m R$.

\subsection{Locally commuting 2-radial SLE curves}
\label{subsec::intro_axioms}
 The random radial curves satisfying conformal invariance and the domain Markov property are easily characterized  
 thanks to the fact that all simply connected domains (excluding $\m C$) with one marked interior point and one marked boundary point are conformally equivalent.
  If we consider the radial Loewner chain of two curves growing from two distinct boundary points $\ee^{\ii \t_1}, \ee^{\ii \t_2} \in \partial \m D$, then we have a one-dimensional moduli space of the boundary data that we need to take into account.

 The goal of this work is to give an axiomatic characterization of locally commuting 2-radial SLE curves. One example of such curves of particular interest is the two-sided radial SLE analyzed in~\cite{Lawler2sidedSLE,field2016twosided, Lawler_Zhou_natural,HealeyLawlerNSidedSLE} which arises as the chordal SLE in $\m D$ connecting two boundary points conditioned on passing through $0$.
 
 More precisely, a \emph{locally commuting 2-radial SLE} is a family of laws $\m P_{(\t_1, \t_2)}$ of the pair of continuous non-self-crossing curves $(\eta^{(1)}, \eta^{(2)})$ in $\m D$ starting from all possible choices of $\ee^{\ii \t_1}, \ee^{\ii \t_2} \in \partial \m D$ under the following axioms. 

We first parameterize $\eta^{(1)}, \eta^{(2)}$ by their intrinsic radial capacities and consider the associated radial Loewner chains $g_{\cdot}^{(1)}$ and $g_{\cdot}^{(2)}$ respectively. 
    For all $\mbt = (t_1, t_2) \in \m R_{>0}^2$, let $g_\mbt$ be the radial mapping-out function of $\eta_{\mbt} : = \eta^{(1)}_{[0,t_1]} \cup \eta^{(2)}_{[0,t_2]}$.
\begin{itemize}[leftmargin=3em]
\item[\textbf{(CI)}] \textbf{Conformal invariance:} For all $a \in \m R / 2\pi \m Z$ and  $\t_1, \t_2 \in \m R/2\pi \m Z$, the pullback measure $\rho_a^* \,\m P_{(\t_1, \t_2)} = \m P_{(\t_1 - a, \t_2 - a)}$.  From this, we may extend the definition of $\m P_{\t_1, \t_2}$ to any simply connected domain $D$ with an interior marked point $z_0$ by pulling back using a uniformizing conformal map $D \to \m D$ sending $z_0$ to $0$.
    \item[\textbf{(DMP)}]\textbf{Domain Markov property: }
    Conditioned on $\eta_\mbt$, 
    \[\left(g_\mbt (\eta^{(1)} \setminus \eta^{(1)}_{[0,t_1]} ),g_\mbt(\eta^{(2)}\setminus \eta^{(2)}_{[0,t_2]}) \right) \sim \m P_{(\t^{(1)}_\mbt,\t^{(2)}_\mbt)},\]where $\left(\exp(\ii \t^{(1)}_\mbt), \exp(\ii \t^{(2)}_\mbt)\right) = \left(g_\mbt (\eta^{(1)}_{t_1}),  g_\mbt (\eta^{(2)}_{t_2})\right)$. 
    \item[\textbf{(MARG)}]
    \textbf{Marginal laws:}
    There exists $\k > 0$ such that the marginal local law of $\eta^{(j)}$ is ``absolutely continuous'' with respect to $\SLE_\k$ for $j = 1,2$.
    \item[\textbf{(INT)}] \textbf{Interchangeability condition:} Let $\tau$ be the map which swaps $\eta^{(1)}$ and $\eta^{(2)}$. We have $\m P_{(\t_1, \t_2)} = \tau^* \m P_{(\t_2, \t_1)}$.
\end{itemize}

We note that, by ``radial'' we mean that the conformal mapping-out functions are normalized at the interior point $0 \in \m D$ and the precise axioms are written in neighborhoods of the starting points of the curves; see Section~\ref{sec:axioms}.  In particular, they do not necessarily imply that the curves end at $0$.

We also note that since $g_\mbt$ is parametrized by two times $t_1, t_2$, the condition \textbf{(DMP)} implies that $g_\mbt$ may be realized by first mapping out $\eta^{(1)}_{[0,t_1]}$ using $g_{t_1}^{(1)}$, then mapping out $g_{t_1}^{(1)} (\eta^{(2)}_{[0,t_2]})$, or vice versa. The image has the same law regardless of the order in which we map out the curves. This observation gives a \textbf{commutation relation} on the infinitesimal generators of the two curves (Proposition~\ref{prop:radial_comm}). Therefore, we call the family of laws $(\m P_{(\t_1, \t_2)})$ an \emph{interchangeable and locally commuting 2-radial SLE$_\k$}. 

These conditions and our analysis are very close to the commutation relation of SLE studied by Dub\'edat in~\cite{Dub_comm}, which focuses on classifying all locally commuting chordal SLE curves.
Dub\'edat also derived the commutation relation in the radial setting in terms of the generators of radial SLE.  
One of our contributions is to find all solutions to the radial commutation relation and identify all interchangeable and locally commuting 2-radial SLE$_{\kappa}$.
It is also natural to use SLE/GFF couplings, particularly, \cite{IG4}, to find examples of commuting $2$-radial SLEs. See Remark~\ref{rem:coupling_GFF}.  Nevertheless, our approach is to give an axiomatic characterization by studying the Belavin-Polyakov-Zamolodchikov (BPZ) equations~\eqref{eqn::BPZ_radial} satisfied by the partition functions, and it is unclear to the authors whether \emph{all} commuting $2$-radial SLEs can be coupled to GFF a priori.

\subsection{Main result: A classification}

We classify all locally commuting 2-radial SLEs with the interchangeability condition \textbf{(INT)}. 
Similar to the analysis in \cite{Dub_comm}, we also show that the conditions \textbf{(CI), (DMP)}, and \textbf{(MARG)} imply that there exists $\mc Z : \{(\t_1, \t_2) \in \m R^2 \,|\, \t_1 < \t_2 < \t_1 +2 \pi\} \to \m R_{>0}$, called a \textit{partition function} which encodes the family of distributions $\m P_
{(\t_1, \t_2)}$. See Section~\ref{sec:partition} for more details.

\begin{thm}\label{thm:main}
Fix $\kappa\in (0,8)$.
Suppose $\mc Z$ is the partition function for an interchangeable and locally commuting 2-radial $\SLE_{\kappa}$. Then $\mc Z$ is one of the following two functions: 
    \begin{enumerate}
        \item There exists $\mu\in\m R$ such that, up to a multiplicative constant, $\mc Z$ is the same as 
        \begin{equation}\label{eqn::2SLEspiral_pf}
    \LG_{\mu}(\theta_1, \theta_2)=\left(\sin\left((\theta_{2}-\theta_1)/2\right)\right)^{2/\k} \exp\left(\frac{\mu}{\kappa}(\theta_1+\theta_2)\right).
\end{equation}
In this case, the law of the corresponding locally commuting 2-radial $\SLE_{\kappa}$ is the same as that of two-sided radial $\SLE_{\kappa}$ with spiraling rate $\mu$; see Section~\ref{subsec::2SLEspiral} and Figure~\ref{fig::simulation}. 
        \item There exists $\alpha<1-\kappa/8$ such that, up to a multiplicative constant, $\mc Z$ is the same as 
\begin{equation}\label{eqn::CR_pf}
    \mc Z_{\alpha}(\theta_1, \theta_2)=\left(\sin\left((\theta_{2}-\theta_1)/2\right)\right)^{(\kappa-6)/\kappa}\phi_{\alpha}\left(\left(\sin\left((\theta_{2}-\theta_1)/4\right)\right)^2\right),
\end{equation}
where $\phi_{\alpha}$ is the unique solution to the following Euler's hypergeometric differential equation
\begin{equation}\label{eqn::Euler_initial}
\begin{cases}
u(1-u)\phi''(u)-\frac{3\kappa-8}{2\kappa}(2u-1)\phi'(u)+\frac{8\alpha}{\kappa}\phi(u)=0, \quad u\in (0,1);\\
\phi(1/2)=1, \quad \phi'(1/2)=0.
\end{cases}
\end{equation}
In this case, the law of the corresponding locally commuting 2-radial $\SLE_{\kappa}$ is the same as that of $\gamma$, the chordal $\SLE_{\kappa}$ in $\m D$ weighted by 
$\CR(\m D\setminus\gamma)^{-\alpha}$, 
where $\CR(\m D\setminus\gamma)$ denotes the conformal radius of the connected component of $\m D\setminus \gamma$ containing the origin;
see Section~\ref{subsec::chordalSLE_CR}. 
    \end{enumerate}
\end{thm}

The first class above --- two-sided radial SLEs with spiral --- also appeared in \cite{IG4} in the context of SLE/GFF coupling. We will discuss other related literature in Section~\ref{sec:comments}. The authors are unaware of studies on the second class of curves.

Both $\LG_{\mu}$ in~\eqref{eqn::2SLEspiral_pf} and $\mc Z_{\alpha}$ in~\eqref{eqn::CR_pf} are solutions to radial BPZ equations (or called BPZ--Cardy equations derived in~\cite{KangMakarov2012,KangMakarovGFFCFT}; see also \cite{BBK, Dub_comm} for the BPZ equation for multichordal SLE): 
\begin{align}\label{eqn::BPZ_radial}
\begin{split}
    \frac{\kappa}{2}\frac{\partial_{11}\mc Z}{\mc Z}+\cot\left((\theta_{2}-\theta_1)/2\right)\frac{\partial_2 \mc Z}{\mc Z}-\frac{(6-\kappa)/(4\kappa)}{\left(\sin((\theta_{2}-\theta_1)/2)\right)^2}&= \aleph,\\
    \frac{\kappa}{2}\frac{\partial_{22}\mc Z}{\mc Z}-\cot\left((\theta_{2}-\theta_1)/2\right)\frac{\partial_1 \mc Z}{\mc Z}-\frac{(6-\kappa)/(4\kappa)}{\left(\sin((\theta_{2}-\theta_1)/2)\right)^2}& = \aleph,
    \end{split}
\end{align}
where $\partial_i$ is the partial derivative with respect to $\theta_i$ and $\aleph$ is the constant: 
\begin{align*}
    \aleph & =\frac{\mu^2-3}{2\kappa}\ge -\frac{3}{2\kappa}, \quad &&\text{when }\mc Z=\LG_{\mu},\\
    \aleph & = \frac{(6-\kappa)(\kappa-2)}{8\kappa}-\alpha>-\frac{3}{2\kappa},\quad&&\text{when }\mc Z=\mc Z_{\alpha}. 
\end{align*}

In general, the partition function $\mc Z_{\alpha}$ in~\eqref{eqn::CR_pf} involves hypergeometric functions and is not explicit; but in the following three special cases, it has a simple explicit expression:
\begin{itemize}
\item When $\alpha=0$, the solution to~\eqref{eqn::Euler_initial} is $\phi=1$. In this case $\mc Z_{\alpha}=\left(\sin\left((\theta_{2}-\theta_1)/2\right)\right)^{(\kappa-6)/\kappa}$. This is the partition function for chordal SLE; see Section~\ref{subsec::chordalSLE_CR}.
\item When $\alpha=\alpha_1(\kappa)$, which is the one-arm exponent for the conformal loop ensemble, $\mc Z_{\alpha_1(\kappa)}$ can be written as trigonometric functions in $\theta=\theta_2-\theta_1$; see Remark~\ref{rem::onearm}. 
\item When $\kappa=4$, $\mc Z_{\alpha}$ can be written as trigonometric functions and hyperbolic functions in $\theta=\theta_2-\theta_1$; see Remark~\ref{rem::CR_kappa4}.    
\end{itemize}

\begin{figure}[ht]
    \begin{center}    
    \includegraphics[width=0.48\textwidth]{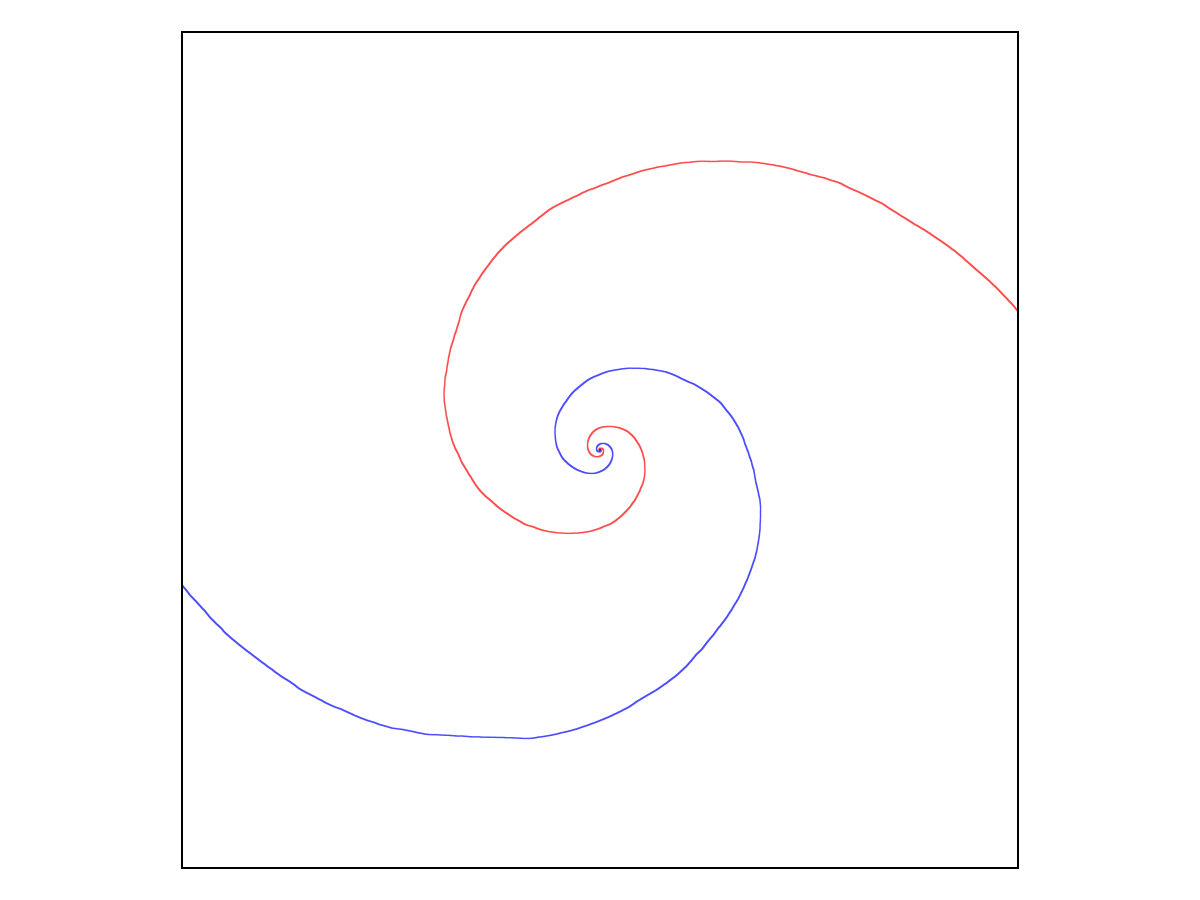}
    \includegraphics[width=0.48\textwidth]{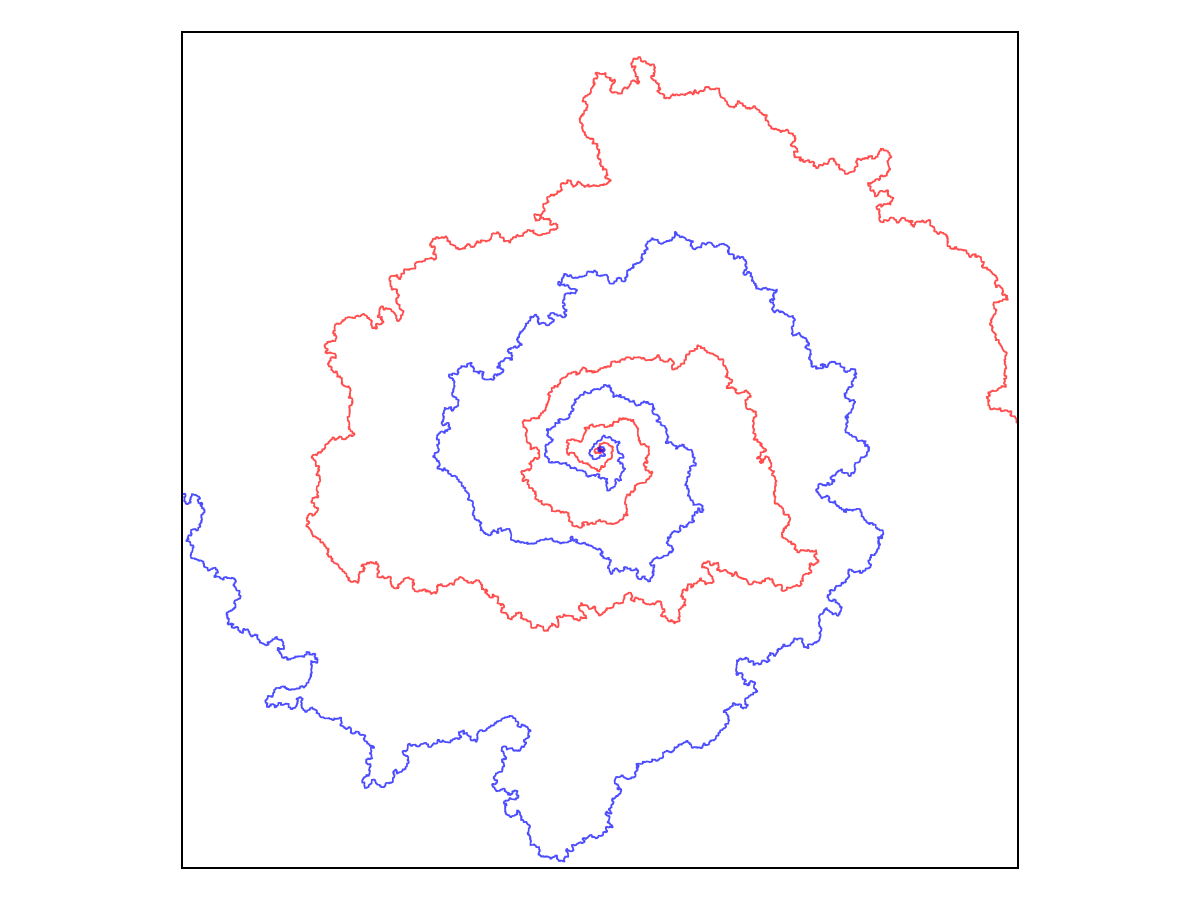}
    \end{center}
    \caption{Simulation by Minjae Park. A two-sided radial $\SLE_\k$ with spiral with $\kappa=0.01$ and $\mu=5$ in the left panel, and $\kappa=2$ and $\mu=5$ in the right panel.}
    \label{fig::simulation}
\end{figure}

The interchangeability condition \textbf{(INT)} is natural, given the reversibility of chordal SLE. It also simplifies our classification and, in particular, the notation; see Section~\ref{subsec::commutation_relation_INT} and, in particular, Lemma~\ref{lem:F_12_equal}. In Remark~\ref{rem::cr_left}, we give an example of locally commuting radial $2$-SLE without \textbf{(INT)}. In Proposition~\ref{prop:remove_int}, we show the classification without the condition \textbf{(INT)}, where the second one-parameter family of solutions \eqref{eqn::CR_pf} becomes a two-parameter family.

\subsection{Resampling property}
\label{subsec::intro_resampling}
Let $D \subset \m C$ be a simply connected domain and let $x_1, x_2$ be distinct prime ends of $\partial D$. 
We denote by $\mf X(D; x_1, x_2)$ the set of continuous non-self-crossing curves in $D$ from $x_1$ to $x_2$. 
The chordal SLE in $(D; x_1, x_2)$ is a random curve in $\mf X(D; x_1, x_2)$ defined by mapping the chordal SLE in $(\HH; 0, \infty)$ conformally from $\HH$ onto $D$ sending $0$ to $x_1$ and $\infty$ to $x_2$. 
Let $D \subset \m D$ be a simply connected domain containing $0$  and let $x_1, x_2$ be distinct prime ends of $\partial D$. 
We denote by $\mf X(D; x_1, x_2; 0)$ the set of pairs of continuous simple curves $(\eta^{(1)}, \eta^{(2)})$ in $D$ such that $\eta^{(1)}$ goes from $x_1$ to $0$ and $\eta^{(2)}$ goes from $x_2$ to $0$ and $\eta^{(1)}\cap\eta^{(2)}=\{0\}$. We say that a law on $(\eta^{(1)}, \eta^{(2)})\in\mf X(D; x_1, x_2; 0)$ satisfies the \textit{resampling property} if 
\begin{itemize}
    \item the conditional law of $\eta^{(2)}$ given $\eta^{(1)}$ is the chordal $\SLE_{\kappa}$ in $(D\setminus\eta^{(1)}; x_2, 0)$,
    \item and the conditional law of $\eta^{(1)}$ given $\eta^{(2)}$ is the chordal $\SLE_{\kappa}$ in $(D\setminus\eta^{(2)}; x_1, 0)$.
\end{itemize}

Resampling properties in the chordal setting are widely investigated; see e.g.~\cite{IG2, IG3, MSW, BPW, Zhan23}. In this article, we focus on the resampling property in the radial setting as described above.
In Theorem~\ref{thm::resampling_property} we show that for each $\mu\in\mathbb{R}$, the two-sided radial $\SLE_{\kappa}$ with spiraling rate $\mu$ satisfies the resampling property for $\kappa\in (0,4]$. This result follows directly from the expression of the partition function $\mc G_\mu$. Another proof appears in \cite{IG4} using a coupling to GFF. Combining Theorem~\ref{thm::resampling_property} with Theorem~\ref{thm:main}, we find that the commutation relation implies the resampling property.

\begin{cor}\label{cor::commutation_resampling}
Fix $\kappa\in (0,4]$ and $\theta_1<\theta_2<\theta_1+2\pi$. 
If an interchangeable and locally commuting 2-radial $\SLE_{\kappa}$ is in $\mf X (\m D;\ee^{\ii\theta_1}, \ee^{\ii\theta_2}; 0)$, then it is a two-sided radial $\SLE_\k$ with spiral, and hence, it satisfies the resampling property. The converse is not true: a linear combination of two-sided radial $\SLE_\k$ with different spiraling rates is in $\mf X (\m D;\ee^{\ii\theta_1}, \ee^{\ii\theta_2}; 0)$ and satisfies the resampling property but is not a locally commuting 2-radial $\SLE_\k$.
\end{cor}

\subsection{Semiclassical limits of partition functions} \label{subsec:intro_semicl}

\begin{figure}[ht!]
    \begin{center}    
\includegraphics[width=0.35\textwidth]{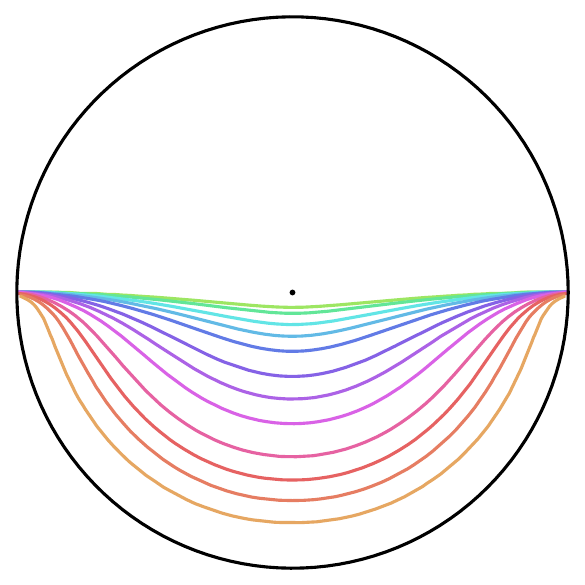}
\includegraphics[width=0.35\textwidth]{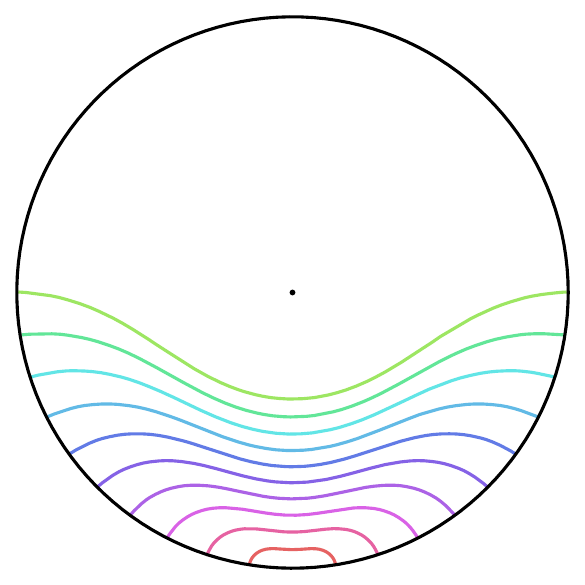}
    \end{center}
    \caption{The curves generated by $\mc U^\l$ from Proposition~\ref{prop::semiclassical_pf}. They correspond to the energy minimizers from Proposition \ref{prop:confRadMinimizer}.
    On the left, $\t_2-\t_1\approx \pi$, and $\l=0.1,\ 0.2,\ 0.5,\ 1,\ 2,\ 5,\ 10,\ 20,\ 50,\ 100,\ 200,\ 500$ ($\l=0.1$ corresponds to the innermost green curve and $\l=500$ corresponds to the outermost orange curve). On the right $\l=10$ while $\t_2-\t_1$ varies.}
    \label{fig::minimizer}
\end{figure}

We also discuss the commutation relation when $\k = 0$,  the corresponding deterministic pair of curves, and the semiclassical limit $\k \to 0\splus$ in Section~\ref{sec:semiclassical}. In particular, the semiclassical limits of partition functions in Theorem~\ref{thm:main} have explicit formulas. 

\begin{prop}\label{prop::semiclassical_pf}
Fix $\theta_1<\theta_2<\theta_1+2\pi$ and write $\theta=\theta_2-\theta_1$. 
\begin{itemize}
    \item For the partition function $\LG_{\mu}$ in~\eqref{eqn::2SLEspiral_pf}, with  $\mu\in\mathbb{R}$ fixed, we have 
    \begin{equation}\label{eqn::classical_mu}
    \lim_{\kappa\to 0}\kappa\log\LG_{\mu}(\theta_1, \theta_2)=2\log\sin\big((\theta_{2}-\theta_1)/2\big)+\mu(\theta_1+\theta_2). 
\end{equation}
\item For the partition function $\LZ_{\alpha}$ in~\eqref{eqn::CR_pf}, if $\alpha\sim -\lambda/\kappa$ for some $\lambda> 0$, then the following limit exists  
\begin{align*}
    \mc U^\l (\t)=&\lim_{\kappa\to 0}\kappa\log\LZ_{\alpha}(\theta_1, \theta_2),
\end{align*}
and similarly, if $\alpha<0$ and $\alpha=o(1/\kappa)$, then
\begin{align*}
    \mc U^0 (\t)=&\lim_{\kappa\to 0}\kappa\log\LZ_{\alpha}(\theta_1, \theta_2)
\end{align*}
exists;
and for $\theta\in (0,\pi]$, 
 \begin{align}\label{eq:U_l_explicit}
   & \mc U^{\lambda}(\theta) 
    =  \mc U^{\lambda}(2\pi-\theta) = -2\log\sin(\theta/2)
    +\int_{\theta}^{\pi}\sqrt{2\lambda+4\cot^2(u/2)}\, \ud u.  
\end{align}
\end{itemize}    
\end{prop}
We prove Proposition~\ref{prop::semiclassical_pf} in Section~\ref{sec:semiclassical}. 
We mention that a similar semiclassical limit of the partition function for multichordal SLE was considered in \cite{peltola_wang,alberts2020pole} and that multiradial $\SLE_0$ was considered in \cite{AbuzaidHealeyPeltolaLargeDeviationDysonBM, ZhangThesis} (see also Section~\ref{sec:comments}).

\begin{prop}\label{prop::semiclassical_minimizer}
We use the same notation as in Proposition~\ref{prop::semiclassical_pf}. The function $\mc U^\lambda$ can be expressed as
$$\mc U^\lambda (\t) = -6 \log \sin (\t/2) - \inf_\g \Big(I(\g) - \lambda \log \CR(\m D \setminus \g)\Big) + C$$
where the infimum is attained and taken over all chords connecting $\ee^{\ii \t_1}$ and $\ee^{\ii \t_2}$, $I(\cdot)$ is the chordal Loewner energy, and $C$ is a normalizing constant depending only on $\lambda$ such that $\mc U^\lambda (\pi) = 0$: 
\[C=\int_0^{\pi}\Big(\sqrt{2\lambda+4\cot^2(u/2)}-2\cot(u/2)\Big)\, \dd u.\]
Moreover, for $\lambda>0$ the minimizer is unique when $\t \neq \pi$ and there are exactly two minimizers when $\t = \pi$. For $\lambda=0$, the minimizer is unique for all $\t\in(0,2\pi)$.
\end{prop}
See Lemma~\ref{lem::semiclassical_Zalpha}
for details and Figure~\ref{fig::minimizer} for a simulation of those energy minimizers. A related relation between the semiclassical limit of the partition function for multichordal SLE and the minimizers of the multichordal Loewner energy was proved in \cite{peltola_wang}.
We point out in particular that $\mc U^{\lambda}$ in~\eqref{eq:U_l_explicit} when $\lambda>0$ is not differentiable at $\theta=\pi$ since the minimizer of $I(\cdot) - \lambda \log \CR(\m D \setminus \cdot)$ is non-unique when $\t = \pi$. Such non-differentiable functions do not appear in earlier literature about the semiclassical limits of SLE partition functions~\cite{peltola_wang,alberts2020pole, AbuzaidHealeyPeltolaLargeDeviationDysonBM, ZhangThesis}.

\subsection{Comments and further developments}\label{sec:comments}
Let us first make some comments on the works related to two-sided radial SLE with spiral. Apart from \cite{IG4} mentioned above in the context of imaginary geometry, we believe it to be also related to the mixed multifractal spectrum introduced by Binder \cite{binder1998harmonic}. More precisely, the points on an SLE curve are classified in terms of the asymptotic behavior of the uniformizing conformal maps in the complement of the curve, according to its H\"older exponent and its winding exponent. The Hausdorff dimension (i.e., mixed multifractal spectrum) of the set of points with a given behavior is computed in \cite{duplantier2002harmonic}. 
It was shown in \cite{Lawler_Zhou_natural} that conditioning a chordal SLE to pass through an interior point $0$ gives the two-sided radial SLE (without spiral). Since the points with zero winding exponent have the maximal spectrum, this conditioning is equivalent to conditioning on the event where $0$ is a point on the SLE curve with zero winding exponent. 
We speculate that ``conditioning'' on the rare event where $0$ is a point with the winding rate $\mu$, the curve obtained should satisfy the commutation relation and hence has to be the two-sided radial SLE with spiraling rate $\mu$. In the terminology of the conformal field theory, this would correspond to inserting a curve-generating operator at $0$ with complex
charges and conformal weights \cite{Belikov08}.

There are also a few recent works built upon this article.
A follow-up paper~\cite{HuangPeltolaWuMultiradialSLEResamplingBP} 
generalizes the first class in Theorem~\ref{thm:main}---two-sided radial SLE with spiral---to the multi-sided case and derives its resampling and boundary perturbation properties. 
Another follow-up paper~\cite{FengWu2024} generalizes the second class in Theorem~\ref{thm:main} to the multi-chordal case and connects the corresponding partition functions with $\alpha=\alpha_1(\kappa)$ to a random-cluster model in polygons conditional on the one-arm event. 
Recent papers \cite{ZhangThesis, ZhangKappaPos} of J. Zhang continue the investigation of the radial commutation by studying locally commuting $n$-radial SLE and the corresponding radial BPZ equations. Zhang constructs solutions to these radial BPZ equations using  Coulomb gas formalism and relates these solutions to quantum Calogero--Sutherland systems in conformal field theory. 
A follow-up paper~\cite{HuangWuCommutationRadial} gives a classification of radial BPZ systems in general.

Regarding related works on semiclassical limits of SLEs, apart from those mentioned in Section~\ref{subsec:intro_semicl}, we also mention that the resampling property in the deterministic $\k = 0$ case is equivalent to the \emph{geodesic property} studied in \cite{MRW1,Janne} (and the two-sided radial SLE$_0$ is a geodesic pair in their terminology). See also~\cite{MRW2,peltola_wang,Bonk_Eremenko} for other occurrences of the geodesic property. Zhang also studied the $\k = 0$ dynamics of multiple radial SLE$_0$ \cite{ZhangThesis} in relation to the classical Calogero--Sutherland system (see also \cite{alberts2020pole} in the multichordal case).

\section{Partition functions of locally commuting 2-radial SLEs}\label{sec:partition}

In this section, we give the precise definition of locally commuting $2$-radial SLEs. We also define and classify the partition functions. 

\subsection{Definition of locally commuting 2-radial SLE}\label{sec:axioms}
In this section, we summarize the axioms in Dub\'edat's commutation relation~\cite{Dub_comm} for the radial setting.
 Let $D \subset \m D$  be a simply connected domain containing $0$ and let $x_1, x_2$ be distinct prime ends of $D$.  Let $U_1, U_2$ be, respectively, closed neighborhoods of $x_1$ and $x_2$ in $D$ that do not contain $0$ and such that $U_1 \cap U_2 = \emptyset$. 
We will consider probability measures $\m P_{(D;\,x_1,x_2)}^{(U_1, U_2)}$ on pairs of unparametrized continuous curves in $U_1$ and $U_2$ starting from $x_1$ and $x_2$, and exiting $U_1$ and $U_2$ almost surely.
We call such a family of measures indexed by different choices of $(U_1, U_2)$ \emph{compatible} if for all $U_1 \subset U_1'$ and $U_2 \subset U_2'$, we have that 
$\m P_{(D;\,x_1,x_2)}^{(U_1, U_2)}$ is obtained by restricting the curves under $\m P_{(D;\,x_1,x_2)}^{(U_1', U_2')}$ to the parts before their first exiting of the subdomains $U_1$ and $U_2$.

A \emph{locally commuting $2$-radial} $\SLE_\k$ is a compatible family of measures $\m P_{(D;\,x_1,x_2)}^{(U_1, U_2)}$  on pairs of continuous non-self-crossing curves $(\eta^{(1)}, \eta^{(2)})$ for all $D$, $(x_1, x_2)$, and $(U_1, U_2)$ as above that additionally satisfy the conditions \textbf{(CI)}, \textbf{(DMP)}, \textbf{(MARG)} below. 
 We say that a locally commuting $2$-radial $\SLE_\k$ is \emph{interchangeable} if it further satisfies the condition \textbf{(INT)}.

\begin{figure}[ht]
    \begin{center}    \includegraphics[width=0.8\textwidth]{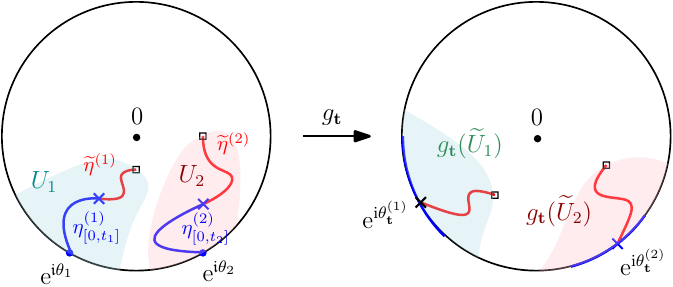}
    \end{center}
    \caption{\textbf{(CI)} and \textbf{(DMP)} imply that a locally commuting 2-radial SLE satisfies $(\widetilde \eta^{(1)}, \widetilde \eta^{(2)})\sim \m P_{(\widetilde  D;\, \eta^{(1)}_{t_1}, \eta^{(2)}_{t_2})}^{(\widetilde U_1, \widetilde U_2)} \sim g_\mbt^* \, \m P_{(\m D; \ee^{\ii \t^{(1)}_\mbt},\ee^{\ii \t^{(2)}_\mbt} )}^{(g_\mbt (\widetilde U_1), g_\mbt (\widetilde U_2))}.$}
    \label{fig::radial21}
\end{figure}
 
\begin{enumerate}
    \item[\textbf{(CI)}] \textbf{Conformal invariance:} If $\varphi : D \to D'$ is a conformal map fixing $0$, then the pullback measure
    $$\varphi^* \m P_{(D';\,\varphi(x_1),\varphi(x_2))}^{(\varphi(U_1), \varphi(U_2))} = \m P_{(D;\,x_1,x_2)}^{(U_1, U_2)}. $$
    Therefore, it suffices to describe the measure when $(D;\,x_1, x_2) = (\m D;\, \ee^{\ii \t_1}, \ee^{\ii \t_2})$.
    \label{as:CI}
    \item[\textbf{(DMP)}] \textbf{Domain Markov property:} Let $(\eta^{(1)}, \eta^{(2)})\sim \m P_{(\m D;\, \ee^{\ii \t_1}, \ee^{\ii \t_2})}^{(U_1, U_2)}$ and we parametrize $\eta^{(1)}$ and $\eta^{(2)}$ by their own capacity in $\m D$. Let $\mbt = (t_1,t_2)$ such that $t_j$ is a stopping time for $\eta^{(j)}$ and $\eta^{(j)}_{[0,t_j]}$ is contained in the interior of $U_j$. Let 
    $$\widetilde {U_j} = U_j \setminus \eta^{(j)}_{[0,t_j]}, \quad \widetilde \eta^{(j)} = \eta^{(j)} \setminus \eta^{(j)}_{[0,t_j]},\quad j=1,2; \quad \widetilde D = \m D \setminus (\eta^{(1)}_{[0,t_1]} \cup \eta^{(2)}_{[0,t_2]}).$$ 
    Then conditionally on $\eta^{(1)}_{[0,t_1]} \cup \eta^{(2)}_{[0,t_2]}$, we have 
    $$(\widetilde \eta^{(1)}, \widetilde \eta^{(2)})\sim \m P_{(\widetilde  D;\, \eta^{(1)}_{t_1}, \eta^{(2)}_{t_2})}^{(\widetilde U_1, \widetilde U_2)}.$$ 
    See Figure~\ref{fig::radial21}. 
    \label{as:DMP}

    \item[\textbf{(MARG)}]     \label{as:MARG}
 \textbf{Marginal laws:} Let $(\eta^{(1)}, \eta^{(2)})\sim \m P_{(\m D;\, \ee^{\ii \t_1}, \ee^{\ii \t_2})}^{(U_1, U_2)}$. We assume that there exist $C^2$ functions $b_j : S^1 \times S^1 \setminus \D \to \m R$ where $S^1 = \m R/2\pi \m Z$, and $\D$ is the diagonal $\{(\t, \t) \, |\, \t \in S^1\}$ such that the capacity parametrized Loewner driving function $t \mapsto \xi^{(1)}_{t}$ of $\eta^{(1)}$ satisfies
 \begin{equation}\label{eq:marg_1}
        \begin{cases}
            \xi^{(1)}_{0}  = \t_1, \quad    V^{(2)}_{0} = \t_2,  \\
            \dd \xi^{(1)}_{t}  = \sqrt \k \,\dd B_t^{(1)} + b_1 \left(\xi^{(1)}_{t}, V^{(2)}_{t}\right) \dd t, \\
            \dd V^{(2)}_{t} = \cot \left((V^{(2)}_{t} - \xi^{(1)}_{t})/2\right) \dd t,
        \end{cases}
    \end{equation}
    where $B^{(1)}$ is one-dimensional standard Brownian motion. 
    In other words, the radial Loewner chain $g^{(1)}_t$ 
    associated with $\eta^{(1)}$ maps the tip $\eta^{(1)}_t$ to $\exp (\ii \xi^{(1)}_t)$ and $\ee^{\ii\theta_2}$ to $\exp (\ii V^{(2)}_t)$ by \eqref{eq::loewner_boundary}.

    Similarly, the capacity parametrized Loewner driving function $t \mapsto \t^{(2)}_{t}$ of $\eta^{(2)}$ satisfies
    \begin{equation}\label{eq:marg_2}
        \begin{cases}
            V^{(1)}_0 = \t_1, \quad    \xi^{(2)}_0 = \t_2,  \\
            \dd \xi^{(2)}_t  =  \sqrt \k \,\dd B_t^{(2)} + b_2 (V^{(1)}_t, \xi^{(2)}_t) \dd t, \\
            \dd V^{(1)}_t  = \cot \left((V^{(1)}_t - \xi^{(2)}_t)/2\right) \dd t.
        \end{cases}
    \end{equation}
     In other words, the radial Loewner chain $g^{(2)}_t$ 
    associated with $\eta^{(2)}$ maps the tip $\eta^{(2)}_t$ to $\exp (\ii \xi^{(2)}_t)$ and $\ee^{\ii\theta_1}$ to $\exp (\ii V^{(1)}_t)$.
    \item[\textbf{(INT)}] \textbf{Interchangeability condition:}
The two curves are unordered. In other words, let $\tau : (\eta^{(1)}, \eta^{(2)}) \mapsto (\eta^{(2)}, \eta^{(1)})$,
    $$\m P_{(D;\,x_1,x_2)}^{(U_1, U_2)} \sim \tau^* \m P_{(D;\,x_2,x_1)}^{(U_2, U_1)}.$$
    \end{enumerate}

\begin{remark}
    Despite the heavy notation, the only purpose of introducing the neighborhoods $U_1, U_2$ is to give a precise meaning of ``local law'' of the $2$-radial SLE near their starting points. 
In particular, if we have a random pair of curves in $\mf X (\m D; x_1, x_2; 0)$, or a random simple curve in $\mf X (\m D; x_1, x_2)$, we obtain a compatible family of probability measures $\m P_{(\m D; x_1, x_2)}^{(U_1, U_2)}$ by restricting the random curve (or the pair of random curves) to all possible pairs of neighborhoods $(U_1, U_2)$.
Note that a priori, a compatible family of local laws does not necessarily imply they can be coupled as the restriction of a random curve or pair of curves in all $(U_1, U_2)$, but this will be a consequence of our classification Theorem~\ref{thm:main}.

To simplify notation, we also denote $\m P_{(\m D; \,\ee^{\ii\theta_1}, \ee^{\ii\theta_2})}^{(U_1, U_2)}$ by $\m P_{(\theta_1, \theta_2)}$ as in Section~\ref{subsec::intro_axioms}. 
\end{remark}

These axioms allow us to characterize the local law of the $2$-radial SLE by considering the infinitesimal generators of the two-time driving function which is the subject of the next sections.

\subsection{Commutation relations}
\label{subsec::commutation_relations}

In this section, we assume that $\k \in (0,\infty)$ and do not assume \textbf{(INT)}. We consider a \emph{locally commuting $2$-radial} $\SLE_\k$, which is a compatible family of measures $\m P_{(\t_1, \t_2)}$ satisfying the conditions \textbf{(CI), (DMP)}, and \textbf{(MARG)}.
Let $b_1, b_2 : S^1 \times S^1 \setminus \D \to \m R$ be $C^2$ functions as in the condition  \textbf{(MARG)}.

We now derive the infinitesimal form of the radial commutation relation. 
Let
\begin{align}\label{eq:gen}
\begin{split}
    \mc L_1 & = \frac{\k}{2} \partial_{11} + b_1 (\t_1, \t_2) \, \partial_1 + \cot \left( \frac{\t_2 - \t_1 }{2}\right) \partial_2 \\
     \mc L_2 & = \frac{\k}{2} \partial_{22} + b_2 (\t_1, \t_2) \, \partial_2 + \cot \left( \frac{\t_1 - \t_2 }{2}\right) \partial_1
\end{split}
\end{align}
be the diffusion generators associated with \eqref{eq:marg_1} and \eqref{eq:marg_2}. 

\begin{prop}\label{prop:radial_comm} 
The diffusion generators~\eqref{eq:gen} of a locally commuting 2-radial $\SLE_\k$ satisfy the \emph{infinitesimal commutation relation}
\begin{equation}\label{eq:comm_cond}
    \comm{\mc L_1, \mc L_2} := \mc L_1 \mc L_2 - \mc L_2 \mc L_1= \frac{\mc L_2 - \mc L_1 }{\left(\sin \left((\t_2 - \t_1)/2\right)\right)^2}. 
\end{equation}
\end{prop}
The proof follows exactly the same steps as in \cite{Dub_comm} for chordal SLEs.  The radial commutation relation was stated briefly in \cite{Dub_comm} but with an opposite sign (we believe this is a typo in \cite{Dub_comm}). For the reader's convenience, we include its proof in Appendix~\ref{appendix::commutation_relation}.

We note that the condition \textbf{(CI)} implies
     \begin{equation}\label{eq:rot}
    b_j (\t_1 + a, \t_2 + a) = b_j(\t_1, \t_2), \qquad \forall j = 1,2 \text{ and } a\in \m R.
    \end{equation}

\begin{prop}[Radial BPZ equations]\label{prop:part_wo_int}
Let $\k \in (0,\infty)$.
Let $b_1, b_2 : S^1 \times S^1 \setminus \D \to \m R$ be $C^2$ functions as in the condition  \textbf{(MARG)}.  Then~\eqref{eq:comm_cond} and~\eqref{eq:rot} imply that there exist $\mc Z : \{(\t_1, \t_2) \in \m R^2 \,|\, \t_1 < \t_2 < \t_1 +2 \pi\} \to \m R_{>0}$, called a \emph{partition function}, and a constant $\aleph\in\mathbb{R}$ such that 
\[b_j = \k \partial_j \log \mc Z, \quad j = 1,2, \]
and
\begin{align}
\label{eqn::BPZ1_unified}
    \frac{\kappa}{2}\frac{\partial_{11}\mc Z}{\mc Z}+\cot(\theta_{21}/2)\frac{\partial_2 \mc Z}{\mc Z}-\frac{h}{2\left(\sin(\theta_{21}/2)\right)^2}=&\aleph,\\
    \frac{\kappa}{2}\frac{\partial_{22}\mc Z}{\mc Z}-\cot(\theta_{21}/2)\frac{\partial_1 \mc Z}{\mc Z}-\frac{h}{2\left(\sin(\theta_{21}/2)\right)^2}=&\aleph,
\label{eqn::BPZ2_unified}
\end{align}
where $h=(6-\kappa)/(2\kappa)$ and $\theta_{21}=\theta_2-\theta_1=-\theta_{12}$. 
\end{prop}

We note that  $\mc Z$ does not always descend to a function on $(\m R / 2\pi \m Z)^2$. The radial BPZ equations (or BPZ--Cardy equations) were also derived in the context of conformal field theory in~\cite{KangMakarov2012,KangMakarovGFFCFT}.

\begin{proof}
    We use the expression~\eqref{eq:gen} and obtain
    \begin{align*}
   \comm{\mc L_1, \mc L_2} & = \Bigg[ - \frac{\k}{2 \left(\sin (\t_{12}/2)\right)^2} \Bigg] \partial_{11} +  \Bigg[ \k \partial_1 b_2 - \k \partial_2 b_1 \Bigg] \partial_{12} + \Bigg[ \frac{\k}{2 \left(\sin (\t_{12}/2)\right)^2} \Bigg] \partial_{22} \\
    + \Bigg[ \frac{\k - 2}{4} &\frac{\cot (\t_{12}/2)}{\left(\sin(\t_{12}/2)\right)^2} - \frac{b_1}{2 \left(\sin (\t_{12}/2)\right)^2}  - \left(\frac{\k}{2} \partial_{22}b_1 +b_2 \partial_2 b_1  + \cot (\t_{12}/2) \partial_1 b_1 \right)\Bigg] \partial_1 \\
     +\Bigg[ \frac{\k - 2}{4} &\frac{ \cot (\t_{12}/2)}{\left(\sin(\t_{12}/2)\right)^2} + \frac{b_2}{2 \left(\sin (\t_{12}/2)\right)^2}  + \left(\frac{\k}{2} \partial_{11}b_2 +b_1 \partial_1 b_2  - \cot (\t_{12}/2) \partial_2 b_2 \right)\Bigg] \partial_2,
\end{align*}
and
\begin{align*}
    \frac{\mc L_2 - \mc L_1}{\left(\sin (\t_{12}/2)\right)^2} = \frac{1}{\left(\sin (\t_{12}/2)\right)^2}\Bigg[\frac{\k}{2} (\partial_{22} - \partial_{11}) + ( \cot (\t_{12}/2) -b_1 ) \partial_1 + ( b_2 +  \cot (\t_{12}/2) ) \partial_2 \Bigg].
\end{align*}
Comparing the coefficients, we conclude from~\eqref{eq:comm_cond} that
\begin{align}
     & \k\partial_1 b_2  = \k\partial_2 b_1, \label{eq:b_comm}\\
      &\frac{\k}{2} \partial_{22}b_1 +b_2 \partial_2 b_1  + \cot (\t_{12}/2) \partial_1 b_1  - \frac{b_1}{2 \left(\sin(\t_{12}/2)\right)^2}   + \left(\frac{6 - \k}{4}\right) \frac{\cot (\t_{12}/2)}{\left(\sin(\t_{12}/2)\right)^2}  = 0,\label{eqn::comm_b1}\\
 &\frac{\k}{2} \partial_{11}b_2 +b_1 \partial_1 b_2  + \cot (\t_{21}/2) \partial_2 b_2  - \frac{b_2}{2 \left(\sin(\t_{21}/2)\right)^2}  + \left(\frac{6- \k}{4}\right) \frac{\cot (\t_{21}/2)}{\left(\sin(\t_{21}/2)\right)^2}  = 0.\label{eqn::comm_b2}
 \end{align}

Eq.~\eqref{eq:b_comm} shows that there exists a function $\mc Z : \{(\t_1, \t_2) \in \m R^2 \,|\, \t_1 < \t_2 < \t_1 +2 \pi\} \to \m R_{>0}$ such that
$$ b_1 = \k \, \partial_1 \log \mc Z = \k \frac{\partial_1 \mc Z}{\mc Z}, \quad b_2 = \k\, \partial_2 \log \mc Z= 
\k \frac{\partial_2 \mc Z}{\mc Z}.$$
Plugging it into~\eqref{eqn::comm_b1} and~\eqref{eqn::comm_b2}, we have
\begin{align}
\kappa\partial_1\left(\frac{\kappa}{2}\frac{\partial_{22}\mc Z}{\mc Z}-\cot(\theta_{21}/2)\frac{\partial_1 \mc Z}{\mc Z}-\frac{h}{2\left(\sin(\theta_{21}/2)\right)^2}\right) & = 0,
\label{eqn::partial1BPZ2}\\
\label{eqn::partial2BPZ1}
\kappa\partial_2\left(\frac{\kappa}{2}\frac{\partial_{11}\mc Z}{\mc Z}+\cot(\theta_{21}/2)\frac{\partial_2 \mc Z}{\mc Z}-\frac{h}{2\left(\sin(\theta_{21}/2)\right)^2}\right)& = 0,
\end{align}
where $h=(6-\kappa)/(2\kappa)$. 

Equations~\eqref{eqn::partial1BPZ2} and~\eqref{eqn::partial2BPZ1} imply that there exist functions $\aleph_1$ and $\aleph_2$:
\begin{align*}
\frac{\kappa}{2}\frac{\partial_{22}\mc Z}{\mc Z}-\cot(\theta_{21}/2)\frac{\partial_1 \mc Z}{\mc Z}-\frac{h}{2\left(\sin(\theta_{21}/2)\right)^2}=&\aleph_2(\theta_2),
\\
    \frac{\kappa}{2}\frac{\partial_{11}\mc Z}{\mc Z}+\cot(\theta_{21}/2)\frac{\partial_2 \mc Z}{\mc Z}-\frac{h}{2\left(\sin(\theta_{21}/2)\right)^2}=&\aleph_1(\theta_1).
\end{align*}
Using the identity $$\frac{\partial_{22} \mc Z}{\mc Z}  = \partial_2 \left(\frac{\partial_2 \mc Z}{\mc Z}\right) + \left(\frac{\partial_2 \mc Z}{\mc Z}\right)^2 $$
and Eq.~\eqref{eq:rot}, we deduce that $\aleph_1$ and $\aleph_2$ are constants and
\begin{align}
\label{eqn::BPZ1}
    \frac{\kappa}{2}\frac{\partial_{11}\mc Z}{\mc Z}+\cot(\theta_{21}/2)\frac{\partial_2 \mc Z}{\mc Z}-\frac{h}{2\left(\sin(\theta_{21}/2)\right)^2}=&\aleph_1,\\
    \frac{\kappa}{2}\frac{\partial_{22}\mc Z}{\mc Z}-\cot(\theta_{21}/2)\frac{\partial_1 \mc Z}{\mc Z}-\frac{h}{2\left(\sin(\theta_{21}/2)\right)^2}=&\aleph_2.
\label{eqn::BPZ2}
\end{align}

It remains to show $\aleph_1=\aleph_2$. Taking the derivative of \[a \mapsto b_j (\t_1 + a, \t_2 +a) =\kappa (\partial_j \mc Z/\mc Z) (\t_1 + a, \t_2 +a)\] and evaluating it at $a = 0$, we get from \textbf{(CI)} and \eqref{eq:rot} that
\begin{align}\label{eq:rot_invariant}
    0 = \partial_1 \left(\frac{\partial_j \mc Z}{\mc Z}\right) + \partial_2 \left(\frac{\partial_j \mc Z}{\mc Z}\right) 
    , \qquad \forall j = 1,2. 
\end{align}
From this, we have 
\begin{align*}
    \frac{\partial_{11} \mc Z}{\mc Z}  - \frac{\partial_{22} \mc Z}{\mc Z} &  = \left(\frac{\partial_1 \mc Z}{\mc Z}\right)^2 -\left(\frac{\partial_2 \mc Z}{\mc Z}\right)^2\\
    \partial_1 \left(\frac{\partial_2 \mc Z}{\mc Z}\right) & = \partial_2 \left(\frac{\partial_1 \mc Z}{\mc Z}\right)\\
      \partial_1 \left(\frac{\partial_1 \mc Z}{\mc Z}\right) & = \partial_2 \left(\frac{\partial_2 \mc Z}{\mc Z}\right).
\end{align*} 
If we take the difference \eqref{eqn::BPZ1} - \eqref{eqn::BPZ2}, we get
\begin{equation}\label{eq:F_12}
    \frac{\k}{2} \left(\left(\frac{\partial_1 \mc Z}{\mc Z}\right)^2 - \left(\frac{\partial_2 \mc Z}{\mc Z}\right)^2 \right) +\cot (
    \t_{21}/2)  \left(\frac{\partial_2 \mc Z}{\mc Z} + \frac{\partial_1 \mc Z}{\mc Z}\right) = \aleph_1 - \aleph_2.
\end{equation}
Note that \eqref{eq:rot_invariant} also implies
$$ 0  = \partial_1 \left(\frac{\partial_1 \mc Z}{\mc Z}\right) + \partial_2 \left(\frac{\partial_1 \mc Z}{\mc Z}\right) =  \partial_1 \left(\frac{\partial_1 \mc Z}{\mc Z}\right) + \partial_1 \left(\frac{\partial_2 \mc Z}{\mc Z}\right) = \partial_1 \left(\frac{\partial_1 \mc Z}{\mc Z} + \frac{\partial_2 \mc Z}{\mc Z}\right) $$
and
$$ 0  = \partial_2 \left(\frac{\partial_1 \mc Z}{\mc Z} + \frac{\partial_2 \mc Z}{\mc Z}\right). $$
Hence there is $\mu \in \m R$ such that 
\begin{equation}\label{def:mu}
    \frac{2\mu}{\k} \equiv \frac{\partial_1 \mc Z}{\mc Z} + \frac{\partial_2 \mc Z}{\mc Z}.
\end{equation}
Plugging into \eqref{eq:F_12} we get
\begin{equation}\label{eq:two_Cases}
\left(\frac{\partial_1 \mc Z}{\mc Z} - \frac{\partial_2 \mc Z}{\mc Z}\right) \mu + \cot (\t_{21}/2) \frac{2\mu}{\k} = \aleph_1-\aleph_2.
\end{equation}
Combining~\eqref{def:mu} and~\eqref{eq:two_Cases}, we investigate two cases. 

\noindent \textbf{Case 1:} If $\mu=0$, then $\aleph_1=\aleph_2$ as desired.

\noindent \textbf{Case 2:} If $\mu \neq 0$, then
\begin{equation}\label{eq:case_1}
     \begin{cases}
  \dfrac{\partial_1 \mc Z}{\mc Z} =    -\frac{1}{\kappa} \cot(\t_{21}/2) +\frac{\mu}{\k}+\frac{\aleph_1-\aleph_2}{2\mu}, \\
   \dfrac{\partial_2 \mc Z}{\mc Z} =   \frac{1}{\kappa} \cot(\t_{21}/2) +\frac{\mu}{\k}-\frac{\aleph_1-\aleph_2}{2\mu}.
\end{cases}
\end{equation}
We solve \eqref{eq:case_1} and obtain that, for some constant $C \in \m R$, 
\begin{equation}\label{eqn::sol_case_1}
    \mc Z (\t_1, \t_2) = C \left(\sin (\t_{21}/2)\right)^{2/\k} \exp \left(\frac{\mu_1}{\k}\theta_1+\frac{\mu_2}{\kappa}\theta_2\right),
\end{equation}
where
\[\mu_1=\mu+\frac{\kappa(\aleph_1-\aleph_2)}{2\mu},\quad \mu_2=\mu-\frac{\kappa(\aleph_1-\aleph_2)}{2\mu}.\]
Plugging this into~\eqref{eqn::BPZ1}-\eqref{eqn::BPZ2}, we obtain $\mu_1=\mu_2$ which implies $\aleph_1=\aleph_2$ as desired. Moreover, in this case we have 
\begin{equation*}
\aleph_1 = \aleph_2 = \frac{-3 + \mu^2}{2 \kappa}, \qquad \mc Z \propto \LG_\mu, \qquad \text{ when } \mu \neq 0.
\end{equation*}
This completes the proof.
\end{proof}

\subsection{Solutions to commutation relations with interchangeability}
\label{subsec::commutation_relation_INT}

We now classify all possible partition functions $\mc Z$ with the additional condition \textbf{(INT)}, which is equivalent to 
    \begin{equation}\label{eq:symm}
    b_1 (\t_1, \t_2) = b_2(\t_2, \t_1).
    \end{equation}

\begin{thm}\label{thm:classify_int}
Let $\kappa\in (0,8)$ and $\mc Z$ be a partition function from Proposition~\ref{prop:part_wo_int} and assume that \eqref{eq:symm} holds. Then, up to a multiplicative constant, $\mc Z$ is one of the following functions:
\begin{enumerate}
    \item $\mc Z =  \LG_\mu$ for some $\mu \in \m R$, where $\LG_\mu$ is defined in \eqref{eqn::2SLEspiral_pf}.
    \item $\mc Z = \mc Z_\a$ for some $\a < 1 -\k/8$, where $\mc Z_\a$ is defined in \eqref{eqn::CR_pf}.
\end{enumerate}
\end{thm}

Throughout this section, we assume that $\mc Z$ is a partition function from Proposition~\ref{prop:part_wo_int} and assume that \eqref{eq:symm} holds.

\begin{lem}\label{lem:F_12_equal}
There exists $\l > 0$ such that for all $\t_1 < \t_2 < \t_1 + 2\pi$,
\begin{equation}\label{eqn::cond_Z_1}
  \l \mc Z(\t_1, \t_2) = \mc Z(\t_2, \t_1 + 2\pi).
\end{equation}
\end{lem}
\begin{proof}
    Assumption \eqref{eq:symm} implies that 
$$\k \partial_1 \log\left(\frac{\mc Z (\t_2, \t_1 + 2\pi)}{\mc Z(\t_1, \t_2)}\right) = b_2 (\t_2, \t_1 + 2\pi) - b_1 (\t_1, \t_2) = b_2 (\t_2, \t_1) - b_1 (\t_1, \t_2) =0$$
and similarly,
$$\partial_2 \log\left(\frac{\mc Z (\t_2, \t_1 + 2\pi)}{\mc Z(\t_1, \t_2)}\right) =0.$$
Therefore, $\mc Z (\t_2, \t_1 + 2\pi)/\mc Z(\t_1, \t_2)$ is a constant $\l$ which also equals $\mc Z(\t_1+2\pi, \t_2+2\pi) / Z (\t_2, \t_1 + 2\pi)$. 
\end{proof}

\begin{proof}[Proof of Theorem~\ref{thm:classify_int}]
From the proof of Proposition~\ref{prop:part_wo_int}, combining~\eqref{def:mu} and~\eqref{eq:two_Cases}, we have two cases: we have either $\mu\neq 0$, then 
$$\mc Z \propto \LG_\mu$$
which satisfies the interchangeability condition \eqref{eq:symm},
or we have $\mu=0$, then 
\begin{equation}\label{eq:case_2}
    \partial_1 \mc Z= -\partial_2 \mc Z.
\end{equation}

Eq.~\eqref{eq:case_2} implies that $\partial_a \mc Z(\t_1 + a, \t_2 +a) = 0$; in other words, $\mc Z(\theta_1, \theta_2)$ only depends on the difference $\theta=\theta_2-\theta_1$. Thus we write with a slight abuse of notation
\[\mc Z(\theta_1, \theta_2)= \mc Z(\theta),\quad\theta\in (0,2\pi).\]
We want to solve the equations~\eqref{eqn::BPZ1_unified} and \eqref{eqn::BPZ2_unified}, which are simplified to the same equation:
\begin{equation}\label{eqn::rot_BPZ}
    \frac{\kappa}{2}\frac{\mc Z''}{\mc Z}+\cot(\theta/2)\frac{\mc Z'}{\mc Z}-\frac{h}{2\left(\sin(\theta/2)\right)^2}=F. 
\end{equation} 
Noticing that 
$$\l ^2 = \frac{\mc Z (\t_1 + 2\pi, \t_2 + 2\pi)}{\mc Z (\t_1, \t_2)} = 1, $$
we obtain $\l = 1$ and
\begin{equation}\label{eqn::periodic}
    \mc Z(\theta)= \mc Z(2\pi-\theta).
\end{equation}

To solve \eqref{eqn::rot_BPZ} under the assumption \eqref{eqn::periodic}, we may write 
\begin{equation}\label{eqn::thetau}
   \theta\in(0,2\pi),\quad u=(\sin(\theta/4))^2\in (0,1),\quad \mc Z(\theta)=C (\sin(\theta/2))^{-2h}\phi(u),
\end{equation}
satisfying for $u \in (0,1)$,
\begin{equation}\label{eqn::BPZ_Euler_initial}
\begin{cases}
   u(1-u)\phi''+\frac{3\kappa-8}{2\kappa}(1-2u)\phi'+\frac{8}{\kappa}\left(\frac{(6-\kappa)(\kappa-2)}{8\kappa}-F\right)\phi=0, \\
\phi(1/2)=1, \quad \phi'(1/2)=0,
\end{cases}
\end{equation}
where $C$ is a positive constant.
Eq.~\eqref{eqn::BPZ_Euler_initial} has a unique solution $\phi$ in $C^2(0,1)$ due to Lemma~\ref{lem::Euler_initial}.
More precisely:
\begin{itemize}
    \item When 
$$\a :  = \frac{(6-\kappa)(\kappa-2)}{8\kappa}-F  <  1 -  \frac{\k}{8}, \quad \text{ i.e., } \quad F > - \frac{3}{2 \k},$$
the unique solution 
$\phi (u)$ is positive for all $u\in (0,1)$.
\item When $\a = 1 -  \k/8$, 
$$\phi (u) = (4 u (1-u))^{4/\k - 1/2} = \left(\sin(\t/2)\right)^{8/\k - 1}$$ by \eqref{eqn::Euler_sol_critical}. Hence,
$$\mc Z (\t) = C \left(\sin(\t/2)\right)^{2/\k} \propto \LG_0.$$
\item  When $\a > 1 -  \k/8$, Lemma~\ref{lem::Euler_initial} shows that the unique solution is not always positive. Therefore it does not give any partition function in Proposition~\ref{prop:part_wo_int}. See Figure~\ref{fig::kappa4} for the case when $\kappa=4$. 
\end{itemize}
This completes the proof.
\end{proof}

\begin{remark}\label{rem::CR_kappa4}
    When $\kappa=4$, the partition function $\LZ_{\alpha}$ in~\eqref{eqn::CR_pf} has an explicit expression. We write $\theta=\theta_2-\theta_1$. 
    \begin{itemize}
        \item If $\alpha\in [0,1/2)$, we have
        \begin{align*}
            \LZ_{\alpha}(\theta_1, \theta_2)=\left(\sin(\theta/2)\right)^{-1/2}\cos\left(\sqrt{\frac{\alpha}{2}}(\theta-\pi)\right). 
        \end{align*}
        \item If $\alpha<0$, we have 
        \begin{align*}
            \LZ_{\alpha}(\theta_1, \theta_2)=\left(\sin(\theta/2)\right)^{-1/2}\cosh\left(\sqrt{\frac{|\alpha|}{2}}(\theta-\pi)\right). 
        \end{align*}
    \end{itemize}
    See Figure~\ref{fig::kappa4}. 
\end{remark}

\begin{figure}[ht!]
\begin{center}
    \includegraphics[width=0.8\textwidth]{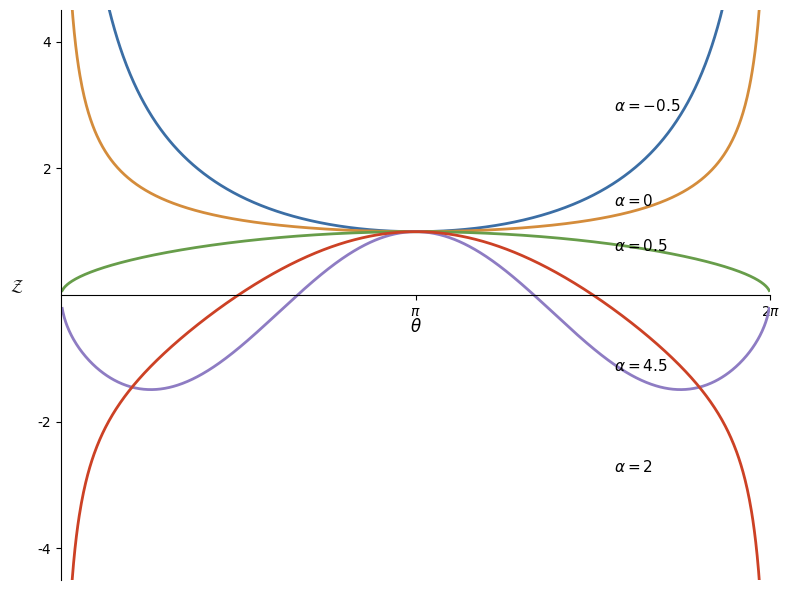}
\end{center}
    \caption{Plot of $\theta\mapsto \mc Z_{\alpha}(\theta)$ with $\kappa=4$ and $\theta=\theta_2-\theta_1$ for different $\alpha$'s. When $\alpha>1/2$, it is not always positive.  This is consistent with the general conclusion in Lemma~\ref{lem::Euler_initial}.}
    \label{fig::kappa4}
\end{figure}

\begin{remark}\label{rem::onearm}
When $\kappa\in (0,8)$ and $\alpha=\alpha_1(\kappa)$ which is the one-arm exponent for the conformal loop ensemble~\cite{SSW2009}:
\[\alpha_1(\kappa)=\frac{(3\kappa-8)(8-\kappa)}{32\kappa},\]
the partition function $\LZ_{\alpha_1(\kappa)}$ has an explicit expression: we write $\theta=\theta_2-\theta_1$,
\begin{equation}\label{eqn::onearm}
    \LZ_{\alpha_1(\kappa)}(\theta_1, \theta_2)=2^{4/\kappa - 3/2}(\sin(\theta/2))^{1-6/\kappa}\left((\sin(\theta/4))^{8/\kappa-1}+(\cos(\theta/4))^{8/\kappa-1}\right).
\end{equation}
For $\kappa\in (4,8)$, this is the conjectured partition function for the scaling limit of interfaces in critical random-cluster models with Dobrushin boundary conditions conditional on the one-arm event; see~\cite{FengWu2024}. This conjecture holds for percolation (with $\kappa=6$) and the FK-Ising model (with $\kappa=16/3$). 
\end{remark}

\begin{remark}\label{rem::hypoelliptic}
    Let us give a comment on the $C^2$ requirement on $b_1, b_2$ in the assumption \textbf{(MARG)}. We impose such a $C^2$ requirement as part of the axioms in the beginning. But in fact, this is not necessary. Suppose we relax the requirement and only assume $b_1, b_2$ in~\textbf{(MARG)} are locally bounded and measurable, then we are still able to derive~\eqref{eqn::BPZ1} and~\eqref{eqn::BPZ2} as weak solutions. The operators on the left-hand side of the radial BPZ equations~\eqref{eqn::BPZ1} and~\eqref{eqn::BPZ2} are hypoelliptic, due to a general characterization by H\"{o}rmander~\cite{Hormander1967}; see also~\cite[Lemma~5]{Dub_SLEVir1} and~\cite{KV26}. Therefore, the weak solutions are strong solutions which are in fact $C^\infty$, and consequently $b_1, b_2$ are $C^\infty$. Note that such an analysis does not work for $\kappa=0$ since the corresponding BPZ equation~\eqref{eqn::int_0} is nonlinear. 
\end{remark}

\section{Identification of locally commuting 2-radial SLEs}
\label{sec::radialSLE}

The goal of this section is to identify all locally commuting 2-radial SLEs, namely, those whose partition functions are classified in Theorem~\ref{thm:classify_int}:
We discuss two-sided radial SLE with spiral in Section~\ref{subsec::2SLEspiral} and chordal SLE weighted by the conformal radius in Section~\ref{subsec::chordalSLE_CR}. 

\subsection{Radial SLE}
\label{subsec::radialLoewner}

For $\theta\in [0,2\pi)$, suppose $\eta: [0,T]\to \overline{\m D}$ is a continuous non-self-crossing curve such that $\eta_0=\ee^{\ii \theta}$ and $\eta_{(0,T)}\subset\m D\setminus\{0\}$. We parameterize the curve by the capacity and denote by $g_t$ the corresponding radial Loewner chain as in~\eqref{eqn::radialLoewner}. 
Denote by $\phi_t$ the covering conformal map of $g_t$, i.e. $g_t(\exp(\ii w))=\exp(\ii \phi_t(w))$ with $\phi_0(w)=w$ for $w\in\mathbb{H}$. Then the radial Loewner equation~\eqref{eqn::radialLoewner} is equivalent to 
\[\partial_t\phi_t(w)=\cot\left((\phi_t(w)-\xi_t)/2\right), \quad \phi_0(w)=w.\]

Radial $\SLE_{\kappa}$ is the radial Loewner chain with $\xi_t=\sqrt{\kappa}B_t$ where $B$ is a one-dimensional Brownian motion. We also call it radial $\SLE_{\kappa}$ in $(\m D; \ee^{\ii\theta}; 0)$. 
In the following lemma, we will describe the boundary perturbation property of radial SLE. We fix the parameters:
\begin{equation}\label{eqn::parameters}
    \kappa>0, \quad h=\frac{6-\kappa}{2\kappa},\quad \widetilde{h}=\frac{(6-\kappa)(\kappa-2)}{8\kappa},\quad c=\frac{(6-\kappa)(3\kappa-8)}{2\kappa}.
\end{equation} 
\begin{lem}[{See \cite[Prop.~5]{JL18} or \cite[Prop.~2.2]{HealeyLawlerNSidedSLE}}]
\label{lem::radialSLE_boundary_perturb}
Fix $\kappa\in (0,8)$ and $\theta\in [0,2\pi)$. 
Suppose $K$ is a compact subset of $\overline{\m D}$ such that $\m D\setminus K$ is simply connected and contains the origin and that $K$ has a positive distance from $\ee^{\ii\theta}$. Suppose $\eta$ is radial $\SLE_{\kappa}$ in $(\m D; \ee^{\ii\theta}; 0)$ and define $\tau=\inf\{t: \eta_t\in K\}$.
For $t<\tau$, denote by $g_{t, K}$ the unique conformal map $\m D\setminus g_t(K)\to \m D$ such that $g_{t, K}(0)=0$ and $g_{t, K}'(0)>0$. We denote by $\phi_{t, K}$ the covering map of $g_{t, K}$. Then the following process is a local martingale: 
\begin{align*}
M_t=1{\{t<\tau\}}\phi_{t, K}'(\xi_t)^hg_{t, K}'(0)^{\widetilde{h}}\exp\left(\frac{c}{2}m_t\right), 
\end{align*}
where $m_t$ is defined through
\[\ud m_t=-\frac{1}{3}\mc S \phi_{t, K}(\xi_t)\ud t+\frac{1}{6}\left(1-\phi_{t, K}'(\xi_t)^2\right)\ud t,\]
and $\mc S\phi=\frac{\phi'''}{\phi'}-\frac{3}{2}\left(\frac{\phi''}{\phi'}\right)^2$ denotes the Schwarzian derivative of $\phi$.

Moreover, when $\kappa\le 4$, the process $M_t$ is a uniformly integrable martingale. The law of radial $\SLE_{\kappa}$ in $(\m D\setminus K; \ee^{\ii\theta}; 0)$ is the same as that of radial $\SLE_{\kappa}$ in $(\m D; \ee^{\ii\theta}; 0)$ weighted by $M_t$.
\end{lem}

\begin{remark}\label{rem:blm}
 It is explained in the proof of \cite[Prop.~5]{JL18} that the term $m_t=m_{\m D}(\eta_{[0,t]}, K)$ is the same as the Brownian loop measure of loops that intersect both $\eta_{[0,t]}$ and $K$ when $\eta_{[0,t]}\cap K=\emptyset$. 
\end{remark}

Fix $\theta_1, \theta_2$ such that $\theta_1<\theta_2<\theta_1+2\pi$. Let $\k \in [0,\infty)$, $\rho\in \mathbb{R}$ and $\mu\in\mathbb{R}$. The radial $\SLE_{\kappa}^{\mu}(\rho)$ in $\m D$ starting from $\ee^{\ii\theta_1}$ with force point $\ee^{\ii\theta_2}$ and spiraling rate $\mu$ is the radial Loewner chain with driving function $\xi_t$ that satisfies the following SDE: 
\begin{equation}\label{eq:def_kappa_mu_rho}
    \begin{cases}
    \xi_0=\theta_1, V_0=\theta_2,\\
    \ud \xi_t=\sqrt{\kappa}\ud B_t+\frac{\rho}{2}\cot\left((\xi_t-V_t)/2\right)\ud t+\mu\ud t, \\
    \ud V_t=\cot\left((V_t-\xi_t)/2\right)\ud t.
    \end{cases}
\end{equation}
The solution to SDE~\eqref{eq:def_kappa_mu_rho} exists for all time when $\kappa\in (0,8)$ and $\rho>-2$ and it is generated by a continuous curve from $\ee^{\ii\theta_1}$ to the origin. 
\begin{lem}\label{lem::radialSLE_continuity}
For $\kappa\in (0,8)$, $\rho>-2$, and $\mu\in\mathbb{R}$, radial $\SLE_{\kappa}^{\mu}(\rho)$ in $\m D$ is almost surely generated by a continuous curve $\eta$ and $\lim_{t\to\infty}\eta_t=0$. 
\end{lem}

\begin{proof}
This is proved in~\cite[Prop.~3.30]{IG4} and~\cite[Sect.~4]{IG4}. See also~\cite{Lawler2sidedSLE}. 
\end{proof}

\begin{remark} \label{rem::2SLE_spiral_marg}
The expression of $\LG_\mu$ in~\eqref{eqn::2SLEspiral_pf} is such that the Loewner driving function of radial $\SLE_\k^\mu(2)$  can be rewritten as
\begin{equation*}
    \begin{cases}
    \xi_0=\theta_1, V_0=\theta_2,\\
    \ud \xi_t=\sqrt{\kappa}\ud B_t+ \k\,\partial_1 \log \LG_\mu (\xi_t, V_t)\ud t, \\
    \ud V_t=\cot\left((V_t-\xi_t)/2\right)\ud t.
    \end{cases}
\end{equation*}
\end{remark}

We define the radial SLE$_{\kappa}^{\mu}(\rho)$ process in a general domain $\Omega$ with $x_1, x_2\in\partial\Omega$ and $z\in\Omega$ as the pushforward measure of the SLE$_{\kappa}^{\mu}(\rho)$ in $(\m D; \ee^{\ii\theta_1}, \ee^{\ii\theta_2}; 0)$ by the map $\varphi^{-1}$, where $\varphi: \Omega\to \m D$ is any conformal map such that $\varphi(z)=0$ and $\varphi(x_j)=\ee^{\ii\theta_j}$ for $j\in\{1,2\}$.

\subsection{Two-sided radial SLE with spiral}
\label{subsec::2SLEspiral}

In this section,
we define two-sided radial $\SLE_\k$ with spiral when $\k \in (0,8)$
by reweighting two independent radial SLE$_\k$. The two-sided radial SLE analyzed in~\cite{Lawler2sidedSLE,field2016twosided,Lawler_Zhou_natural,HealeyLawlerNSidedSLE} is a special case where the spiraling rate $\mu = 0$. 

We introduce the following notation. 
Fix $\theta_1<\theta_2<\theta_1+2\pi$. We will describe the growth of a pair of continuous non-self-crossing curves $(\eta^{(1)}, \eta^{(2)})$ in $\m D$ such that $\eta^{(1)}_0=\ee^{\ii\theta_1}$ and $\eta^{(2)}_0=\ee^{\ii\theta_2}$. For $\mbt=(t_1, t_2)\in \mathbb{R}_{+}^2$, suppose $\eta^{(1)}_{[0,t_1]}$ and $\eta^{(2)}_{[0,t_2]}$ are disjoint. We consider the following mapping-out functions: 
\begin{itemize}
    \item $g^{(j)}_{t_j}: \m D\setminus\eta^{(j)}_{[0,t_j]}\to \m D$ is conformal with $g^{(j)}_{t_j}(0)=0$ and $\left(g^{(j)}_{t_j}\right)'(0)=\exp(t_j)>0$, for $j=1,2$. 
    \item $g_{\mbt}: \m D\setminus \left(\eta^{(1)}_{[0,t_1]}\cup\eta^{(2)}_{[0,t_2]}\right)\to \m D$ is conformal with $g_{\mbt}(0)=0$ and $g_{\mbt}'(0)
    >0$. 
    \item $g_{\mbt, 1}: \m D\setminus g^{(1)}_{t_1}\left(\eta^{(2)}_{[0,t_2]}\right)\to \m D$ is conformal with $g_{\mbt, 1}(0)=0$ and $g_{\mbt, 1}'(0)>0$.
    \item $g_{\mbt, 2}: \m D\setminus g^{(2)}_{t_2}\left(\eta^{(1)}_{[0,t_1]}\right)\to \m D$ is conformal with $g_{\mbt, 2}(0)=0$ and $g_{\mbt, 2}'(0)>0$.
\end{itemize}
Using such notation, we have $g_{\mbt}=g_{\mbt, j}\circ g^{(j)}_{t_j}$ for $j=1,2$. Let $\phi^{(j)}_{t_j}, \phi_{\mbt}, \phi_{\mbt, j}$ be the covering maps of $g^{(j)}_{t_j}, g_{\mbt}, g_{\mbt, j}$ respectively, i.e., the continuous functions such that
$g_\cdot (\ee^{\ii \t}) = \ee^{\ii\phi_\cdot (\t)}$ and $\phi_0 (\t) = \t$.
Denote by $(\xi^{(j)}_{t_j}, t_j\ge 0)$ the driving function of $\eta^{(j)}$ as a radial Loewner chain for $j=1,2$. Let 
\[\theta^{(1)}_{\mbt}=\phi_{\mbt, 1}\left(\xi^{(1)}_{t_1}\right), \qquad\theta^{(2)}_{\mbt}=\phi_{\mbt, 2}\left(\xi^{(2)}_{t_2}\right).\]
The pair $\mbt \mapsto (\theta^{(1)}_{\mbt},\theta^{(2)}_{\mbt})$ may be viewed as the two-time driving function of the pair $(\eta^{(1)}, \eta^{(2)})$. See Figure~\ref{fig::gtcommutation}.

\begin{figure}[ht!]
    \begin{center}    \includegraphics[width=0.8\textwidth]{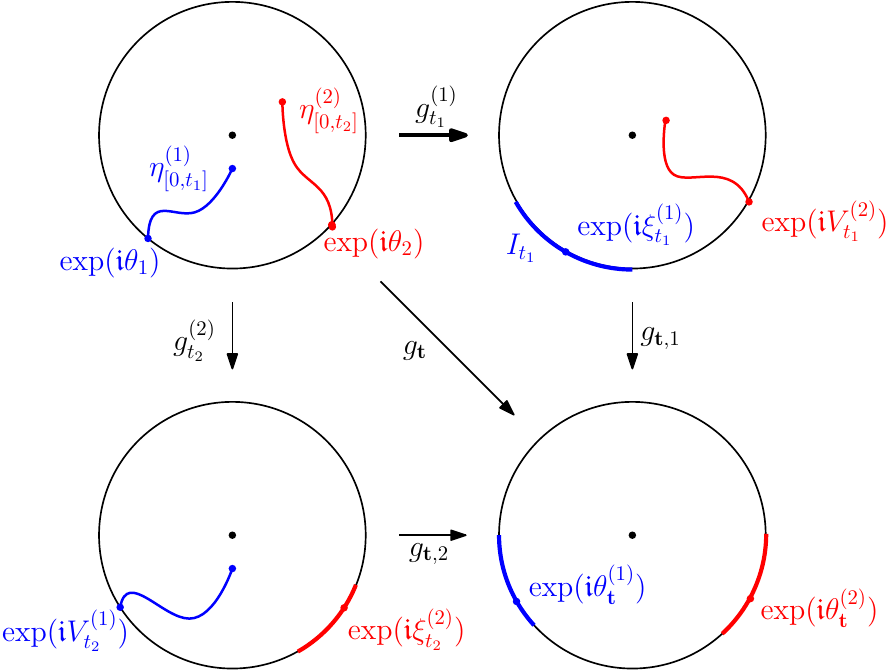}
    \end{center}
    \caption{We have $g_{\mbt}=g_{\mbt, 1}\circ g_{t_1}^{(1)}=g_{\mbt, 2}\circ g_{t_2}^{(2)}$.}
    \label{fig::gtcommutation}
\end{figure}

\begin{lem}\label{lem::2SLE_spiral_mart}
Fix $\kappa\in (0,8)$ and $\theta_1<\theta_2<\theta_1+2\pi$.
We fix the parameters $h, \widetilde{h}, c$ as in~\eqref{eqn::parameters}.
For $\mu\in\mathbb{R}$, we define $\LG_{\mu}$ as in~\eqref{eqn::2SLEspiral_pf}. 
Let $\PP$ denote the probability measure under which $(\eta^{(1)}, \eta^{(2)})$ are two independent radial $\SLE_{\kappa}$ in $\m D$ starting from $\ee^{\ii\theta_1}$ and $\ee^{\ii\theta_2}$ respectively. 
We define 
\begin{align}\label{eqn::2SLE_mart_new}
    M_{\mbt}(\LG_{\mu})=&1\left\{\eta^{(1)}_{[0,t_1]}\cap \eta^{(2)}_{[0,t_2]}=\emptyset\right\}g_{\mbt}'(0)^{\frac{3-\mu^2}{2\kappa}-\widetilde{h}}\times\prod_{j=1}^2 \phi_{\mbt, j}'\left(\xi_{t_j}^{(j)}\right)^hg_{\mbt, j}'(0)^{\widetilde{h}}\notag\\
    &\times \LG_{\mu}\left(\theta_{\mbt}^{(1)}, \theta_{\mbt}^{(2)}\right) \exp\left(\frac{c}{2}m_{\mbt}\right),
\end{align}
where  $m_{\mbt}$ is defined through
\begin{equation}\label{eq:blm_2_curves}
    \ud m_{\mbt}= \sum_{j=1}^2\left(-\frac{1}{3}\mc S\phi_{\mbt, j}\left(\xi^{(j)}_{t_j}\right)+\frac{1}{6}\left(1-\phi_{\mbt, j}'\left(\xi_{t_j}^{(j)}\right)^2\right)\right)\ud t_j. 
\end{equation}
Then $M_{\mbt}(\LG_{\mu})$ is a two-time-parameter local martingale with respect to $\PP$. 
\end{lem}

Lemma~\ref{lem::radialSLE_boundary_perturb} and Remark~\ref{rem:blm} show that  $m_{\mbt}= m_{\m D}\left(\eta_{[0,t_1]}^{(1)},\eta_{[0,t_2]}^{(2)}\right)$ is the Brownian loop measure of loops in $\m D$ intersecting both $\eta_{[0,t_1]}^{(1)}$ and $\eta_{[0,t_2]}^{(2)}$ when $\eta_{[0,t_1]}^{(1)}\cap \eta_{[0,t_2]}^{(2)}=\emptyset$. 

We note that $\LG_{\mu}$ satisfies the ``radial BPZ equations''
\begin{align}
\label{eqn::BPZ1_Gmu}
    \frac{\kappa}{2}\frac{\partial_{11}\LG_{\mu}}{\LG_{\mu}}+\cot(\theta_{21}/2)\frac{\partial_2\LG_{\mu}}{\LG_{\mu}}-\frac{h}{2\left(\sin(\theta_{21}/2)\right)^2}=&\frac{\mu^2-3}{2\kappa};\\
\frac{\kappa}{2}\frac{\partial_{22}\LG_{\mu}}{\LG_{\mu}}-\cot(\theta_{21}/2)\frac{\partial_1\LG_{\mu}}{\LG_{\mu}}-\frac{h}{2\left(\sin(\theta_{21}/2)\right)^2}=&\frac{\mu^2-3}{2\kappa}.
\label{eqn::BPZ2_Gmu}
\end{align}
 
Relations~\eqref{eq:blm_2_curves}, \eqref{eqn::BPZ1_Gmu} and~\eqref{eqn::BPZ2_Gmu} play an essential role in the proof of Lemma~\ref{lem::2SLE_spiral_mart}. 

\begin{proof}[Proof of Lemma~\ref{lem::2SLE_spiral_mart}]

Let us first compute the variations of terms appearing in \eqref{eqn::2SLE_mart_new}.
From standard calculations (see e.g.~\cite[Lem.~3.2]{HealeyLawlerNSidedSLE}), we have the following variational formula of the capacity parameterizations:
\begin{align}\label{eq:cap_change}
\begin{split}
\frac{\ud g'_{\mbt, 1}(0)}{g'_{\mbt, 1}(0)}=&\left(\phi'_{\mbt, 1}\left(\xi^{(1)}_{t_1}\right)^2-1\right)\ud t_1+\phi'_{\mbt, 2}\left(\xi^{(2)}_{t_2}\right)^2\ud t_2;\\
\frac{\ud g'_{\mbt, 2}(0)}{g'_{\mbt, 2}(0)}=&\,\phi'_{\mbt, 1}\left(\xi^{(1)}_{t_1}\right)^2\ud t_1+\left(\phi'_{\mbt, 2}\left(\xi^{(2)}_{t_2}\right)^2-1\right)\ud t_2;\\
    \frac{\ud g_{\mbt}'(0)}{g_{\mbt}'(0)}=&\sum_{j=1}^2\phi_{\mbt, j}'\left(\xi_{t_j}^{(j)}\right)^2\ud t_j. 
\end{split}
    \end{align}

    From the assumption that $\eta^{(1)}$ and $\eta^{(2)}$ are two independent radial $\SLE_\k$ under $\PP$, we have that $\xi^{(1)}=\sqrt{\kappa}B^{(1)} + \t_1$ and $\xi^{(2)}=\sqrt{\kappa}B^{(2)} + \t_2$ where $B^{(1)}$ and $B^{(2)}$ are two independent Brownian motions. 
From this, It\^o's calculus gives
        \begin{align}
    \ud \theta_{\mbt}^{(1)}=\ud \phi_{\mbt, 1}\left(\xi^{(1)}_{t_1}\right)=&\,\phi'_{\mbt, 1}\left(\xi^{(1)}_{t_1}\right)\ud \xi_{t_1}^{(1)}-\kappa h\phi''_{\mbt, 1}\left(\xi^{(1)}_{t_1}\right)\ud t_1\notag\\
    &+\cot\left((\theta_{\mbt}^{(1)}-\theta_{\mbt}^{(2)})/2\right)\phi'_{\mbt, 2}\left(\xi^{(2)}_{t_2}\right)^2\ud t_2,\label{eqn::2SLE_ind_driving1}\end{align}
    where we used the expansions
    \begin{align}
    \partial_{t_1} \phi_{\mbt,1} (w) & = \left(\phi_{\mbt,1}^{(1)}\right)' \left(\xi_{t_1}^{(1)}\right)^2 \cot \left((\phi_{\mbt,1}^{(1)} (w) - \t_{\mbt}^{(1)})/2\right)  -  \left(\phi_{\mbt,1}^{(1)}\right)' (w) \cot \left((w - \xi_{t_1}^{(1)})/2\right)\notag\\
     & =  -3 \left(\phi_{\mbt,1}^{(1)}\right)'' \left(\xi_{t_1}^{(1)}\right) + O(w - \xi_{t_1}^{(1)}) \label{eq:part_t_1_phi_t_1}, \quad\text{as }w \to \xi_{t_1}^{(1)},\\
      \partial_{t_2} \phi_{\mbt,1} (w) & =  \cot \left((\phi_{\mbt,1}^{(1)} (w) - \t_{\mbt}^{(2)})/2\right) \phi_{\mbt,2}' \left(\xi_{t_2}^{(2)}\right)^2.\label{eq:part_t_2_phi_t_1}
    \end{align}
    Similarly, we have
    \begin{align}
    \ud \theta_{\mbt}^{(2)}=\ud \phi_{\mbt, 2}\left(\xi^{(2)}_{t_2}\right)=&\phi'_{\mbt, 2}\left(\xi^{(2)}_{t_2}\right)\ud \xi_{t_2}^{(2)}-\kappa h\phi''_{\mbt, 2}\left(\xi^{(2)}_{t_2}\right)\ud t_2\notag\\
    &+\cot\left((\theta_{\mbt}^{(2)}-\theta_{\mbt}^{(1)})/2\right)\phi'_{\mbt, 1}\left(\xi^{(1)}_{t_1}\right)^2\ud t_1.\label{eqn::2SLE_ind_driving2}
\end{align}
Taking derivatives with respect to $w$ in \eqref{eq:part_t_1_phi_t_1} and \eqref{eq:part_t_2_phi_t_1} and letting $w \to \xi^{(1)}_{t_1}$ we obtain 
\begin{align*}
    \frac{(\partial_{t_1}\phi'_{\mbt, 1})\left(\xi^{(1)}_{t_1}\right)}{\phi'_{\mbt, 1}\left(\xi^{(1)}_{t_1}\right)} =& \frac{1}{2}\frac{\phi''_{\mbt, 1}\left(\xi^{(1)}_{t_1}\right)^2}{\phi'_{\mbt, 1}\left(\xi^{(1)}_{t_1}\right)^2}-\frac{4}{3}\frac{\phi'''_{\mbt, 1}\left(\xi^{(1)}_{t_1}\right)}{\phi'_{\mbt, 1}\left(\xi^{(1)}_{t_1}\right)}  - \frac {1}{6} \left(\phi'_{\mbt, 1}\left(\xi^{(1)}_{t_1}\right)^2 - 1\right),\\
    \frac{(\partial_{t_2}\phi'_{\mbt, 1})\left(\xi^{(1)}_{t_1}\right)}{\phi'_{\mbt, 1}\left(\xi^{(1)}_{t_1}\right)} =& - \frac{1}{2}\csc^2\left((\theta_{\mbt}^{(1)}-\theta_{\mbt}^{(2)})/2\right)\phi'_{\mbt, 2}\left(\xi^{(2)}_{t_2}\right)^2.
\end{align*}
It\^o's formula gives
\begin{align*}
\frac{\ud \phi'_{\mbt, 1}\left(\xi^{(1)}_{t_1}\right)}{\phi'_{\mbt, 1}\left(\xi^{(1)}_{t_1}\right)}
=&\frac{\phi''_{\mbt, 1}\left(\xi^{(1)}_{t_1}\right)}{\phi'_{\mbt, 1}\left(\xi^{(1)}_{t_1}\right)}\ud \xi_{t_1}^{(1)}-\frac{1}{2}\csc^2\left((\theta_{\mbt}^{(1)}-\theta_{\mbt}^{(2)})/2\right)\phi'_{\mbt, 2}\left(\xi^{(2)}_{t_2}\right)^2\ud t_2\\
&+ \left(\frac{\kappa}{2}\frac{\phi'''_{\mbt, 1}\left(\xi^{(1)}_{t_1}\right)}{\phi'_{\mbt, 1}\left(\xi^{(1)}_{t_1}\right)}+\frac{1}{2}\frac{\phi''_{\mbt, 1}\left(\xi^{(1)}_{t_1}\right)^2}{\phi'_{\mbt, 1}\left(\xi^{(1)}_{t_1}\right)^2}-\frac{4}{3}\frac{\phi'''_{\mbt, 1}\left(\xi^{(1)}_{t_1}\right)}{\phi'_{\mbt, 1}\left(\xi^{(1)}_{t_1}\right)}-\frac{1}{6}\left(\phi'_{\mbt, 1}\left(\xi^{(1)}_{t_1}\right)^2-1\right)
\right)\ud t_1,\\
\frac{\ud \phi'_{\mbt, 2}\left(\xi^{(2)}_{t_2}\right)}{\phi'_{\mbt, 2}\left(\xi^{(2)}_{t_2}\right)}
=&\frac{\phi''_{\mbt, 2}\left(\xi^{(2)}_{t_2}\right)}{\phi'_{\mbt, 2}\left(\xi^{(2)}_{t_2}\right)}\ud \xi_{t_2}^{(2)}-\frac{1}{2}\csc^2\left((\theta_{\mbt}^{(2)}-\theta_{\mbt}^{(1)})/2\right)\phi'_{\mbt, 1}\left(\xi^{(1)}_{t_1}\right)^2\ud t_1\notag\\
&+  \left(\frac{\kappa}{2}\frac{\phi'''_{\mbt, 2}\left(\xi^{(2)}_{t_2}\right)}{\phi'_{\mbt, 2}\left(\xi^{(2)}_{t_2}\right)}+\frac{1}{2}\frac{\phi''_{\mbt, 2}\left(\xi^{(2)}_{t_2}\right)^2}{\phi'_{\mbt, 2}\left(\xi^{(2)}_{t_2}\right)^2}-\frac{4}{3}\frac{\phi'''_{\mbt, 2}\left(\xi^{(2)}_{t_2}\right)}{\phi'_{\mbt, 2}\left(\xi^{(2)}_{t_2}\right)}-\frac{1}{6}\left(\phi'_{\mbt, 2}\left(\xi^{(2)}_{t_2}\right)^2-1\right)
\right)\ud t_2.
\end{align*}
 
Now we are ready to prove that $M_\mbt(\LG_{\mu})$ is a two-time-parameter local martingale. Combining the above with~\eqref{eq:blm_2_curves},~\eqref{eqn::BPZ1_Gmu} and~\eqref{eqn::BPZ2_Gmu}, we have
\begin{align}
    \frac{\ud M_{\mbt}(\LG_{\mu})}{M_{\mbt}(\LG_{\mu})}=&\left(\frac{3-\mu^2}{2\kappa}-\widetilde{h}\right)\frac{\ud g_{\mbt}'(0)}{g_{\mbt}'(0)}+\widetilde{h}\sum_{j=1}^2\frac{\ud g'_{\mbt, j}(0)}{g'_{\mbt, j}(0)}+\frac{c}{2} \ud m_{\mbt}\notag\\
    &+\sum_{j=1}^2\left(\frac{\partial_j\LG_{\mu}}{\LG_{\mu}}\ud \theta_{\mbt}^{(j)}+\frac{\kappa}{2}\frac{\partial_{jj}\LG_{\mu}}{\LG_{\mu}}\phi'_{\mbt, j}\left(\xi^{(j)}_{t_j}\right)^2\ud t_j\right)\notag\\
    &+\sum_{j=1}^2 \left(h\frac{\ud \phi'_{\mbt, j}\left(\xi^{(j)}_{t_j}\right)}{\phi'_{\mbt, j}\left(\xi^{(j)}_{t_j}\right)}+\frac{\kappa h(h-1)}{2}\frac{\phi''_{\mbt, j}\left(\xi^{(j)}_{t_j}\right)^2}{\phi'_{\mbt, j}\left(\xi^{(j)}_{t_j}\right)^2}\ud t_j\right)\notag\\
    &+\kappa h \sum_{j=1}^2 \frac{\partial_j\LG_{\mu}}{\LG_{\mu}}\phi''_{\mbt, j}\left(\xi^{(j)}_{t_j}\right)\ud t_j\notag\\
    =&    \sum_{j=1}^2\left(\frac{\partial_j\LG_{\mu}}{\LG_{\mu}}\phi'_{\mbt, j}\left(\xi^{(j)}_{t_j}\right)+h\frac{\phi''_{\mbt, j}\left(\xi^{(j)}_{t_j}\right)}{\phi'_{\mbt, j}\left(\xi^{(j)}_{t_j}\right)}\right)\ud \xi_{t_j}^{(j)}.\label{eqn::2SLEspiral_mart_aux}
\end{align}
In particular, 
$M_{\mbt}(\LG_{\mu})$ is a two-time-parameter local martingale under $\PP$. 
\end{proof}

\begin{definition}\label{def::2SLEspiral}
Fix $\kappa\in (0,8)$ and $\theta_1<\theta_2<\theta_1+2\pi$.
For $\mu\in\mathbb{R}$, we define $\LG_{\mu}$ as in~\eqref{eqn::2SLEspiral_pf}. 
Let $\PP$ denote the probability measure under which $(\eta^{(1)}, \eta^{(2)})$ are two independent radial $\SLE_{\kappa}$ in $\m D$ starting from $\ee^{\ii\theta_1}$ and $\ee^{\ii\theta_2}$ respectively. We define $M_{\mbt}(\LG_{\mu})$ as~\eqref{eqn::2SLE_mart_new} in Lemma~\ref{lem::2SLE_spiral_mart} and denote by $\PP(\LG_{\mu})$ the probability measure obtained by tilting $\PP$ by $M_{\mbt}(\LG_{\mu})$ and call it the \emph{two-sided radial $\SLE_{\kappa}$ with spiraling rate $\mu$} in $(\m D; \ee^{\ii\theta_1}, \ee^{\ii\theta_2}; 0)$.
\end{definition}

 From Corollary~\ref{cor::2SLEspiral_marginal} below, it follows that when $\kappa\le 4$, two-sided radial $\SLE_{\kappa}$ with spiral is well-defined for all time and the two curves do not touch each other before they reach the origin. When $\kappa\in (4,8)$, the above definition of two-sided radial $\SLE_{\kappa}$ with spiral is only valid up to the time the two curves touch each other. When the spiraling rate is $0$, we obtain the standard two-sided radial $\SLE$ analyzed in~\cite{Lawler2sidedSLE,field2016twosided,Lawler_Zhou_natural,HealeyLawlerNSidedSLE}. 

\begin{cor}\label{cor::2SLEspiral_marginal}
Under $\PP(\LG_\mu)$, for every $\mbt = (t_1, t_2)$, 
    \begin{itemize}
        \item $g_{\mbt}(\eta^{(1)})$ is radial $\SLE_\k^{\mu} (2)$ in $\m D$ starting from $\exp(\ii \t_{\mbt}^{(1)})$ with force point at $\exp(\ii \t_{\mbt}^{(2)})$,
        \item $g_{\mbt}(\eta^{(2)})$ is radial $\SLE_\k^{\mu} (2)$ in $\m D$ starting from $\exp(\ii \t_{\mbt}^{(2)})$ with force point at $\exp(\ii \t_{\mbt}^{(1)})$.
    \end{itemize}
    In particular, $\PP(\LG_{\mu})$ on pairs $(\eta^{(1)}, \eta^{(2)})$ satisfies~\textbf{(CI)}, \textbf{(DMP)}, \textbf{(MARG)} and~\textbf{(INT)} with 
    \begin{equation}\label{eqn::2SLE_marg}
  b_j=\kappa\partial_j\log\LG_{\mu},\quad j=1,2.      
    \end{equation}
\end{cor}
\begin{proof}
Combining~\eqref{eqn::2SLEspiral_mart_aux} with~\eqref{eqn::2SLEspiral_pf}, we have
\begin{align*}
\frac{\ud M_{\mbt}(\LG_{\mu})}{M_{\mbt}(\LG_{\mu})}=&\left(\frac{1}{\kappa}\cot\left((\theta_{\mbt}^{(1)}-\theta_{\mbt}^{(2)})/2\right)\phi'_{\mbt, 1}\left(\xi^{(1)}_{t_1}\right)+\frac{\mu}{\kappa}\phi'_{\mbt, 1}\left(\xi^{(1)}_{t_1}\right)+h\frac{\phi''_{\mbt, 1}\left(\xi^{(1)}_{t_1}\right)}{\phi'_{\mbt, 1}\left(\xi^{(1)}_{t_1}\right)}\right)\ud \xi_{t_1}^{(1)}\\
&+\left(\frac{1}{\kappa}\cot\left((\theta_{\mbt}^{(2)}-\theta_{\mbt}^{(1)})/2\right)\phi'_{\mbt, 2}\left(\xi^{(2)}_{t_2}\right)+\frac{\mu}{\kappa}\phi'_{\mbt, 2}\left(\xi^{(2)}_{t_2}\right)+h\frac{\phi''_{\mbt, 2}\left(\xi^{(2)}_{t_2}\right)}{\phi'_{\mbt, 2}\left(\xi^{(2)}_{t_2}\right)}\right)\ud \xi_{t_2}^{(2)}.
\end{align*}
From Girsanov's theorem and~\eqref{eqn::2SLE_ind_driving1} and~\eqref{eqn::2SLE_ind_driving2}, under $\PP(\LG_{\mu})$, we have 
   \begin{align}
       \ud \theta_{\mbt}^{(1)}=&\sqrt{\kappa}\phi_{\mbt, 1}'\left(\xi_{t_1}^{(1)}\right)\ud \widetilde{B}_{t_1}^{(1)}
       +\mu\phi_{\mbt, 1}'\left(\xi_{t_1}^{(1)}\right)^2\ud t_1\notag\\
       &+\cot\left((\theta_{\mbt}^{(1)}-\theta_{\mbt}^{(2)})/2\right)
       \left(\phi_{\mbt, 1}'\left(\xi_{t_1}^{(1)}\right)^2\ud t_1+\phi_{\mbt, 2}'\left(\xi_{t_2}^{(2)}\right)^2\ud t_2\right),\label{eqn::2SLE_driving1}\\
        \ud \theta_{\mbt}^{(2)}=&\sqrt{\kappa}\phi_{\mbt, 2}'\left(\xi_{t_2}^{(2)}\right)\ud \widetilde{B}_{t_2}^{(2)}
       +\mu\phi_{\mbt, 2}'\left(\xi_{t_2}^{(2)}\right)^2\ud t_2\notag\\
       &+\cot\left((\theta_{\mbt}^{(2)}-\theta_{\mbt}^{(1)})/2\right)
       \left(\phi_{\mbt, 1}'\left(\xi_{t_1}^{(1)}\right)^2\ud t_1+\phi_{\mbt, 2}'\left(\xi_{t_2}^{(2)}\right)^2\ud t_2\right),\label{eqn::2SLE_driving2}
   \end{align}
   where $\widetilde{B}^{(1)}$ and $\widetilde{B}^{(2)}$ are two independent Brownian motions under $\PP(\LG_{\mu})$.  

    Therefore, taking into account the variation of the capacity parametrization \eqref{eq:cap_change},  \eqref{eqn::2SLE_driving1} and \eqref{eqn::2SLE_driving2} show that under $\PP(\LG_\mu)$, for every $\mbt = (t_1, t_2)$, $g_{\mbt}(\eta^{(1)})$ is a radial $\SLE_\k^{\mu} (2)$ in $\m D$ starting from $\exp(\ii \t_{\mbt}^{(1)})$ with force point at $\exp(\ii \t_{\mbt}^{(2)})$. Similarly, $g_{\mbt}(\eta^{(2)})$ is a radial $\SLE_\k^{\mu} (2)$ in $\m D$ starting from $\exp(\ii \t_{\mbt}^{(2)})$ with force point at $\exp(\ii \t_{\mbt}^{(1)})$. These imply~\textbf{(CI)}, \textbf{(DMP)}, \textbf{(MARG)} and~\textbf{(INT)}. Eq.~\eqref{eqn::2SLE_marg} follows from Remark~\ref{rem::2SLE_spiral_marg}. 
\end{proof}

\subsection{Resampling property of two-sided radial SLE with spiral}
In this section, we will prove the resampling property of two-sided radial SLE with spiral as we described in Section~\ref{subsec::intro_resampling} and in Corollary~\ref{cor::commutation_resampling}. We fix $\kappa\in (0,4]$ in the following Theorem~\ref{thm::resampling_property}, because we will use the boundary perturbation property in Lemma~\ref{lem::radialSLE_boundary_perturb} with $\kappa\in (0,4]$ in the proof.

\begin{thm}[Resampling property]\label{thm::resampling_property}
Fix $\kappa\in (0,4]$, $\mu\in\mathbb{R}$ and $\theta_1<\theta_2<\theta_1+2\pi$. Suppose $(\eta^{(1)}, \eta^{(2)})\sim \PP(\LG_{\mu})$ is the two-sided radial $\SLE_{\kappa}$ with spiraling rate $\mu$ as in Definition~\ref{def::2SLEspiral}, then the following statements hold.
\begin{itemize}
    \item The marginal law of $\eta^{(1)}$ is radial $\SLE_{\kappa}^{\mu}(2)$ in $\m D$ starting from $\ee^{\ii\theta_1}$ with force point $\ee^{\ii\theta_2}$ and spiraling rate $\mu$.  
    \item Given $\eta^{(1)}$, the conditional law of $\eta^{(2)}$ is chordal $\SLE_{\kappa}$ in 
    $\m D\setminus\eta^{(1)}$ from $\ee^{\ii\theta_2}$ to $0$.
\end{itemize}
The same is true when we interchange $\eta^{(1)}$ and $\eta^{(2)}$.
\end{thm}

\begin{remark}\label{rem:coupling_GFF}
    Radial SLE with spiral appears as flow lines in the setup  
of imaginary geometry~\cite{IG4}.
An alternative way to construct two-sided radial $\SLE_{\kappa}$ with spiraling rate $\mu$ is to take a pair of flow lines of 
\[\Gamma+\frac{(8-\kappa)}{2\sqrt{\kappa}}\arg(\cdot)+\frac{\mu}{\sqrt \k}\log|\cdot|,\]
where $\Gamma$ is a GFF in $\m D$ with properly chosen boundary data, and the angles of the two flow lines are also chosen properly.
See Figure~\ref{fig::simulation}.
Using such coupling, one is able to derive the resampling property in Theorem~\ref{thm::resampling_property}; see~\cite[Prop.\,3.28]{IG4}.  
However, our proof of the resampling property in 
Section~\ref{subsec::2SLEspiral} does not use the coupling with imaginary geometry. We derive it directly using a refined analysis of the Radon--Nikodym derivative $M_{\mbt}(\LG_{\mu})$ in Definition~\ref{def::2SLEspiral}.
\end{remark}

\begin{proof}[Proof of Theorem~\ref{thm::resampling_property}]
Since $\LG_\mu (\t_1, \t_2) = \l \,\LG_\mu (\t_2, \t_1+2\pi)$ for some constant $\l$ which does not depend on $\t_1$ and $\t_2$, $\eta^{(1)}$ and $\eta^{(2)}$ are interchangeable.  Therefore, it suffices to show the bullet points in the statement.
The marginal law of $\eta^{(1)}$ is a consequence of Corollary~\ref{cor::2SLEspiral_marginal}, and it remains to show the conditional law of $\eta^{(2)}$ given $\eta^{(1)}$. 

The law of $\eta^{(1)}_{[0,t_1]}$ is the same as that of radial $\SLE_{\kappa}$ weighted by the following local martingale:
    \begin{align*}
    M_{(t_1, 0)}(\LG_{\mu})=&\left(g_{t_1}^{(1)}\right)'(0)^{\frac{3-\mu^2}{2\kappa}-\widetilde{h}}\times \LG_{\mu}\left(\xi_{t_1}^{(1)}, \phi_{t_1}^{(1)}(\theta_2)\right)\times\left(\phi_{t_1}^{(1)}\right)'(\theta_2)^h\left(g_{t_1}^{(1)}\right)'(0)^{\widetilde{h}}.
    \end{align*}
    The law of $(\eta^{(1)}, \eta^{(2)})$ is the same as that of two independent radial $\SLE_{\kappa}$ weighted by the local martingale $M_{\mbt}(\LG_{\mu})$. Therefore, the conditional law of $\eta^{(2)}_{[0,t_2]}$ given $\eta^{(1)}_{[0,t_1]}$ is the same as that of radial $\SLE_{\kappa}$ weighted by 
    \begin{align*}
    \frac{M_{(t_1, t_2)}(\LG_{\mu})}{M_{(t_1, 0)}(\LG_{\mu})}=&\underbrace{\frac{\phi_{\mbt, 2}'\left(\xi_{t_2}^{(2)}\right)^hg_{\mbt, 2}'(0)^{\widetilde{h}}\exp\left(\frac{c}{2}m_{\mbt}\right)}{\left(\phi_{t_1}^{(1)}\right)'(\theta_2)^h\left(g_{t_1}^{(1)}\right)'(0)^{\widetilde{h}}}}_{=:P_{\mbt}}\times g_{\mbt, 1}'(0)^{\frac{3-\mu^2}{2\kappa}}\frac{\LG_{\mu}\left(\theta_{\mbt}^{(1)}, \theta_{\mbt}^{(2)}\right)}{\LG_{\mu}\left(\xi_{t_1}^{(1)}, \phi_{t_1}^{(1)}(\theta_2)\right)}\phi_{\mbt, 1}'\left(\xi_{t_1}^{(1)}\right)^h. 
    \end{align*}
From the boundary perturbation property Lemma~\ref{lem::radialSLE_boundary_perturb}, radial $\SLE_{\kappa}$ in $(\m D; \ee^{\ii\theta_2}; 0)$ weighted by $P_{\mbt}$ has the same law as that of radial $\SLE_{\kappa}$ in $(\m D\setminus\eta^{(1)}_{[0,t_1]}; \ee^{\ii\theta_2}; 0)$. Thus, the conditional law of $\eta^{(2)}_{[0,t_2]}$ given $\eta^{(1)}_{[0,t_1]}$ is the same as that of radial $\SLE_{\kappa}$ in $(\m D\setminus\eta^{(1)}_{[0,t_1]}; \ee^{\ii\theta_2}; 0)$ weighted by 
\begin{align}\label{eqn::2SLE_conditional_RN1}
g_{\mbt, 1}'(0)^{\frac{3-\mu^2}{2\kappa}}\frac{\LG_{\mu}\left(\theta_{\mbt}^{(1)}, \theta_{\mbt}^{(2)}\right)}{\LG_{\mu}\left(\xi_{t_1}^{(1)}, \phi_{t_1}^{(1)}(\theta_2)\right)}\phi_{\mbt, 1}'\left(\xi_{t_1}^{(1)}\right)^h. 
\end{align}
Note that the law of radial $\SLE_{\kappa}(2)$ in $\m D\setminus\eta^{(1)}_{[0,t_1]}$ from $\ee^{\ii\theta_2}$ to $0$ with force point $\eta^{(1)}_{t_1}$ is the same as that of
radial $\SLE_{\kappa}$ in $(\m D\setminus\eta^{(1)}_{[0,t_1]}; \ee^{\ii\theta_2}; 0)$ weighted by~\cite{SW05}: 
\begin{align}\label{eqn::2SLE_conditional_RN2}
g_{\mbt, 1}'(0)^{\frac{3}{2\kappa}}\frac{\LG_{0}\left(\theta_{\mbt}^{(1)}, \theta_{\mbt}^{(2)}\right)}{\LG_{0}\left(\xi_{t_1}^{(1)}, \phi_{t_1}^{(1)}(\theta_2)\right)}\phi_{\mbt, 1}'\left(\xi_{t_1}^{(1)}\right)^h. 
\end{align}
Combining~\eqref{eqn::2SLE_conditional_RN1} and~\eqref{eqn::2SLE_conditional_RN2}, we see that the conditional law of $\eta^{(2)}_{[0,t_2]}$ given $\eta^{(1)}_{[0,t_1]}$ is the same as that of radial $\SLE_{\kappa}(2)$ in $\m D\setminus\eta^{(1)}_{[0,t_1]}$ from $\ee^{\ii\theta_2}$ to $0$ with force point $\eta^{(1)}_{t_1}$ weighted by 
\begin{align}\label{eqn::2SLE_conditional_RN3}
R_{\mbt}=&g_{\mbt, 1}'(0)^{\frac{-\mu^2}{2\kappa}}\frac{\LG_{\mu}\left(\theta_{\mbt}^{(1)}, \theta_{\mbt}^{(2)}\right)}{\LG_{\mu}\left(\xi_{t_1}^{(1)}, \phi_{t_1}^{(1)}(\theta_2)\right)}\frac{\LG_{0}\left(\xi_{t_1}^{(1)}, \phi_{t_1}^{(1)}(\theta_2)\right)}{\LG_{0}\left(\theta_{\mbt}^{(1)}, \theta_{\mbt}^{(2)}\right)}\notag\\
=& g_{\mbt, 1}'(0)^{\frac{-\mu^2}{2\kappa}} \exp\left(\frac{\mu}{\kappa}\left(\theta_{\mbt}^{(1)}+\theta_{\mbt}^{(2)}\right)-\frac{\mu}{\kappa}\left(\xi_{t_1}^{(1)}+\phi_{t_1}^{(1)}(\theta_2)\right)\right).
\end{align}
Let $U_1$ be a neighborhood of $\ee^{\ii\theta_1}$ containing the origin as an interior point, and let $U_2$ be a neighborhood of $\ee^{\ii\theta_2}$, and suppose there is a positive distance between $U_1$ and $U_2$. See the left-hand side of Figure~\ref{fig::limit}.
We will prove
in Lemma~\ref{lem::SLEradalTochordal} that, when restricting to the event $\{\eta^{(1)}\subset U_1, \eta^{(2)}_{[0,t_2]}\subset U_2\}$, the law of radial $\SLE_{\kappa}(2)$ in $\m D\setminus\eta^{(1)}_{[0,t_1]}$ from $\ee^{\ii\theta_2}$ to $0$ with force point $\eta^{(1)}_{t_1}$ converges to chordal $\SLE_{\kappa}$ in $(\m D\setminus\eta^{(1)}; \ee^{\ii\theta_2}, 0)$ as $t_1\to\infty$, and we will show in Lemma~\ref{lem::Rt_limit} that $R_{\mbt}\to 1$ almost surely as $t_1\to\infty$ uniformly as long as $\{\eta^{(1)}\subset U_1, \eta^{(2)}_{[0,t_2]}\subset U_2\}$. Combining these two parts, we obtain that the conditional law of $\eta^{(2)}$ before exiting $U_2$ given $\eta^{(1)}$, on the event $\{\eta^{(1)} \subset U_1\}$ is the same as that of chordal $\SLE_{\kappa}$ before exiting $U_2$. By shrinking the domain $U_1$ and enlarging $U_2$, we obtain that conditionally on $\eta^{(1)}$, for all neighborhoods $U_2 \subset \m D \setminus \eta^{(1)}$ of $\ee^{\ii \theta_2}$, the law of $\eta^{(2)}$ before exiting $U_2$ coincides with the law of chordal $\SLE_{\kappa}$ before exiting $U_2$. Let $(U_{2,n})_{n \ge 1}$ be an increasing family of open neighborhoods such that $\cup_{n\ge 1} U_{2,n} = \m D \setminus \eta^{(1)}$. Since chordal $\SLE_{\kappa}$ for $\kappa \le 4$ is almost surely transient and does not touch the boundary of the domain, its law is characterized by the restriction to the domains $U_{2,n}$. This shows that conditionally on $\eta^{(1)}$, $\eta^{(2)}$ is a chordal $\SLE_{\kappa}$ in $\m D\setminus\eta^{(1)}$.
\end{proof}

\begin{lem}\label{lem::SLEradalTochordal}
    Fix $\kappa\in (0,4]$ and $\theta_1<\theta_2<\theta_1+2\pi$. We assume $\eta^{(1)}\in \mf X(\m D; \ee^{\ii\theta_1}; 0)$, namely a simple curve  in $\m D$ from $\ee^{\ii \t_1}$ to $0$. 
    \begin{itemize}
        \item For $t_1 \in (0,\infty)$, we denote by $Q_{t_1}$ the law of radial $\SLE_{\kappa}(2)$ in $\m D\setminus\eta^{(1)}_{[0,t_1]}$ from $\ee^{\ii\theta_2}$ to $0$ with force point $\eta^{(1)}_{t_1}$. 
        \item We denote by $Q_{\infty}$ the conditional law of chordal $\SLE_{\kappa}$ in $(\m D\setminus\eta^{(1)}; \ee^{\ii\theta_2}, 0)$. 
    \end{itemize}
    Suppose $U$ is any neighborhood of $\ee^{\ii\theta_2}$ in $\m D\setminus\eta^{(1)}$ such that there is a positive distance between $U$ and $\eta^{(1)}$ and let $\tau$ be the first time $\eta^{(2)}$ exits $U$. Then the laws of the curve $\eta^{(2)}$ restricted to $[0,\tau]$ under $Q_{t_1}$ and $Q_\infty$ are absolutely continuous.  We have
    \begin{align*}
        \lim_{t_1\to \infty}\frac{\ud Q_{t_1}}{\ud Q_{\infty}}\left(\eta^{(2)}_{[0,\tau]}\right)=1,\quad Q_{\infty}-a.s.
    \end{align*}
\end{lem}
\begin{proof}
The convergence of radial $\SLE_{\kappa}$ in $(\m D\setminus\eta^{(1)}_{[0,t_1]}; \ee^{\ii\theta_2}; 0)$ to chordal $\SLE_{\kappa}$ in $(\m D\setminus\eta^{(1)}; \ee^{\ii\theta_2}, 0)$ is well known; see e.g. \cite[Lem.~3.2]{Viklund_LERWE} and \cite{Lawler_Lind}. In our setup, we need the convergence of the radial $\SLE_{\kappa}(2)$ process. 

We fix a sequence of conformal maps $f_{t_1}: \m D\setminus\eta^{(1)}_{[0,t_1]}\to \m D$ and a conformal map $f_{\infty}: \m D\setminus\eta^{(1)}\to \m D$ such that 
\[f_{t_1}(\eta^{(1)}_{t_1})=\ee^{\ii\theta_1}, \quad f_{t_1}(\ee^{\ii\theta_2})=\ee^{\ii\theta_2}, \qquad f_{\infty}(0)=\ee^{\ii\theta_1}, \quad f_{\infty}(\ee^{\ii\theta_2})=\ee^{\ii\theta_2};\]
 and that $f_{t_1}$ converges to $f_{\infty}$ locally uniformly. Note that $(f_{t_1})_*(Q_{t_1})$ is the same as radial $\SLE_{\kappa}(2)$ in $\m D$ from $\ee^{\ii\theta_2}$ to $w=f_{t_1}(0)$ with force point $\ee^{\ii\theta_1}$, and $(f_{\infty})_*(Q_{\infty})$ is the same as chordal $\SLE_{\kappa}$ in $(\m D; \ee^{\ii\theta_2}, \ee^{\ii\theta_1})$.  

The law of radial $\SLE_{\kappa}(2)$ in $\m D$ from $\ee^{\ii\theta_2}$ to $w$ with force point $\ee^{\ii\theta_1}$ is the same as the law of radial 
$\SLE_{\kappa}$ in $(\m D; \ee^{\ii\theta_2}; 0)$ weighted by the local martingale~\cite{SW05}:
\begin{align}\label{eqn::radialtochordal_aux1}
&g_t'(0)^{-\widetilde{h}}\phi_t'(\theta_1)^h\left(\sin((\xi_t-\phi_t(\theta_1))/2)\right)^{\frac{2}{\kappa}}\notag\\
&\times|g_t'(w)|^{\widetilde{h}+\frac{3}{2\kappa}}(1-|g_t(w)|^2)^{\frac{(\kappa-8)^2}{8\kappa}}|g_t(w)-\ee^{\ii\xi_t}|^{1-\frac{8}{\kappa}}|g_t(w)-\ee^{\ii\phi_t(\theta_1)}|^{1-\frac{8}{\kappa}}.
\end{align}
The law of chordal $\SLE_{\kappa}$ in $(\m D; \ee^{\ii\theta_2}, \ee^{\ii\theta_1})$ is the same as the law of radial $\SLE_{\kappa}$ in $(\m D; \ee^{\ii\theta_2}; 0)$ weighted by the local martingale~\cite{SW05}:
\begin{equation}\label{eqn::radialtochordal_aux2}
g_t'(0)^{-\widetilde{h}}\phi_t'(\theta_1)^h\left(\sin((\xi_t-\phi_t(\theta_1))/2)\right)^{-2h}. 
\end{equation}
Combining~\eqref{eqn::radialtochordal_aux1} and~\eqref{eqn::radialtochordal_aux2}, we find that the law $(f_{t_1})_*(Q_{t_1})$ is the same as $(f_{\infty})_*(Q_{\infty})$ weighted by the local martingale
\begin{align}\label{eqn::radialtochordal_aux3}
M_t(w)=&\left(\frac{\sin((\xi_t-\phi_t(\theta_1))/2)}{\sin((\theta_2-\theta_1)/2)}\right)^{\frac{8}{\kappa}-1}|g_t'(w)|^{\widetilde{h}+\frac{3}{2\kappa}}\left(\frac{1-|g_t(w)|^2}{1-|w|^2}\right)^{\frac{(\kappa-8)^2}{8\kappa}}\notag\\
&\times\left|\frac{g_t(w)-\ee^{\ii\xi_t}}{w-\ee^{\ii\theta_2}}\right|^{1-\frac{8}{\kappa}}\left|\frac{g_t(w)-\ee^{\ii\phi_t(\theta_1)}}{w-\ee^{\ii\theta_1}}\right|^{1-\frac{8}{\kappa}}. 
\end{align}
As $t_1\to\infty$, we have $w=f_{t_1}(0)\to \ee^{\ii\theta_1}$. Thus,
\begin{align*}
   & |g_t'(w)|\to |g_t'(\ee^{\ii\theta_1})|, \quad \frac{1-|g_t(w)|^2}{1-|w|^2}\to |g_t'(\ee^{\ii\theta_1})|, \\
   & \left|\frac{g_t(w)-\ee^{\ii\xi_t}}{w-\ee^{\ii\theta_2}}\right|\to \frac{\sin((\xi_t-\phi_t(\theta_1))/2)}{\sin((\theta_2-\theta_1)/2)}, \quad \left|\frac{g_t(w)-\ee^{\ii\phi_t(\theta_1)}}{w-\ee^{\ii\theta_1}}\right|\to |g_t'(\ee^{\ii\theta_1})|.
\end{align*}
Consequently, $M_t(w)\to 1$ almost surely. These give the desired conclusion. 
\end{proof}

\begin{figure}[ht!]
    \begin{center}
    \includegraphics[width=\textwidth]{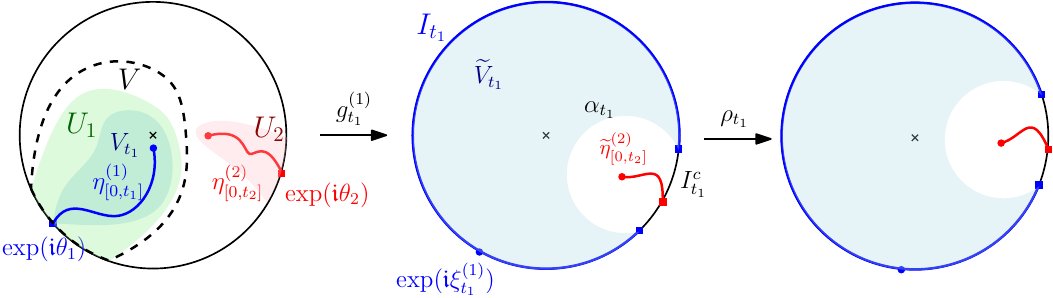}
    \end{center}
    \caption{Illustration of the domains $U_1$, $U_2$, $V_{t_1}$ and $\widetilde V_{t_1} := g_{t_1}^{(1)} (V_{t_1})$ and the rotation map $\rho_{t_1}$ used in the proof of Theorem~\ref{thm::resampling_property} and Lemma~\ref{lem::Rt_limit}.}
    \label{fig::limit}
\end{figure}

\begin{lem}\label{lem::Rt_limit}
Assume the same notation as in the proof of Theorem~\ref{thm::resampling_property} and recall that $R_{\mbt}$ is defined in~\eqref{eqn::2SLE_conditional_RN3}. We have
\begin{align*}
    \lim_{t_1\to\infty}R_{\mbt}=1,\quad a.s.
\end{align*}
uniformly on the event $\{\eta^{(1)}\subset U_1, \eta^{(2)}_{[0,t_2]}\subset U_2\}$.
\end{lem}
\begin{proof}
We write $I_{t_1}$ for the arc in $\partial \m D$ that is the image of both sides of $\eta^{(1)}_{[0,t_1]}$ under the conformal map $g_{t_1}^{(1)}$ extended to the boundary; see Figure~\ref{fig::gtcommutation}. It is easy to see the harmonic measure of $\partial \m D$ seen from $0$ of the domain $\m D \setminus \eta^{(1)}_{[0,t_1]}$ is decreasing to $0$,   uniformly in $\eta^{(1)}$, as $t_1 \to \infty$. Therefore,
$$|I_{t_1}^c| = |\partial \m D \setminus I_{t_1}| \xrightarrow[]{t_1 \to \infty } 0$$
uniformly, where $I_{t_1}^c = \partial \m D \setminus I_{t_1}$. 

Since $U_1$ is at a positive distance from $U_2$, there exists $\lambda \in (0,1)$ such that the neighborhood 
$$V = \left\{z \in \m D \,|\, \m P_z\left( \beta \text{ hits } U_1 \text{ before exiting } \m D\right) \ge \lambda\right\}$$
satisfies $ \overline{U_2} \cap V = \emptyset$, where $\m P_z$ is the law of two-dimensional Brownian motion $\beta$ starting from $z$. 
Let 
$$V_{t_1}= \left\{z \in \m D \,|\, \m P_z\left( \beta \text{ hits } \eta^{(1)}_{[0,t_1)} \text{ before exiting } \m D\right) \ge \lambda\right\} \subset V.$$
The image $\widetilde V_{t_1} := g_{t_1}^{(1)} (V_{t_1})$ is bounded by $I_{t_1}$ and a circular arc $\a_{t_1}$ meeting the endpoints of $I_{t_1}$ with angle $\l/\pi$. 
Since $|I^c_{t_1}| \to 0$, the diameter of the domain $\m D \setminus \widetilde V_{t_1}$, which contains $\widetilde \eta^{(2)}_{[0,t_2]} = g_{t_1}^{(1)} \left(\eta^{(2)}_{[0,t_2]}\right)$, converges to $0$. 

Recall that the map $g_{\mbt,1}$ maps out the curve $\widetilde \eta^{(2)}_{[0,t_2]}=g_{t_1}^{(1)}\left(\eta_{[0,t_2]}^{(2)}\right)$.
If we conjugate $g_{\mbt,1}$ by the rotation $\rho_{t_1} :\m D\to \m D$ such that the image of the mid-point of $\a_{t_1}$ under $\rho_{t_1}$ lies in $(0,1)$ (so that the white region is symmetric with respect to the real line and $\rho_{t_1} (\a_{t_1})$ shrinks to the point $1 \in \partial \m D$), the map $\widetilde g_{\mbt,1} := \rho_{t_1} \circ g_{\mbt,1} \circ \rho_{t_1}^{-1}$ converges in Carath\'eodory topology (namely, uniformly on compact subsets) to the identity map in $\m D$ as $t_1 \to \infty$. If we Schwarz-reflect $g_{\mbt,1}$ along $\partial \m D \setminus g_{t_1}^{(1)} (\exp(\ii \t_2))$, we see that the convergence also extends to the boundary; more precisely, we obtain that $\widetilde g_{\mbt,1}$ converges uniformly on all compact subsets of $\ad {\m D} \setminus \{1\}$ (and the map is well-defined on every such compact subset for large enough $t_1$).  Moreover, this convergence is uniform in $\eta^{(1)}\subset U_1$ and  $\eta^{(2)}\subset U_2$ (owing to the fact that $\tilde\eta^{(2)}_{[0,t_2]}$ is contained in $\m D\setminus \widetilde V_{t_1}$ for all such choices of $\eta^{(1)}$ and $\eta^{(2)}$). Similarly, $\widetilde g_{\mbt,1}'(0)$ converges to $1$ uniformly in $\eta^{(1)}\subset U_1$ and  $\eta^{(2)}\subset U_2$. Thus, we obtain that $R_{\mbt} \to 1$ almost surely, which completes the proof.
\end{proof}

\subsection{Chordal SLE weighted by conformal radius}
\label{subsec::chordalSLE_CR}
In this section, we show that the partition functions $\mc Z_\a$ correspond to chordal SLE weighted by the conformal radius to the power $-\a$.  

For this, we first calculate the Laplace transform of the conformal radius of the complement of chordal SLE. Usually, chordal SLE is defined in the upper-half plane. It is more convenient here to describe it in the unit disc $\m D$ via a change of coordinates. 
Fix $\kappa\in (0,8)$ and $\theta_1<\theta_2<\theta_1+2\pi$. Suppose $\gamma$ is chordal $\SLE_{\kappa}$ in $(\m D; \ee^{\ii\theta_1}, \ee^{\ii\theta_2})$. 
We parameterize it by the capacity and define $g_t, \xi_t$ accordingly as in Section~\ref{subsec::radialLoewner}.  Denote by $T$ the first time $\gamma$ disconnects $\ee^{\ii\theta_2}$ from the origin. Chordal $\SLE_{\kappa}$ in $(\m D; \ee^{\ii\theta_1}, \ee^{\ii\theta_2})$, up to $T$, has the same law as radial $\SLE_{\kappa}(\kappa-6)$ starting from $\ee^{\ii\theta_1}$ with force point $\ee^{\ii\theta_2}$, up to the same time; see~\cite{SW05}. In other words, its driving function $\xi_t$ satisfies the following SDE: 
\begin{equation}\label{eqn::chordalSLE_radial}
    \begin{cases}
    \xi_0=\theta_1, V_0=\theta_2,\\
    \ud \xi_t=\sqrt{\kappa}\ud B_t+\dfrac{\kappa-6}{2}\cot\left((\xi_t-V_t)/2\right)\ud t,\\
    \ud V_t=\cot\left((V_t-\xi_t)/2\right)\ud t. 
    \end{cases}
\end{equation}
Note that the conformal radius $\CR(\m D\setminus\gamma)$ is the same as $\ee^{-T}$; thus its Laplace transform can be derived from the SDE~\eqref{eqn::chordalSLE_radial}.  

\begin{lem}\label{lem::CR_expectation}
Fix $\kappa\in (0,8)$ and $\theta_1<\theta_2<\theta_1+2\pi$. We write $\theta=\theta_2-\theta_1\in (0,2\pi)$. Suppose $\gamma$ is chordal $\SLE_{\kappa}$ in $(\m D; \ee^{\ii \theta_1}, \ee^{\ii\theta_2})$ and denote by $\E_{\theta}$ the expectation with respect to $\gamma$. 
Denote by $\CR(\m D\setminus\gamma)$ the conformal radius of $\m D\setminus\gamma$ seen from the origin. For $\alpha\in \mathbb{R}$, 
we define 
\begin{equation}\label{eqn::CR_expectation}
\Phi(\kappa, \alpha; u):=\E_{\theta}\left[\CR(\m D\setminus\gamma)^{-\alpha}\right],\quad \text{where }u=\left(\sin(\theta/4)\right)^2\in (0,1).
\end{equation}
Then $\Phi(\kappa, \alpha; u)$ is finite for $u\in (0,1)$ if and only if $\alpha<1-\kappa/8$. Moreover, when $\alpha<1-\kappa/8$, $\Phi(u)=\Phi(\kappa, \alpha; u)$ satisfies the following ODE
\begin{equation}\label{eqn::CR_ODE}
    u(1-u)\Phi''+\frac{3\kappa-8}{2\kappa}(1-2u)\Phi'+\frac{8\alpha}{\kappa}\Phi=0,
\end{equation}
and the symmetry
\begin{equation}\label{eqn::CR_sym}
    \Phi(u)=\Phi(1-u),\quad u\in(0,1).
\end{equation}
\end{lem}
\begin{proof}
The statement is basically proved in~\cite{SSW2009,PeteWuSLEexcursions}. For the reader's convenience, we summarize its proof here.
We parameterize $\gamma$ by capacity, then its driving function $\xi_t$ satisfies the SDE~\eqref{eqn::chordalSLE_radial}. 
We write $\theta_t=V_t-\xi_t$. The process $\theta_t$ satisfies the SDE:
\begin{equation}\label{eqn::SLE_radial_SDE}
  \ud \theta_t=\sqrt{\kappa}\ud B_t+\frac{\kappa-4}{2}\cot(\theta_t/2)\ud t.  
\end{equation}
The disconnection time $T$ is the first time $\theta_t$ hits $0$ or $2\pi$. 
Note that $\CR(\m D\setminus\gamma)=\ee^{-T}$ and $\ee^{\alpha T}$ has finite expectation if and only if $\alpha<1-\kappa/8$ because 
\begin{align}\label{eqn::CR_asy}
    \PP[T\ge t]= \mathrm{c}_{\kappa}\ee^{-t(1-\kappa/8)}\left(\sin(\theta/2)\right)^{8/\kappa-1}(1+O(\ee^{-vt})),\quad\text{as }t\to\infty, 
    \end{align}
    where $c_{\kappa}>0$ and $v>0$ are constants; see~\cite[Lemma~2.10]{LawlerWerner2013}. Thus, $\Phi(u)=\Phi(\kappa,\alpha;u)=\E_{\theta}\left[\CR(\m D\setminus\gamma)^{-\alpha}\right]$ is finite if and only if $\alpha<1-\kappa/8$. The symmetry~\eqref{eqn::CR_sym} is clear from the definition.
    
Let us show that $\Phi$ satisfies the ODE~\eqref{eqn::CR_ODE}. 
Taking conditional expectation, the process 
\begin{equation}\label{eqn::CR_mart}
\ee^{\alpha t}\Phi\left(\left(\sin(\theta_t/4)\right)^2\right)=\mathbb{E}_{\theta}\left[\CR(\m D\setminus \gamma)^{-\alpha}\cond \gamma_{[0,t]}\right]
\end{equation}
is a martingale. Note that $\Phi$ is locally bounded and measurable. Hence, $\Phi$ is a weak solution for~\eqref{eqn::CR_ODE} and
\[\left(\sin((\theta_2-\theta_1)/2)\right)^{-2h}\Phi\left(\left(\sin((\theta_2-\theta_1)/4)\right)^2\right)\]
is a weak solution to the radial BPZ equations (see~\cite[Theorem~2.2 and Proposition~3.2]{KV26}). As the operators in the radial BPZ equations are hypoelliptic (see Remark~\ref{rem::hypoelliptic}), weak solutions are strong solutions. Thus $\Phi$ is a $C^2$ solution to~\eqref{eqn::CR_ODE}.
\end{proof}

\begin{cor}
Fix $\kappa\in (0,8)$ and $\theta_1<\theta_2<\theta_1+2\pi$. We write $\theta=\theta_2-\theta_1\in (0,2\pi)$. Suppose $\gamma$ is chordal $\SLE_{\kappa}$ in $(\m D; \ee^{\ii \theta_1}, \ee^{\ii\theta_2})$ and denote by $\E_{\theta}$ the expectation with respect to $\gamma$. 
    Recall that $\mc Z_{\alpha}$ is defined in~\eqref{eqn::CR_pf} for $\alpha<1-\kappa/8$ and $\LG_{\mu}$ is defined in~\eqref{eqn::2SLEspiral_pf} for $\mu\in \m R$. Recall that $h=\frac{6-\kappa}{2\kappa}$ from~\eqref{eqn::parameters}. 
    Then we have
    \begin{equation}\label{eqn::Zalpha_cr}
        \mc Z_{\alpha}(\theta_1, \theta_2)=\left(\sin(\theta/2)\right)^{-2h}\frac{\E_{\theta}[\CR(\m D\setminus\gamma)^{-\alpha}]}{\E_{\pi}[\CR(\m D\setminus\gamma)^{-\alpha}]}, \quad\text{for }\alpha<1-\kappa/8.
    \end{equation}
    Moreover, we have
    \begin{equation}\label{eqn::CR_limit}
\frac{\E_{\theta}[\CR(\m D\setminus\gamma)^{-\alpha}]}{\E_{\pi}[\CR(\m D\setminus\gamma)^{-\alpha}]}=\left(\sin(\theta/2)\right)^{2h}\mc Z_{\alpha}(\theta_1, \theta_2)\to \left(\sin(\theta/2)\right)^{2h}\LG_0(\theta_1, \theta_2)
    \end{equation}
    as $\alpha\uparrow (1-\kappa/8)$.
\end{cor}
\begin{proof}
We write $u=(\sin(\theta/4))^2$. 
    Recall from~\eqref{eqn::CR_pf} that $\phi_{\alpha}$ is the unique solution to~\eqref{eqn::Euler_initial}. Comparing this with~\eqref{eqn::CR_ODE} and~\eqref{eqn::CR_sym}, one sees that 
\[\phi_{\alpha}(\cdot)=\frac{\Phi(\kappa, \alpha; \cdot)}{\Phi(\kappa,\alpha; 1/2)}.\]
Thus 
\[\mc Z_{\alpha}(\theta_1, \theta_2):=\left(\sin(\theta/2)\right)^{-2h}\phi_{\alpha}(u)=\left(\sin(\theta/2)\right)^{-2h}\frac{\E_{\theta}[\CR(\m D\setminus\gamma)^{-\alpha}]}{\E_{\pi}[\CR(\m D\setminus\gamma)^{-\alpha}]},\]
as desired in~\eqref{eqn::Zalpha_cr}. 
Moreover, we have
\begin{align*}
\frac{\E_{\theta}[\CR(\m D\setminus\gamma)^{-\alpha}]}{\E_{\pi}[\CR(\m D\setminus\gamma)^{-\alpha}]} & =\left(\sin(\theta/2)\right)^{2h}\mc Z_{\alpha}(\theta_1, \theta_2)=\phi_{\alpha}(u) \end{align*}
which converges to $\phi_{\alpha_0}(u) = \left(\sin(\theta/2)\right)^{2h}\LG_0(\theta_1, \theta_2)$ as $\a \to \a_0 = 1 -\k/8$ by Lemma~\ref{lem::Euler_initial}. 
This gives~\eqref{eqn::CR_limit}. 
\end{proof}

\begin{cor}\label{cor:CR_driving}
Fix $\kappa\in (0,8)$ and $\alpha<1-\kappa/8$.
Denote by $\PP$ the law of $\gamma$, a chordal $\SLE_{\kappa}$ in $(\m D; \ee^{\ii\theta_1}, \ee^{\ii\theta_2})$ with $\theta_1<\theta_2<\theta_1+2\pi$.
Let $\eta^{(1)}$ be $\gamma$ and let $\eta^{(2)}$ be the time-reversal of $\gamma$ and still denote by $\PP$ the induced law on $(\eta^{(1)}, \eta^{(2)})$. 
We define $\LZ_{\alpha}$ as in~\eqref{eqn::CR_pf}. 
 Denote by $\PP(\LZ_\alpha)$ the probability measure obtained by weighting $\PP$ by $\CR(\m D\setminus\gamma)^{-\alpha}$. Then, under $\PP(\LZ_\alpha)$, the family of local laws obtained by restricting the pair $(\eta^{(1)}, \eta^{(2)})$ in disjoint neighborhoods satisfies~\textbf{(CI)}, \textbf{(DMP)}, \textbf{(MARG)} and~\textbf{(INT)} with 
 \begin{equation*}
     b_j=\kappa\partial_j \log \LZ_{\alpha}, \quad j=1,2.
 \end{equation*}
More precisely, the driving function of $\eta^{(1)}$ satisfies the following SDE, up to the first time $\ee^{\ii\theta_2}$ is disconnected from the origin: 
\begin{equation}\label{eqn::SLE_CR_SDE1}
    \begin{cases}
    \xi_0^{(1)}=\theta_1, V_0^{(2)}=\theta_2,\\
    \ud \xi_t^{(1)}=\sqrt{\kappa}\ud \widetilde{B}_t^{(1)}+\kappa\partial_1(\log \LZ_{\alpha})(\xi_t^{(1)}, V_t^{(2)})\ud t,\\
    \ud V_t^{(2)}=\cot\left((V_t^{(2)}-\xi_t^{(1)})/2\right)\ud t,
    \end{cases}
\end{equation}
where $\widetilde{B}^{(1)}$ is Brownian motion under $\PP(\LZ_\alpha)$. Similarly, the driving function of $\eta^{(2)}$ satisfies the following SDE, up to the first time $\ee^{\ii\theta_1}$ is disconnected from the origin: 
\begin{equation}\label{eqn::SLE_CR_SDE2}
    \begin{cases}
    V_0^{(1)}=\theta_1, \xi_0^{(2)}=\theta_2,\\
    \ud \xi_t^{(2)}=\sqrt{\kappa}\ud \widetilde{B}_t^{(2)}+\kappa\partial_2(\log \LZ_{\alpha})(V_t^{(1)}, \xi_t^{(2)})\ud t,\\
    \ud V_t^{(1)}=\cot\left((V_t^{(1)}-\xi_t^{(2)})/2\right)\ud t,
    \end{cases}
\end{equation}
where $\widetilde{B}^{(2)}$ is Brownian motion under $\PP(\LZ_\alpha)$. 
\end{cor}

\begin{proof}
The fact that the local laws obtained by restricting the pair $(\eta^{(1)}, \eta^{(2)})\sim\PP(\LZ_{\alpha})$ in disjoint neighborhoods satisfy~\textbf{(CI)}, \textbf{(DMP)}, \textbf{(MARG)} and~\textbf{(INT)} follows from the reversibility of SLE (proved in~\cite{Zhan, IG2, IG3}): suppose $\gamma$ is chordal $\SLE_{\kappa}$ in $\m D$ from $\ee^{\ii\theta_1}$ to $\ee^{\ii\theta_2}$ with $\kappa\in (0,8)$, then the time-reversal of $\gamma$ has the same law as that of a chordal $\SLE_{\kappa}$ in $\m D$ from $\ee^{\ii\theta_2}$ to $\ee^{\ii\theta_1}$. In particular, \textbf{(DMP)} is shown by applying the domain Markov property of chordal SLE twice: first to an initial segment of the chord, and then to a final segment of the chord, this time referencing reversibility.

It remains to check~\eqref{eqn::SLE_CR_SDE1} and~\eqref{eqn::SLE_CR_SDE2}. As the pair $(\eta^{(1)}, \eta^{(2)})$ is interchangeable, it suffices to check~\eqref{eqn::SLE_CR_SDE1}. 

Write $\Phi(\cdot)=\Phi(\kappa, \alpha; \cdot)$ as in~\eqref{eqn::CR_expectation}.
Using the same notation as in the proof of Lemma~\ref{lem::CR_expectation}, we denote the martingale in~\eqref{eqn::CR_mart} by
\begin{equation*}
    M_t(\alpha):=\ee^{\alpha t}\Phi\left(\left(\sin(\theta_t/4)\right)^2\right).
\end{equation*}
Then $\PP(\LZ_\alpha)$ is the same as $\PP$ tilted by $M_t(\alpha)$. Recall from~\eqref{eqn::SLE_radial_SDE}, under $\PP$, we have
 \[\ud \theta_t=\sqrt{\kappa}\ud B_t+\frac{\kappa-4}{2}\cot(\theta_t/2)\ud t.  \]
Thus, under $\PP$, we have
\[\frac{\ud M_t(\alpha)}{M_t(\alpha)}=\frac{\sqrt{\kappa}}{4}\frac
{\Phi'}{\Phi}\sin(\theta_t/2)\ud B_t.\]
Girsanov's theorem tells that $\widetilde{B}_t$ defined by  $\widetilde{B}_0=0$ and
\[\ud\widetilde{B}_t=\ud B_t-\frac{\sqrt{\kappa}}{4}\frac{\Phi'}{\Phi}\sin(\theta_t/2)\ud t\]
is Brownian motion under $\PP(\LZ_\alpha)$. 
Combining this with~\eqref{eqn::Zalpha_cr}, we obtain~\eqref{eqn::SLE_CR_SDE1}. 
\end{proof}

Now, we are ready to complete the proof of Theorem~\ref{thm:main}. 
\begin{proof}[Proof of Theorem~\ref{thm:main}]
    The conclusion follows from Theorem~\ref{thm:classify_int}, Corollary~\ref{cor::2SLEspiral_marginal} and Corollary~\ref{cor:CR_driving}. 
\end{proof}

\subsection{Commuting SLEs without interchangeability}

We may also classify the commuting SLEs without interchangeability. 
Let us first give an example.
\begin{remark}\label{rem::cr_left}
We use the same notation as in Lemma~\ref{lem::CR_expectation}. Recall that $\theta=\theta_2-\theta_1$ and $u=(\sin(\theta/4))^2$.
For $\kappa\in (0,8)$ and $\alpha<1-\kappa/8$, 
we consider
\begin{align*}
    \Phi^L(\kappa, \alpha; u):=&\mathbb{E}_{\theta}\left[\CR(\m D\setminus\gamma)^{-\alpha}1_{\{0\text{ is to the left of }\gamma\}}\right];\\
    \Phi^R(\kappa, \alpha; u):=&\mathbb{E}_{\theta}\left[\CR(\m D\setminus\gamma)^{-\alpha}1_{\{0\text{ is to the right of }\gamma\}}\right].
\end{align*}
Then $\Phi^L(\cdot)=\Phi^L(\kappa, \alpha; \cdot)$ and $\Phi^R(\cdot)=\Phi^R(\kappa, \alpha; \cdot)$ also satisfy the ODE~\eqref{eqn::CR_ODE}, but they do not enjoy the symmetry~\eqref{eqn::CR_sym} anymore. This gives an example of locally commuting 2-radial SLE without interchangeability.

Moreover, we set 
\begin{align*}
    \LZ^L(\theta_1, \theta_2)=\left(\sin(\theta/2)\right)^{-2h}\Phi^L(\kappa, \alpha; \left(\sin(\theta/4)\right)^2);\\
    \LZ^R(\theta_1, \theta_2)=\left(\sin(\theta/2)\right)^{-2h}\Phi^R(\kappa, \alpha; \left(\sin(\theta/4)\right)^2).
\end{align*}
Then both $\LZ^L$ and $\LZ^R$ satisfy~\eqref{eqn::BPZ1} and~\eqref{eqn::BPZ2} with 
\[\aleph=\frac{(6-\kappa)(\kappa-2)}{8\kappa}-\alpha.\]
 
\end{remark}

\begin{prop}\label{prop:remove_int}
    If one removes the interchangeability condition in Theorem~\ref{thm:classify_int}, one obtains partition functions, $\mc Z$, of the form:
    \begin{enumerate}
        \item $\mc Z=\mc G_\mu$ for some $\mu\in\m R$, where $\mc G_\mu$ is defined as in \eqref{eqn::2SLEspiral_pf}.
        \item $\mc Z=\mc Z_{\a,\b}$, for $\a<1-\k/8$ and $\b\in[0,1]$, where
        \begin{equation*}
            \mc Z_{\a,\b}(\t_1,\t_2)=(\sin(\t_{21}/2))^{-2h}(\b\Phi^L((\sin(\t_{21}/4))^2)+(1-\b)\Phi^R((\sin(\t_{21}/4))^2)).
        \end{equation*}
        \end{enumerate}
        \end{prop}
      
    The locally commuting 2-radial $\SLE_\kappa$ corresponding to the second case above is obtained analogously to Corollary~\ref{cor:CR_driving}. That is, one weights the law of a chordal $\SLE_\kappa$ in $(\m D;\ee^{\ii\t_1},\ee^{\ii\t_2})$, denoted by $\g$, by 
    \begin{equation*}
        \CR(\m D\setminus\g)^{-\a}(\b 1_{\{0 \text{ is to the left of }\g\}}+(1-\b)1_{\{0 \text{ is to the right of }\g\}}),
    \end{equation*}
    then lets $\eta^{(1)}$ be $\g$ and $\eta^{(2)}$ be the time-reversal of $\g$, and finally restricts the law of $(\eta^{(1)},\eta^{(2)})$ to disjoint neighborhoods. One sees that the obtained family of local laws satisfies~\textbf{(CI)}, \textbf{(DMP)}, and \textbf{(MARG)} by using the reversibility of SLE: this makes $\eta^{(2)}$ a chordal SLE in $(\m D;\ee^{\ii\t_2},\ee^{\ii\t_1})$ weighted by
    \begin{equation*}
        \CR(\m D\setminus\eta^{(2)})^{-\a}(\b 1_{\{0 \text{ is to the right of }\eta^{(2)}\}}+(1-\b)1_{\{0 \text{ is to the left of }\eta^{(2)}\}}).
    \end{equation*}

      \begin{proof}[Proof of Proposition~\ref{prop:remove_int}]
        If, in the proof of Theorem \ref{thm:classify_int}, we assume that $\mu=0$ without assuming interchangeability, then $\mc Z$ is a positive solution of \eqref{eqn::rot_BPZ}. By changing variables by \eqref{eqn::thetau}, we find that $\mc Z$ corresponds to a positive solution $\phi$ of \eqref{eqn::CR_ODE}. Lemma \ref{lem::Euler_positive} gives the three scenarios:
        \begin{itemize}
            \item if $\a<1-\k/8$, up to a multiplicative constant, 
            \begin{equation}
            \phi=\b\Phi^L+(1-\b)\Phi^R,\quad \b\in[0,1];
        \end{equation}
        \item if $\a=1-\k/8$, there is only one positive solution, corresponding to $\mc Z=\mc G_0$; 
        \item if $\alpha>1-\k/8$, there are no positive solutions.
        \end{itemize}
       \end{proof}

\section{Semiclassical limits of commutation relation}\label{sec:semiclassical}

\subsection{Commutation relation when $\k = 0$}

We now consider the commutation relation when $\k = 0$.  In the case of multichordal SLE, a similar semiclassical limit of partition functions was considered in \cite{peltola_wang} and \cite{alberts2020pole}.

It is not hard to see that the infinitesimal commutation relation (Proposition~\ref{prop:radial_comm}) also holds when $\k = 0$. However, we need to derive the BPZ equation and classify the partition functions differently, which we summarize in the following proposition.

\begin{prop}\label{prop:part_wo_int_0}
We consider an interchangeable and locally commuting $2$-radial $\SLE_0$. Let $b_1, b_2 : S^1 \times S^1 \setminus \D \to \m R$ be $C^2$ functions as in the condition  \textbf{(MARG)}. Then \eqref{eq:comm_cond} and~\eqref{eq:rot} imply that 
there exist $\mc U: \{(\t_1, \t_2) \in \m R^2 \,|\, \t_1 < \t_2 < \t_1 +2 \pi\} \to \m R$ and a constant $F$ such that 
\[b_j = \partial_j \mc U, \quad j = 1,2 \]
and
\begin{align}\label{eqn::int_0}
\begin{split}
      (\partial_2 \mc U )^2 + 2 \cot(\t_{12}/2) \partial_1 \mc U - \frac{3}{\left(\sin(\t_{12}/2)\right)^2} & = F,\\
   (\partial_1 \mc U )^2 + 2 \cot(\t_{21}/2) \partial_2 \mc U - \frac{3}{\left(\sin(\t_{21}/2)\right)^2} & = F,
\end{split}
\end{align}
where $\theta_{21}=\theta_2-\theta_1=-\theta_{12}$. 
The only solutions are 
\begin{equation}\label{eq:sol_1_zero}
\mc U(\t_1, \t_2) = \mc U_\mu (\t_1, \t_2) := 2 \log \sin (\t_{21}/2) + \mu (\t_1 + \t_2)
\end{equation}
for some $\mu \in \m R$ or
\begin{equation}\label{eq:sol_2_zero}
\mc U(\t_1, \t_2) =  - 6 \log \sin (\t_{21}/2)
\end{equation}
up to an additive constant.
\end{prop}
\begin{remark}
    Note that we require that $b_1,\ b_2$ be defined on all of $S^1\times S^1\setminus\Delta$, which is consistent with the definition of locally commuting 2-radial SLE in Section~\ref{sec:axioms}.
\end{remark}
\begin{remark}
    It follows immediately from the proof that, if one removes the interchangeability assumption from Proposition \ref{prop:part_wo_int_0}, then the only additional solutions $\mc U$ are of the form $$\mc U(\t_1,\t_2) = \mc U_{\lambda,+}(\t_1,\t_2) := -2\log\sin(\t_{21}/2)+\int_{\t_{21}}^\pi\sqrt{2\lambda+4\cot^2(u/2)}\ud u,$$
    $$\mc U(\t_1,\t_2) = \mc U_{\lambda,-}(\t_1,\t_2) := -2\log\sin(\t_{21}/2)-\int_{\t_{21}}^\pi\sqrt{2\lambda+4\cot^2(u/2)}\ud u,$$
    for $\lambda>0$, up to an additive constant.\label{rmk:wo_int}
\end{remark}
\begin{proof}
    The proof of Proposition~\ref{prop:part_wo_int} up to \eqref{eqn::comm_b2} holds verbatim when $\k = 0$. By conformal invariance $b_j$ is invariant under rotation $(\t_1,\t_2)\mapsto(\t_1+a,\t_2+a)$. Thus, $\hat b_j:S^1\setminus\{0\}\to\m R$ defined by $\hat b_j(\t)=b_j(\t,0)$, $j=1,2$, satisfies 
    $$\partial_1 b_j(\t_1,\t_2)=\partial_1 \hat b_j'(\t_{12})=-\partial_2 b_j(\t_1,\t_2).$$
    Equations~\eqref{eqn::comm_b1} and \eqref{eqn::comm_b2} give
    \begin{align}
      &(\cot(\t/2)-\hat b_2(\t)) \hat b_1'(\t)-\frac{\hat b_1(\t)}{2\sin^2(\t/2)}+\frac{3\cot(\t/2)}{2\sin^2(\t/2)}= 0 \label{eqn::comm_hatb1}\\
     & (\cot(\t/2)+\hat b_1(\t))\hat b_2'(\t)-\frac{\hat b_2(\t)}{2\sin^2(\t/2)}-\frac{3\cot(\t/2)}{2\sin^2(\t/2)}= 0.\label{eqn::comm_hatb2}
    \end{align}
    Taking the difference of these equations yields
    $$\frac{\ud}{\ud \t}((\cot(\t/2)+\hat b_1(\t))(\cot(\t/2)-\hat b_2(\t))+\frac{4\cot(\t/2)}{\sin^2(\t/2)}=0,$$
    so that
    \begin{equation}(\cot(\t/2)+\hat b_1(\t))(\cot(\t/2)-\hat b_2(\t))=\frac{4}{\sin^2(\t/2)}+C,\label{eq:b1b2rel}\end{equation}
    for some $C\in\m R$. Thus, multiplying~\eqref{eqn::comm_hatb1} by $2\sin^2(\t/2)(\hat b_1(\t)+\cot(\t/2))$ yields
    \begin{align*}0=&(8+2C\sin^2(\t/2))\hat b_1'(\t)-(\hat b_1(\t)+\cot(\t/2))(\hat b_1(\t)-3\cot(\t/2)) \\
    =& (8+2C\sin^2(\t/2))f'(\t)-((f(\t))^2+C+4)\end{align*}
    where $f$ is defined by $f(\t)=\hat b_1(\t)-\cot(\t/2)$. Thus, $f$ satisfies the separable differential equation
    \begin{equation}
        (8+2C\sin^2(\t/2))f'(\t)=(f(\t))^2+C+4.\label{eqn::separable}
    \end{equation}
    Similarly, we see using~\eqref{eqn::comm_hatb2} that $g$ defined by $g(\t)=\hat b_2(2\pi-\t)-\cot(\t/2)$ also satisfies~\eqref{eqn::separable} and by~\eqref{eq:b1b2rel}
    \begin{equation}
        (2\cot(\t/2)+f(\t))(2\cot(\t/2)-g(2\pi-\t))=\frac{4}{\sin^2(\t/2)}+C.\label{eq:fgrel}
    \end{equation}
    The solutions of~\eqref{eqn::separable} defined on $(0,2\pi)$ are studied in Lemma~\ref{lemma:solseparable}.\par
    \noindent\textbf{Case 1:} Suppose $C+4>0$. Then both $f$ and $g$ are of the form~\eqref{eq:tansol} for some constants $D=D_f,D_g\in[-\pi/4,\pi/4].$ From~\eqref{eq:fgrel} it follows that 
    $$-(C+4)=f(\pi)g(\pi)=(C+4)\tan(D_f)\tan(D_g)$$
    and thus we must have $-D_g=D_f\in\{\pi/4,-\pi/4\}.$ Using the half-angle formula for the tangent function one can see that these two solutions can be rewritten as 
    $$\begin{cases}f(\t)=\sqrt{(C+4)+4\cot^2(\t/2)}-\cot(\t/2)\\
    g(\t)=-\sqrt{(C+4)+4\cot^2(\t/2)}-\cot(\t/2)
    \end{cases}$$
    $$\begin{cases}f(\t)=-\sqrt{(C+4)+4\cot^2(\t/2)}-\cot(\t/2)\\
    g(\t)=\sqrt{(C+4)+4\cot^2(\t/2)}-\cot(\t/2).
    \end{cases}$$
    \noindent\textbf{Case 2:} If $C+4=0$ we have only the solutions 
    $$\begin{cases}f(\t)=0\\
    g(\t)=0\end{cases}\qquad\text{and}\qquad \begin{cases}f(\t)=-4\cot(\t/2)\\
    g(\t)=-4\cot(\t/2).\end{cases}$$ 
    Indeed, if 
    $$f(\t)=\begin{cases}
            (D-\frac{1}{4}\tan(\t/2))^{-1}& \t\in(0,\pi)\\
            (E-\frac{1}{4}\tan(\t/2))^{-1}& \t\in(\pi,2\pi)\\
    \end{cases}
    $$
    for some $D\leq 0$ and $E\geq 0$ then solving~\eqref{eq:fgrel} for $g$ yields
    $$g(\t)=\begin{cases}
            (E-\frac{1}{4}\tan(\t/2))^{-1}& \t\in(0,\pi)\\
            (D-\frac{1}{4}\tan(\t/2))^{-1}& \t\in(\pi,2\pi).\\
    \end{cases}
    $$
    Therefore, we must have $D=E=0$ (otherwise $g$ is not well-defined on all of $(0,2\pi)$).
    \par
    \noindent{\textbf{Case 3:}} For $C+4<0$ we have the solutions $$\begin{cases}f(\t)=\sqrt{-C-4}\\
    g(\t)=\sqrt{-C-4}\end{cases}\quad\text{and}\quad \begin{cases}f(\t)=-\sqrt{-C-4}\\
    g(\t)=-\sqrt{-C-4}.\end{cases}$$ Here we have again used that $f(\pi)g(\pi)=-(C+4).$\par
    From the solutions $(f,g)$ above we obtain the corresponding solutions $(\hat b_1,\hat b_2)$ of \eqref{eqn::comm_hatb1} and~\eqref{eqn::comm_hatb2} and $(b_1,b_2)$ of~\eqref{eqn::comm_b1} and ~\eqref{eqn::comm_b2} for $\k=0$. We note that these solutions satisfy $\partial_1 b_2 = \partial_2 b_1$ (or equivalently $\hat b_1' = -\hat b_2'$, or $f'(\t)=g'(2\pi-\t)$) so there exists a function $$\mc U : \{(\t_1, \t_2) \,|\,\t_1 < \t_2 < \t_1 + 2\pi\} \to \m R$$ such that $b_1 = \partial_1 \mc U$ and $b_2 = \partial_2 \mc U$. Up to an additive constant we have that 
    $$\mc U(\t_1,\t_2) = \mc U_{\lambda,+}(\t_1,\t_2) := -2\log\sin(\t_{21}/2)+\int_{\t_{21}}^\pi\sqrt{2\lambda+4\cot^2(u/2)}\ud u,$$
    $$\mc U(\t_1,\t_2) = \mc U_{\lambda,-}(\t_1,\t_2) := -2\log\sin(\t_{21}/2)-\int_{\t_{21}}^\pi\sqrt{2\lambda+4\cot^2(u/2)}\ud u,$$
    for $2\lambda=C+4>0$,
    $$\mc U(\t_1,\t_2)=-6\log\sin(\t_{21}/2),$$
    and
    $$\mc U(\t_1,\t_2) = \mc U_\mu(\t_1,\t_2) = 2\log\sin(\t_{21}/2)+\mu(\t_1+\t_2)$$
    for $\mu=\pm\sqrt{-C-4}\in\m R$. It is easily verified that each of these $\mc U$ satisfies~\eqref{eqn::int_0} for some constant $F$. We note that $\mc U_{\l,\pm}$ does not satisfy the interchangeability condition but that the others do.
\end{proof}

We have defined two-sided radial $\SLE_\k$ with spiraling rate $\mu$ for $\k \in (0,8)$ in Definition~\ref{def::2SLEspiral}, and we now extend the definition to the case $\k = 0$.

\begin{df}\label{def:mu_zero}
We use the same notation as in Figure~\ref{fig::gtcommutation}, and we say that a deterministic pair of continuous simple curves $(\eta^{(1)}, \eta^{(2)})$ is the two-sided radial $\SLE_0$ with spiraling rate $\mu$ in  $(\m D; \ee^{\ii \t_1}, \ee^{\ii \t_2} ;0)$ if we have for all $\mbt = (t_1, t_2)$,
   \begin{align*}
   \begin{cases}
   \t_{(0,0)}^{(1)}=\theta_1,\quad \t_{(0,0)}^{(2)}=\theta_2,\\
            \ud \theta_{\mbt}^{(1)}=\mu\phi_{\mbt, 1}'\left(\xi_{t_1}^{(1)}\right)^2\ud t_1 +\cot\left((\theta_{\mbt}^{(1)}-\theta_{\mbt}^{(2)})/2\right)
       \left(\phi_{\mbt, 1}'\left(\xi_{t_1}^{(1)}\right)^2\ud t_1+\phi_{\mbt, 2}'\left(\xi_{t_2}^{(2)}\right)^2\ud t_2\right),\\
      \ud \theta_{\mbt}^{(2)} = \mu\phi_{\mbt, 2}'\left(\xi_{t_2}^{(2)}\right)^2\ud t_2 +\cot\left((\theta_{\mbt}^{(2)}-\theta_{\mbt}^{(1)})/2\right)
       \left(\phi_{\mbt, 1}'\left(\xi_{t_1}^{(1)}\right)^2\ud t_1+\phi_{\mbt, 2}'\left(\xi_{t_2}^{(2)}\right)^2\ud t_2\right).
   \end{cases}
   \end{align*}
\end{df}

To justify such a pair exists, we note that when $t_2 \equiv 0$, we write $\theta_{\mbt}^{(1)}=\xi_{t_1}^{(1)}$ and $\theta_{\mbt}^{(2)}=V_{t_1}^{(2)}$, then 
   \begin{align}\label{eq:zero_marginal}
   \begin{cases}
   \xi_0^{(1)}=\theta_1, \quad V_0^{(2)}=\theta_2,\\
   \ud   \xi_{t_1}^{(1)} = \mu \ud t_1+ \cot \left( (\xi_{t_1}^{(1)} - V_{t_1}^{(2)})/2\right)\ud t_1 =  \partial_1 \mc U_\mu (\xi_{t_1}^{(1)}, V_{t_1}^{(2)})\ud t_1 ,\\
       \ud  V_{t_1}^{(2)}  = \cot \left( (V_{t_1}^{(2)}- \xi_{t_1}^{(1)})/2\right)\ud t_1,
   \end{cases}
   \end{align}
   which shows $\eta^{(1)}$ is the radial SLE$_0^\mu(2)$ with force point $\ee^{\ii \t_2}$.  Similarly, $\eta^{(2)}$ is the radial SLE$_0^\mu(2)$ with force point $\ee^{\ii \t_1}$. As in Corollary~\ref{cor::2SLEspiral_marginal}, the system of equations in Definition~\ref{def:mu_zero} is simply the two-time version of \eqref{eq:zero_marginal} starting from $\mbt$ after capacity-reparametrization.
   
Since $\mc U_\mu$ satisfies the commutation relation as we showed in Proposition~\ref{prop:part_wo_int_0}, we know that $(\eta^{(1)}, \eta^{(2)})$ gives a pair in Definition~\ref{def:mu_zero} if we show that $\t_\mbt^{(2)} \notin \{\t_\mbt^{(1)}, \t_\mbt^{(1)} + 2\pi\}$ for all $\mbt$. Indeed, the map $t_2 \mapsto \t_{(0,t_2)}^{(2)}-\t_{(0,t_2)}^{(1)}$ is repulsive away from $\{0,2\pi\}$, so is well-defined and for all $t_2 \ge 0$. Similarly, for fixed $t_2$, the function $t_1 \mapsto \t_{(t_1, t_2)}^{(2)}-\t_{(t_1, t_2)}^{(1)}$ is repulsive away from $\{0,2\pi\}$, so $\t_\mbt^{(2)} \notin \{\t_\mbt^{(1)}, \t_\mbt^{(1)} + 2\pi\}$ for all $t_1,t_2 \ge 0$.

\begin{cor} \label{cor:0_commuting_classified}
The only interchangeable and locally commuting $2$-radial $\SLE_0$ are the two-sided radial $\SLE_0$ with spiraling rate $\mu \in \m R$ and the chordal $\SLE_0$. 
\end{cor}
\begin{proof}
When $\mc U = \mc U_\mu$, the corresponding commuting $\SLE_0$ is the two-sided radial $\SLE_0$ (with spiraling rate $\mu$). 

    When $\mc U (\t_1, \t_2) =  - 6 \log \sin (\t_{21}/2)$, the corresponding commuting $\SLE_0$ is the chordal $\SLE_0$ in $(\m D; \ee^{\ii \t_1}, \ee^{\ii \t_2})$ and its time-reversal.
\end{proof}

   \begin{remark}
         We note that when $\mc U (\t_1, \t_2) = \mc U_0(\t_1, \t_2) = 2 \log \sin (\t_{21}/2)$, the corresponding commuting $\SLE_0$ is the two-sided radial $\SLE_0$ (with zero spiraling rate). This is also the geodesic pair in $\mf X(\m D; \ee^{\ii \t_1}, \ee^{\ii \t_2}; 0)$ studied in \cite{W1,MRW1,mesikepp2022deterministic}, or in other words, the concatenation of $\eta^{(1)}$ and $\eta^{(2)}$ is the chord minimizing the chordal Loewner energy among all chords in $(\m D; \ee^{\ii \t_1}, \ee^{\ii \t_2})$ passing through $0$, and can therefore be viewed as the chordal $\SLE_0$ ``conditioned'' to pass through $0$.

         We remark that, unlike the $\k >0$ case, we do not have the second one-parameter family given by weighting by the conformal radius as described in Section~\ref{subsec::chordalSLE_CR}. 
   \end{remark}

\subsection{Semiclassical limits of partition functions}
To understand the reason why the second one-parameter family given by $\mc Z_\a$ disappears when $\k =0$, we now take a closer look at the semiclassical limit ($\k \to 0$) of $\LG_\mu$ and $\mc Z_\a$.

For $\mu\in\m R$, recall that $\LG_{\mu}$ is defined in~\eqref{eqn::2SLEspiral_pf}: 
\begin{equation*}
    \LG_{\mu}(\theta_1, \theta_2)=\left(\sin\left((\theta_{2}-\theta_1)/2\right)\right)^{2/\k} \exp\left(\frac{\mu}{\kappa}(\theta_1+\theta_2)\right).
\end{equation*}
\begin{lem}\label{lem::semiclassical_Gmu}
Fix $\mu\in\m R$ and $0<\theta_1<\theta_2<\theta_1+2\pi$. We have~\eqref{eqn::classical_mu}: 
\begin{equation*}
    \lim_{\kappa\to 0}\kappa\log\LG_{\mu}(\theta_1, \theta_2)=2\log\sin\left((\theta_{2}-\theta_1)/2\right)+\mu(\theta_1+\theta_2)
\end{equation*}
which is the solution~\eqref{eq:sol_1_zero} in Proposition~\ref{prop:part_wo_int_0}.
\end{lem}
\begin{proof}
    This is immediate from the expression of $\LG_\mu$.
\end{proof}

Lemma~\ref{lem::semiclassical_Gmu} gives the first part of Proposition~\ref{prop::semiclassical_pf}. We will prove its second part below.
For $\alpha<1-\kappa/8$, recall that $\mc Z_{\alpha}$ is defined in~\eqref{eqn::CR_pf}: 
\[ \mc Z_{\alpha}(\theta_1, \theta_2)=\left(\sin\left((\theta_{2}-\theta_1)/2\right)\right)^{(\kappa-6)/\kappa}\phi_{\alpha}\left(\left(\sin\left((\theta_{2}-\theta_1)/4\right)\right)^2\right),\]
where $\phi_{\alpha}$ is the unique solution to~\eqref{eqn::Euler_initial}. 
 Before we derive the semiclassical limit of $\mc Z_{\alpha}$, let us first address the chordal Loewner energy~\cite{peltola_wang}. Recall that $\mf X (\m D; \ee^{\ii \t_1}, \ee^{\ii \t_2})$ is the space of all continuous curves in $\m D$ connecting $\ee^{\ii \t_1}$ and $\ee^{\ii \t_2}$. 
\begin{lem}\label{lem::chordalLoewnerenergy}
Fix $\t_1<\t_2<\t_1+2\pi$. 
    Suppose $\gamma\in \mf X (\m D; \ee^{\ii \t_1}, \ee^{\ii \t_2})$. We parameterize it by the capacity and define $g_t, \xi_t$ accordingly as in Section~\ref{subsec::radialLoewner} and set $T=-\log\CR(\m D\setminus\gamma)$. Suppose $t\mapsto \xi_t$ is absolutely continuous and denote its derivative by $\dot\xi_t$. 
    We denote by $t\mapsto V_t$ the solution to 
    \begin{equation}\label{eqn::defV}
      \dot V_t=\cot\left((V_t-\xi_t)/2\right),\quad V_0=\theta_2.  
    \end{equation}
    Then the chordal Loewner energy $I(\gamma)$ can be written as 
    \begin{equation}\label{eqn::Loewnerenergy_def}
        I(\gamma)=\frac{1}{2}\int_0^T \left(\dot\xi_s-3\cot\left((V_s-\xi_s)/2\right)\right)^2\ud s. 
    \end{equation}
    Furthermore, the infimum of $I(\gamma)$ in $\mf X (\m D; \ee^{\ii \t_1}, \ee^{\ii \t_2})$ is zero and is attained by the hyperbolic geodesic. 
\end{lem}
\begin{proof}
    Eq.~\eqref{eqn::Loewnerenergy_def} is a standard calculation by changing coordinates. 
\end{proof}

\begin{lem}\label{lem::semiclassical_Zalpha}
\begin{itemize}
    \item If $\alpha<0$ and $\alpha=o(1/\kappa)$, then 
    \begin{equation} \label{eq:lim_sub}
         \lim_{\k \to 0} \k \log \mc Z_\a (\theta_1, \theta_2)= -6 \log \sin (\t_{21}/2)
    \end{equation}
    which is the solution~\eqref{eq:sol_2_zero} in Proposition~\ref{prop:part_wo_int_0}.
    \item If $\a \sim  -\l /\k$ for some $\l > 0$, then the limit $\lim_{\k \to 0} \k \log \mc Z_\a (\t_1, \t_2)$ exists and is equal to
    \begin{align}\label{eq:lim_U_lambda}
        \begin{split}
             \mc U^\l (\t_1, \t_2) & : =
              -6 \log \sin (\t_{21}/2) - \inf_{\g \in \mf X (\m D; \ee^{\ii \t_1}, \ee^{\ii \t_2})} \left(I(\g) - \l \log \CR(\m D \setminus \g)\right)\\
             & \quad +\inf_{\g \in \mf X (\m D; -1,1)} \left(I(\g) - \l \log \CR(\m D \setminus \g)\right),
        \end{split}
    \end{align}
    where $I(\g)$ is the chordal Loewner energy of $\g$ in $\mf X(\m D; \ee^{\ii \t_1}, \ee^{\ii \t_2})$ and the infimums are attained.
\end{itemize}
\end{lem}
We note that 
the last term in \eqref{eq:lim_U_lambda} is a constant such that $\mc U^\l(\t_1,\t_1+\pi) = 0$ as we have normalized $\mc Z_\a$ such that $\mc Z_\a (\t_1, \t_1 + \pi) = 1$.
\begin{proof}
For $\alpha<1-\kappa/8$, recall from Lemma~\ref{lem::CR_expectation} and~\eqref{eqn::Zalpha_cr} that $\mc Z_{\alpha}$ is defined as
    \begin{equation*}
    \mc Z_{\alpha}(\theta_1, \theta_2)=\left(\sin(\theta/2)\right)^{(\kappa-6)/\kappa} \frac{\m E_{\theta} [\CR (\m D \setminus \g)^{-\a}]}{\m E_{\pi}[\CR(\m D\setminus\gamma)^{-\alpha}]},
\end{equation*}
where $\m E_{\theta}$ is the expectation with respect to the law of $\g$ which is  a chordal $\SLE_\k$ in $(\m D; \ee^{\ii \t_1}, \ee^{\ii \t_2})$ and $\theta=\theta_2-\theta_1$.
When $\alpha\sim-\lambda/\kappa$ for $\lambda>0$, the result follows from the large deviation principle for chordal SLE as $\k \to 0\splus$ \cite{peltola_wang}. In fact, the Loewner energy is the large deviation rate function of chordal $\SLE_{0\splus}$ for the Hausdorff metric and $\g \mapsto -\log \CR(\m D \setminus \g)$ is a continuous function $\mf X (\m D; \ee^{\ii \t_1}, \ee^{\ii \t_2}) \to [0,\infty]$ with respect to the same metric. Indeed, the conformal radius is continuous with respect to the Carathéodory topology on the family of simply connected domains containing $0$.   For Jordan domains, the Carath\'eodory convergence is implied by the convergence for the Hausdorff metric of their boundary. See \cite[Sec. 2.5.1]{marden2016}. We apply this to the connected component of $\m D \setminus \gamma$ containing $0$. Varadhan's lemma \cite[Lem.~4.3.4 and~4.3.6]{DZ10} shows that if $\a \sim  -\l /\k$, then
    $$\lim_{\k \to 0} \k \log  \m E_{\theta} [\CR (\m D \setminus \g)^{-\a}] = - \inf_{\g \in \mf X (\m D; \ee^{\ii \t_1}, \ee^{\ii \t_2})} \left(I(\g) - \l \log \CR(\m D \setminus \g)\right),$$
    which proves the limit \eqref{eq:lim_U_lambda}. Since the large deviation rate function $I$ of chordal SLE$_{0\splus}$ is good, the infimum in \eqref{eq:lim_U_lambda} is attained.
Moreover, if $\alpha<0$ and $\alpha=o(1/\kappa)$ then we can conclude from the above that 
$$- \inf_{\g \in \mf X (\m D; \ee^{\ii \t_1}, \ee^{\ii \t_2})} \left(I(\g) - \l \log \CR(\m D \setminus \g)\right)\leq\lim_{\kappa\to 0}\k \log\m E_{\t}[\CR (\m D \setminus \g)^{-\a}] \leq 0$$
for all $\lambda>0$. By taking $\lambda\to 0$ we obtain
$$\lim_{\kappa\to 0}\k \log\m E_{\t}[\CR (\m D \setminus \g)^{-\a}]=0.$$
\end{proof}

\begin{prop}\label{prop:confRadMinimizer}
Suppose $\lambda>0$ and $\theta_1<\theta_2\le \theta_1+\pi$, and write $\theta=\theta_2-\theta_1$. We have
    \begin{equation}\label{eqn::energy_inf}
    \inf_{\g \in \mf X (\m D; \ee^{\ii \t_1}, \ee^{\ii \t_2})} \left(I(\g) - \l \log \CR(\m D \setminus \g)\right) = \int_0^{\t}\Big(\sqrt{2\lambda+4\cot^2(u/2)}-2\cot(u/2)\Big)\, \dd u.
    \end{equation}
    \begin{itemize}
        \item If $\theta_1<\theta_2<\theta_1+\pi$, the infimum in~\eqref{eqn::energy_inf} is attained for a unique curve $\g^*\in \mf X (\m D; \ee^{\ii \t_1}, \ee^{\ii \t_2})$ whose radial driving function $\xi^*_\cdot$, defined on $[0,T]$, satisfies 
    \begin{equation}\begin{cases}
    \xi_0^*=\theta_1, V_0^*=\theta_2,\\
    \dot \xi^*_t =  \cot((V^*_t-\xi^*_t)/2)+\sqrt{2\l+4\cot^2((V^*_t-\xi^*_t)/2)},\\
    \dot V^*_t =  \cot((V^*_t-\xi^*_t)/2),
    \end{cases}\label{eq:minimaldriver}\end{equation}
    and $\lim_{t\to T-}(V_t^*-\xi_t^*)=0$. 
    \item If instead $\t_2=\t_1+\pi,$ then the infimum in~\eqref{eqn::energy_inf} is attained for two curves, $\g^*$ and $\g^{**}$. One of the corresponding driving functions, say $\xi^*_\cdot$, satisfies \eqref{eq:minimaldriver} so that $\lim_{t\to T-}(V_t^*-\xi_t^*)=0$, while the other $\xi^{**}_\cdot,$ satisfies $\dot\xi^{**}_t=-\dot\xi^*_t$ for all $t\in[0,T)$, so that $\lim_{t\to T-}(V_t^{**}-\xi_t^{**})=2\pi$.
    \end{itemize}
\end{prop}

\begin{proof}
    We denote the right-hand side of~\eqref{eqn::energy_inf} by
    \begin{equation}
        H_\lambda(\t) = \int_0^{\t}\Big(\sqrt{2\lambda+4\cot^2(u/2)}-2\cot(u/2)\Big)\ud u,\label{eq:defHlambda}
    \end{equation}
    for $\t\in[0,\pi]$, and $H_\l(\t)=H_\l(2\pi-\t)$ for $\t\in(\pi,2\pi]$. 
    Fix $t\in(0,\infty)$ and let $\widetilde\gamma:[0,t]\to\m D\setminus\{0\}$ be a simple curve, parametrized by capacity, with an absolutely continuous driving function $t\mapsto\widetilde\xi_t$ and set $\widetilde V_t$ as in~\eqref{eqn::defV}. Define
    \begin{equation}\label{eqn::defJ}
        J^{\lambda}_{(\theta_1, \theta_2)}\left(\widetilde\gamma\right):=\frac{1}{2}\int_0^t\bigg(\dot{\widetilde\xi}_s-3\cot((\widetilde V_s-\widetilde\xi_s)/2)-H'_\lambda(\widetilde V_s-\widetilde \xi_s)\bigg)^2\ud s.
    \end{equation}
    Since $H_\lambda$ is not differentiable at $\theta=\pi$ (the left and right derivatives differ by a sign), the integrand on the right-hand side of Eq. \eqref{eqn::defJ} is not necessarily well-defined for $s$ such that $\widetilde V_s-\widetilde \xi_s=\pi$. However, this is not an issue: Let $E=\{s\in[0,t]:\widetilde V_s-\widetilde \xi_s=\pi\}$. In the interior of $E$ we have $\dot{\widetilde\xi}_s=0$. Hence, the integrand is well-defined and takes the value $2\lambda$ for $s\in E^\circ$. 
    Since $\partial E$ has measure zero the integrand need not be well-defined there.

    Let us connect $J^{\lambda}_{(\theta_1, \theta_2)}$ to the chordal Loewner energy $I$ in Lemma~\ref{lem::chordalLoewnerenergy}. We have the following two observations.
    \begin{itemize}
        \item Denote by $\hat\gamma$ the union of $\widetilde\gamma$ and the hyperbolic geodesic from $\widetilde\gamma_t$ to $\ee^{\ii\theta_2}$ in $\m D\setminus\widetilde\gamma$. Then Lemma~\ref{lem::chordalLoewnerenergy} gives the energy of $\hat{\gamma}$: 
        \begin{equation*}
        I(\hat{\g})=\frac{1}{2}\int_0^t\bigg(\dot{\widetilde\xi}_s - 3\cot((\widetilde V_s-\widetilde \xi_s)/2)\bigg)^2 \ud s.
    \end{equation*}
    \item For $H_{\lambda}$, we have
    \begin{equation*}
        H_\l(\widetilde V_t-\widetilde\xi_t)-H_\l(\t_2-\t_1)=\int_0^t H'_\l(\widetilde V_s-\widetilde \xi_s)\Big(\cot((\widetilde V_s-\widetilde \xi_s)/2)-\dot{\widetilde\xi}_s\Big)\ud s.
    \end{equation*}
    \end{itemize}
    Plugging these two observations into~\eqref{eqn::defJ}, we have
    \begin{equation}\label{eq:Jintegrated}
        J^{\lambda}_{(\theta_1, \theta_2)}\left(\widetilde \g\right)=I(\hat{\g})-\l\log\CR(\m D\setminus\widetilde\g)+H_\l(\widetilde V_t-\widetilde \xi_t)-H_\l(\t_2-\t_1).
    \end{equation}
    
    Suppose $\g:(0,T)\to\m D\setminus\{0\},$ with $\g(0+)=\ee^{\ii \t_1}$ and $\g(T-)=\ee^{\ii \t_2},$ has finite chordal Loewner energy. Then the associated radial driving function, in the capacity parametrization, $\xi_\cdot$, is absolutely continuous on $[0,t]$ for all $t\in[0,T)$. Furthermore, a harmonic measure argument shows that 
    $$V_t-\xi_t\to 0\quad\text{or}\quad V_t-\xi_t\to 2\pi,\quad \text{as }t\to T-$$
    (depending on which side of the origin $\g$ passes). Hence, applying  Eq.~\eqref{eq:Jintegrated} to $\widetilde\gamma=\gamma_{[0,t]}$ and taking the limit $t\to T-$ implies
    \begin{equation}
        I(\g)-\l\log\CR(\m D\setminus\g) = H_\l(\t_2-\t_1) + \lim_{t\to T-}J^{\lambda}_{(\theta_1, \theta_2)}(\g_{[0,t]})\geq H_\l(\t_2-\t_1).\label{eq:Jlimit}
    \end{equation}
    Let $\t_1<\t_2\leq \t_1+\pi$, let $(\xi^*_t, V^*_t)$ be the solution to~\eqref{eq:minimaldriver}, and let $\g^\ast$ be the corresponding curve (or, a priori, hull family). Then, 
    \begin{equation*}
        \partial_t (V^*_t-\xi^*_t)=-\sqrt{2\l+4\cot^2((V^*_t-\xi^*_t)/2)},
    \end{equation*}
    so that there does, indeed, exist $T\in(0,\infty)$ so that $\lim_{t\to T-}(V_t^*-\xi_t^*)=0$. Note that, by~\eqref{eq:defHlambda} and~\eqref{eqn::defJ} we have $J^\l_{(\t_1,\t_2)}(\g^\ast_t)=0$ for all $t\in(0,T)$ and thus $$I(\g^*)-\l\log\CR(\m D\setminus\g^*)=H_\l(\t_2-\t_1)$$
    by~\eqref{eq:Jlimit}. As a consequence, we see that $I(\gamma^*)<\infty$ from which we deduce that $\g^*$ is a simple curve. Since $V_t^*-\xi^*_t\to 0$ as $t\to T-$ we conclude that $\g^* \in \mf X (\m D; \ee^{\ii \t_1}, \ee^{\ii \t_2})$.
    Using ~\eqref{eq:Jlimit}, we conclude~\eqref{eqn::energy_inf} and that $\g^*$ is a minimizer. Furthermore, any $\widetilde\g\in\mf X (\m D; \ee^{\ii \t_1}, \ee^{\ii \t_2})$ which minimizes $I(\g)-\l\log\CR(\m D\setminus\g)$ must satisfy $J^{\lambda}_{(\theta_1, \theta_2)}\left(\widetilde\g_{[0,t]}\right)=0$ for all $t$, or equivalently, its driving function must satisfy
    \begin{equation}\label{eq:DEminimialdriver}
        \dot\xi_t = 3\cot((V_t-\xi_t)/2))+H_\lambda'(V_t-\xi_t), \quad \text{a.e.}
    \end{equation}
    If $\t_2<\t_1+\pi$, then the unique continuous solution of~\eqref{eq:DEminimialdriver} is $\xi_\cdot^*$. If $\t_2=\t_1+\pi$, then~\eqref{eq:DEminimialdriver} has two continuous solutions, $\xi_\cdot^*$ and $\xi_\cdot^{**}$, where $\dot\xi_t^{**}=-\dot\xi_t^*$.
\end{proof}

Since $\mc Z_\a$ is a function of $\t : = \t_2 - \t_1$, so we may write $\mc U^\l (\t) = \mc U^\l (\t_1, \t_2)$. The next corollary explains why we do not find this solution in Proposition~\ref{prop:part_wo_int_0}.
\begin{cor}\label{cor::Ulambda_explicit}
Fix $\lambda>0$.
    We let $\theta = \t_2 - \t_1$ and write $\mc U^\lambda (\t) = \mc U^\lambda (\t_1, \t_2)$. Then we have~\eqref{eq:U_l_explicit}: for $\theta\in(0,\pi)$, 
 \begin{align*}
   & \mc U^{\lambda}(\theta) 
    =  \mc U^{\lambda}(2\pi-\theta) = -2\log\sin(\theta/2)
    +\int_{\theta}^{\pi}\sqrt{2\lambda+4\cot^2(u/2)}\, \ud u. 
\end{align*}
In particular, 
it satisfies 
  \begin{equation}\label{eqn::BPZ_nonC2}
    (\mc U'(\t))^2 +2 \cot (\t/2) \mc U'(\t) - \frac{3}{\left(\sin(\t/2)\right)^2} = -3 +2\l\end{equation}
    and has the left derivative $-\sqrt{2\l}$ and right derivative $\sqrt {2\l}$ at $\t = \pi$.
\end{cor}
 In other words, $\mc U^\l$ is not differentiable at $\pi$. That is why we do not see it in Proposition~\ref{prop:part_wo_int_0}. Moreover, the driving function $\xi^*$ in \eqref{eq:minimaldriver} satisfies
$\dot \xi^*_t = \partial_1 \mc U^\lambda (\xi^*_t, V^*_t)$ which is analogous to \eqref{eq:zero_marginal}.
 \begin{proof}
    The expression \eqref{eq:U_l_explicit} follows directly from \eqref{eq:lim_U_lambda} and Proposition~\ref{prop:confRadMinimizer}. It is straightforward to check that it satisfies \eqref{eqn::BPZ_nonC2}.
 \end{proof}
 \begin{proof}[Proofs of Proposition~\ref{prop::semiclassical_pf} and Proposition~\ref{prop::semiclassical_minimizer}]
This follows from Lemma~\ref{lem::semiclassical_Gmu}, Lemma~\ref{lem::semiclassical_Zalpha}, Proposition~\ref{prop:confRadMinimizer} and Corollary~\ref{cor::Ulambda_explicit}.     
\end{proof}

\begin{remark}
When $\lambda=2$, Corollary~\ref{cor::Ulambda_explicit} is consistent with Remark~\ref{rem::onearm}:  
when $\alpha=\alpha_1(\kappa)$ as in~\eqref{eqn::onearm}, we have $\alpha_1(\kappa)\sim -2/\kappa$ and thus 
\begin{align*}
    \mc U^{\l=2}(\theta)=&\lim_{\kappa\to 0+}\kappa\log\mc Z_{\alpha_1(\kappa)}(\theta_1, \theta_2)\\
    =&\begin{cases}
    4\log 2-6\log\sin(\theta/2)+8\log\cos(\theta/4),&\text{if }\theta\in (0,\pi];\\
    4\log 2-6\log\sin(\theta/2)+8\log\sin(\theta/4),&\text{if }\theta\in [\pi,2\pi).
    \end{cases}
\end{align*}
Then
\begin{equation*}
    \partial_{\theta}\mc U^{\l=2}(\theta)=\begin{cases}
    -3\cot(\theta/2)-2\tan(\theta/4),&\text{if }\theta\in (0,\pi);\\
    -3\cot(\theta/2)+2\cot(\theta/4),&\text{if }\theta\in (\pi,2\pi).
    \end{cases}
\end{equation*}
Note that, when $\theta\in (0,\pi)$, the derivative of the right-hand side of~\eqref{eq:U_l_explicit} is 
\begin{align*}
    -\cot(\theta/2)-\frac{2}{\sin(\theta/2)}=-3\cot(\theta/2)-2\tan(\theta/4). 
\end{align*}
\end{remark}

\begin{remark}
    Recall from Remark~\ref{rmk:wo_int} that if one removes the assumption of interchangeability from Proposition~\ref{prop:part_wo_int_0}, then one obtains, in addition to the solutions~\eqref{eq:sol_1_zero} and~\eqref{eq:sol_2_zero}, the solutions 
    $\mc U_{\l,\pm}$, $\l>0$. Note that $\mc U_{\l,+}(\t_1,\t_2)=\mc U^{\l}(\t_1,\t_2)$ for $\t_1<\t_2\leq \t_1+\pi$ and $\mc U_{\l,-}(\t_1,\t_2)=\mc U^\l(\t_1,\t_2)$ for $\t_1+\pi\leq \t_2< \t_1+2\pi$.
    Following Lemma~\ref{lem::semiclassical_Zalpha}, one can take the semiclassical limit of $\mc Z_{\a,\b}$ from Proposition~\ref{prop:remove_int}. For example, if $\l>0$ and $\a\sim-\l/\k$, we have
    $$\lim_{\k\to 0}\k\log\frac{\mc Z_{\a,1}(\t_1,\t_2)}{\mc Z_{\a,1}(0,\pi)}=\mc U_{\l,+}(\t_1,\t_2).$$
    This follows from 
    $$\inf_{\substack{\g\in \mf X (\m D; \ee^{\ii \t_1}, \ee^{\ii \t_2})\\ 0\text{ is to the left of }\g}} \left(I(\g) - \l \log \CR(\m D \setminus \g)\right)=\int_0^{\t_{21}} \Big(\sqrt{2\lambda+4\cot^2(u/2)}-2\cot(u/2)\Big)\, \dd u,$$
    which can be seen by examining the proof of Proposition~\ref{prop:confRadMinimizer}. By a similar analysis one finds
    $$\lim_{\k\to 0}\k\log\frac{\mc Z_{\a,0}(\t_1,\t_2)}{\mc Z_{\a,0}(0,\pi)}=\mc U_{\l,-}(\t_1,\t_2).$$
\end{remark}

\appendix
\section{Proof of Proposition~\ref{prop:radial_comm}}
\label{appendix::commutation_relation}
\begin{proof}[Proof of Proposition~\ref{prop:radial_comm}] 
The strategy of the proof consists of mapping out $\left(\eta^{(1)}_{[0,\vare]}, \eta^{(2)}_{[0,\vare]}\right)$ in two ways, where both curves are parametrized by their intrinsic capacity seen in $\m D$.
We can either 
\begin{itemize}
    \item first map out $\eta^{(1)}_{[0,\vare]}$ by the conformal map $g_\vare^{(1)}$, then by $g_{(\vare, \vare),1}$,
    \item or first map out $\eta^{(2)}_{[0,\vare]}$ by the conformal map $g_\vare^{(2)}$, then by $g_{(\vare, \vare),2}$.
\end{itemize}
We then compare the expansions of $(\t_{(\vare,\vare)}^{(1)}, \t_{(\vare,\vare)}^{(2)})$ in $\vare$ which are expressed in terms of $\mc L_1$ and $\mc L_2$.
 
 More precisely, we first follow the Loewner flow $t \mapsto g_{t}^{(1)}$ until  $t = \vare$. 
    Under this flow,  the radial driving function is $t \mapsto \xi_{t}^{(1)}$. Moreover, $t\mapsto V_t^{(2)}$ satisfies
    $$\partial_t V_t^{(2)} = \cot \left( (V_t^{(2)}  - \xi_{t}^{(1)})/2\right).$$
Hence, for any smooth test function $F \colon (\t_1, \t_2) \mapsto \m R$, we have
$$\m E_{(\t_1, \t_2)} \left[F (\xi_\vare^{(1)}, V_\vare^{(2)})\right] = \left(1 + \vare \mc L_1  + \frac{\vare^2\mc L_1^2}{2} \right) F (\t_1, \t_2) + o (\vare^2).$$

    Now we use $g_{(\vare,\vare),1}$ to map out $\widetilde \eta^{(2)}_{[0,\vare]} : =  g_{\vare}^{(1)} (\eta_{[0,\vare]}^{(2)})$. We note that the capacity of $\widetilde \eta^{(2)}_{[0,\vare]}$ is not $\vare$.
  We compute its capacity and note that
    $$\partial_t \left(g_t^{(1)}\right)'(z)|_{t = 0} =  - \frac{z+\ee^{\ii \t_1}}{z - \ee^{\ii \t_1}} + z \frac{2\ee^{\ii \t_1}}{(z - \ee^{\ii \t_1})^2} = - \frac{z^2 - \ee^{2\ii \t_1} - 2 z \ee^{\ii \t_1}}{(z - \ee^{\ii \t_1})^2}.$$
    Therefore, for small $\vare$,
    \begin{align*}
    \left(g_\vare^{(1)}\right)' (\ee^{\ii \t_2}) 
    & = 1 - \frac{\vare}{2 \left(\sin ((\t_2-\t_1)/2)\right)^2}  +  \ii \vare \cot  ((\t_2-\t_1)/2) + O (\vare^2).
    \end{align*}
    So we have
    $$\abs{\left(g_\vare^{(1)}\right)'(\ee^{\ii\t_2})} = 1 - \frac{\vare}{2 \left(\sin ((\t_2-\t_1)/2)\right)^2} + o (\vare).$$
    The image under $g_\vare^{(1)}$ of a set of capacity $\vare$ near $\ee^{\ii \t_2}$ has capacity  
    $$\vare' = \vare \left(\abs{\left(g_\vare^{(1)}\right)'(\ee^{\ii\t_2})}^2 + o (\vare) \right) = \vare \left(1 - \frac{\vare}{\left(\sin ((\t_2-\t_1)/2)\right)^2} \right) + o (\vare^2).$$
    In particular, $\widetilde \eta^{(2)}_{[0,\vare]}$  has capacity $\vare'$.
Therefore, using the conditions \textbf{(CI)} and \textbf{(DMP)}, we have
\begin{align}
   & \m E_{(\t_1,\t_2)} \left[F (\t_{(\vare,\vare)}^{(1)},\t_{(\vare,\vare)}^{(2)}) \right] \label{eq:exp_double_vare}\\
   = & \left(1 + \vare \mc L_1  + \frac{\vare^2\mc L_1^2}{2} \right) \left(1 + \vare' \mc L_2  + \frac{(\vare')^2\mc L_2^2}{2} \right) F (\t_1, \t_2) + o (\vare^2) \nonumber\\
   = & \left(1 + \vare  (\mc L_1 + \mc L_2) + \left(- \vare \d \mc L_2   + \frac{\vare^2\mc L_1^2}{2}  + \frac{\vare^2\mc L_2^2}{2} + \vare^2 \mc L_1 \mc L_2  \right)\right) F (\t_1, \t_2) + o (\vare^2) \nonumber
\end{align}
where $\d =  \vare \left(\sin ((\t_2 - \t_1)/2)\right)^{-2}$ so that $\vare' = \vare (1 -\d)+o(\vare^2)$.

If we first map out the second curve, then the first curve, and notice that the value of $\d$ is unchanged, we obtain that the above expectation also equals
$$\left(1 + \vare  (\mc L_1 + \mc L_2) + \left(- \vare \d \mc L_1   + \frac{\vare^2\mc L_1^2}{2}  + \frac{\vare^2\mc L_2^2}{2} + \vare^2 \mc L_2 \mc L_1  \right)\right) F (\t_1, \t_2) + o (\vare^2). $$
Comparing these two expansions, the coefficients of the $\vare^2$-order terms have to coincide, we obtain the condition
$$
    \comm{\mc L_1, \mc L_2} := \mc L_1 \mc L_2 - \mc L_2 \mc L_1= \frac{\d}{\vare} (\mc L_2 - \mc L_1 )= \frac{\mc L_2 - \mc L_1 }{\left(\sin ((\t_2- \t_1)/2)\right)^2} 
$$
as claimed.
\end{proof}

    \section{Euler's hypergeometric differential equations}
\label{app:hypergeometric}

\begin{lem}\label{lem::Euler_initial}
Fix $\kappa\in (0,8)$ and $\alpha\in\mathbb{R}$. We consider Euler's hypergeometric differential equation~\eqref{eqn::Euler_initial}: 
\begin{equation*}
\begin{cases}
u(1-u)\phi''(u)-\dfrac{3\kappa-8}{2\kappa}(2u-1)\phi'(u)+\dfrac{8\alpha}{\kappa}\phi(u)=0,\quad u\in(0,1);\\
\phi(1/2)=1, \quad \phi'(1/2)=0.
\end{cases}
\end{equation*}
There is a unique solution $\phi_{\alpha}$ to~\eqref{eqn::Euler_initial} in $C^2(0,1)$ and the solution $\phi_{\alpha}$ is continuous in $\alpha$. Furthermore, 
\begin{itemize}
    \item when $\alpha=\alpha_0:=1-\kappa/8$, we have    \begin{equation}\label{eqn::Euler_sol_critical}
        \phi_{\alpha_0}(u)=\left(4u(1-u)\right)^{4/\kappa-1/2};
    \end{equation} 
    \item when $\alpha< 1-\kappa/8$, we have $\phi_{\alpha}(u)>0$ for all $u\in (0,1)$;
    \item when $\alpha>1-\kappa/8$, there exists $u\in (0,1)$ such that $\phi_{\alpha}(u)\le 0$. 
\end{itemize}
\end{lem}
\begin{proof}
    By a direct calculation, we see that~\eqref{eqn::Euler_sol_critical} satisfies~\eqref{eqn::Euler_initial} when $\alpha=\alpha_0$. We write
    \[\alpha_0=1-\kappa/8,\quad \phi_{\alpha_0}(u)=\left(4u(1-u)\right)^{4/\kappa-1/2},\quad \phi(u)=\phi_{\alpha_0}(u)f(u).\]
    Then Eq.~\eqref{eqn::Euler_initial} becomes
    \begin{equation}\label{eqn::Euler_f}
        \begin{cases}
        u(1-u)f''(u)-\dfrac{\kappa+8}{2\kappa}(2u-1)f'(u)+\dfrac{8}{\kappa}(\alpha-\alpha_0)f(u)=0,\quad u\in (0,1);\\
        f(1/2)=1,\quad f'(1/2)=0.
        \end{cases}
    \end{equation}
    We write $y_1=f$ and $y_2=f'$, then Eq.~\eqref{eqn::Euler_f} becomes
    \begin{equation}\label{eqn::Euler_f_y12}
        \begin{cases}
        y_1'=\mathcal{F}_1(u, y_1, y_2):=y_2;\\
        y_2'=\mathcal{F}_2(u, y_1, y_2):=\dfrac{8}{\kappa}\dfrac{\alpha_0-\alpha}{u(1-u)}y_1+\dfrac{\kappa+8}{2\kappa}\dfrac{(2u-1)}{u(1-u)}y_2;\\
        y_1(1/2)=1,\quad y_2(1/2)=0.
        \end{cases}
    \end{equation}
    The functional $(\mathcal{F}_1, \mathcal{F}_2)$ is continuous in $\Lambda=(0,1)\times \mathbb{R}\times\mathbb{R}$ and satisfies a local Lipschitz condition with respect to $(y_1, y_2)$ in $\Lambda$. Thus, the initial value problem~\eqref{eqn::Euler_f_y12} has exactly one solution, and the solution can be extended up to the boundary of $\Lambda$. This gives the unique solution $\phi_{\alpha}$ to~\eqref{eqn::Euler_initial} and the solution $\phi_{\alpha}$ is continuous in $\alpha$. 

    Next, let us check the positivity of the solution. There are two cases.

     \noindent \textbf{Case 1:} $\alpha\le\alpha_0$. In this case, $\mathcal{F}_1$ is increasing in $y_2$ and $\mathcal{F}_2$ is increasing in $y_1$. Thus we have a comparison principle for the unique solution to~\eqref{eqn::Euler_f_y12}. In particular, the solution $f$ to~\eqref{eqn::Euler_f} is decreasing in $\alpha$ as long as $\alpha\le \alpha_0$. Consequently, the unique solution $\phi_{\alpha}$ to~\eqref{eqn::Euler_initial} is decreasing in $\alpha$ as long as $\alpha\le \alpha_0$ and
    $\phi_{\alpha}(u)\ge \phi_{\alpha_0}(u)>0$ for all $u\in (0,1)$ when $\alpha\le\alpha_0$.

\noindent \textbf{Case 2:} $\alpha>\alpha_0$. We prove by contradiction and assume $f(u)\ge 0$ for all $u\in (1/2,1)$. As $f'(1/2)=0$ and $f''(1/2)=32(\alpha_0-\alpha)/\kappa<0$, there exists $\eps>0$ small such that 
\[-\delta:=f'(1/2+\eps)<0.\]
The ODE in~\eqref{eqn::Euler_f} implies that 
\[f'' (u) \le \left(\frac{\k+8}{2\k}\right) \frac{2u-1}{u(1-u)} f'(u),\quad u\in(1/2,1).\]
Using Gr\"onwall's inequality, we obtain, for $u\in(1/2+\eps,1)$, 
\[f'(u) \le f'(1/2 + \vare) \exp \left(\int_{1/2 + \vare}^u  \left(\frac{\k+8}{2\k}\right) \frac{2s-1}{s(1-s)} \dd s  \right) = - \d \left(\frac{u(1-u)}{1/4-\eps^2}\right)^{-\frac{\k+8}{2\k} }.\]
Since $\kappa\in (0,8)$, we have $\frac{\kappa+8}{2\kappa}>1$. The derivative $f'$ is not integrable as $u\to 1$. This implies that $f(u)\to -\infty$ as $u\to 1$, which contradicts our assumption and completes the proof.  
\end{proof}

\begin{lem}\label{lem::Euler_positive}
    Fix $\k\in(0,8)$ and $\a\in\m R$. Consider Euler's hypergeometric differential equation:
    \begin{equation}\label{eq:eulerhyp}
        u(1-u)\phi''(u)-\frac{3\k-8}{2\k}(2u-1)\phi'(u)+\frac{8\alpha}{\k}\phi(u)=0,\quad u\in(0,1).
    \end{equation}
    We have the following:
    \begin{itemize}
        \item When $\alpha>1-\k/8$, there exists, for each solution $\phi$ to \eqref{eq:eulerhyp}, $u\in(0,1)$ such that $\phi(u)\leq 0.$
        \item When $\alpha=1-\k/8$, the only positive solution (up to a multiplicative constant) of \eqref{eq:eulerhyp} is
        \begin{equation*}
            \phi_{\a_0}(u)=(4u(1-u))^{4/\k-1/2}.
        \end{equation*}
        \item When $\alpha<1-\kappa/8$, the positive solutions of \eqref{eq:eulerhyp}, are (up to a multiplicative constant) all of the form
        \begin{equation*}
            \phi(u)=\b\Phi^L(u) + (1-\b)\Phi^R(u),\quad \b\in[0,1],
        \end{equation*}
        with $\Phi^L$ and $\Phi^R$ as in Remark~\ref{rem::cr_left}.
    \end{itemize}
\end{lem}
\begin{proof}
    We first consider $\a\geq 1-\k/8.$ Let $\phi_\a$ be as in Lemma~\ref{lem::Euler_initial}, and let $\psi_\a$ be the solution of \eqref{eq:eulerhyp} satisfying $\psi_\a(1/2)=0$ and $\psi_\a'(1/2)=1$. It is easily verified that 
    \begin{equation}\label{eq:tildephi}
        \psi_\a(u)=-\psi_\a(1-u),\quad u\in(0,1).
    \end{equation}
     Since $\phi_\a$ and $\psi_\a$ are linearly independent, any solution of \eqref{eq:eulerhyp} can be expressed as a linear combination of them. When $\a>1-\k/8$, it follows from Lemma \ref{lem::Euler_initial} and Eq. \eqref{eq:tildephi} that there is no positive solution of \eqref{eq:eulerhyp}. Similarly, when $\a=\a_0=1-\k/8$, it follows from Lemma \ref{lem::Euler_initial} and Eq. \eqref{eq:tildephi} that the only positive solutions of \eqref{eq:eulerhyp} are $C\phi_{\a_0}$, with $C>0$.\par
     Now consider $\a<1-\k/8$. As stated in Remark~\ref{rem::cr_left}, $\Phi^L$ and $\Phi^R$ are $C^2$-solutions to \eqref{eq:eulerhyp} ($\Phi^L,\Phi^R\in C^2((0,1))$ follow from the same argument as for $\Phi$ in Lemma~\ref{lem::CR_expectation} using Remark~\ref{rem::hypoelliptic}). Moreover, $0<\Phi^L(u),\Phi^R(u)\leq \Phi(u)$, for $u\in(0,1)$, and
     \begin{equation*}
         \Phi^L(0)=1,\quad \Phi^L(1)=0,\quad\text{and}\quad\Phi^R(0)=0,\quad\Phi^R(1)=1.
     \end{equation*}
     Since $\Phi^L$ and $\Phi^R$ are linearly independent, any solution $\phi$ of \eqref{eq:eulerhyp} can be decomposed as $\phi = C_L \Phi^L +C_R\Phi^R,$ for some real constants $C_L$ and $C_R$. If $C_L<0$, then $\phi(u)<0$ for $u>0$ small. So, for a positive solution $\phi$ we must have $C_L\geq 0$. Similarly, we deduce that $C_R\geq 0$, and clearly we may not have $C_L=C_R=0$. This finishes the proof. 
\end{proof}

\section{Solutions of the separable equation \eqref{eqn::separable}}
\begin{lem}\label{lemma:solseparable}
    Let $C\in\m R$ and consider the separable differential equation
    \begin{equation}(8+2C\sin^2(\t/2))f'(\t) = (f(\t))^2+C+4.\label{eq:sep}\end{equation}
    The following are the only $C^1$ solutions of~\eqref{eq:sep} defined on $(0,2\pi)$ for different values of $C\in\m R.$ 
    \begin{itemize}
        \item If $C+4>0$, then the solutions are of the form\begin{equation}f(\t)=\begin{cases}\sqrt{C+4}\tan(\frac{\arctan(\frac{\sqrt{C+4}}{2}\tan(\t/2))}{2}+D-\pi/4)& \t\in(0,\pi)\\
        \sqrt{C+4}\tan(D)& \t=\pi\\
        \sqrt{C+4}\tan(\frac{\arctan(\frac{\sqrt{C+4}}{2}\tan(\t/2))}{2}+D+\pi/4)& \t\in(\pi,2\pi)\\
        \end{cases}\label{eq:tansol}\end{equation} 
        for any $D\in[-\pi/4,\pi/4].$
        \item If $C+4=0$, then the solutions are of the form
        $$f(\t)=0\quad\text{and}\quad 
        f(\t)=\begin{cases}
            (D-\frac{1}{4}\tan(\t/2))^{-1}& \t\in(0,\pi)\\
            0& \t=\pi\\
            (E-\frac{1}{4}\tan(\t/2))^{-1}& \t\in(\pi,2\pi)\\
        \end{cases}$$
        for any $D\leq 0$ and $E\geq 0$.
        \item If $C+4<0$, then the solutions are of the form
        $$f(\t)=\sqrt{-C-4}\quad\text{and}\quad f(\t)=-\sqrt{-C-4}.$$
    \end{itemize}
\end{lem}
\begin{proof}
    Note that 
    $$\frac{1}{8+2C\sin^2(\t/2)}=\frac{1}{4+(C+4)\tan^2(\t/2)}\frac{\ud}{\ud \t}\tan(\t/2).$$
    Thus, we can then rewrite~\eqref{eq:sep} as
    \begin{equation}\frac{\ud f}{f^2+(C+4)} = \frac{\ud \tan(\t/2)}{4+(C+4)\tan^2(\t/2)},\label{eq:sepint}\end{equation}
    under the assumptions $(f(\t))^2\neq -(C+4),$ $\t\neq\pi$ and $(C+4)\tan^2(\t/2)\neq -4.$\par
    \noindent\textbf{Case 1: }Suppose $C+4>0$. By the Picard-Lindelöf theorem Eq.~\eqref{eqn::separable} has a unique local solution $f$ around $\theta_0\in(0,2\pi)$ for each initial condition $f(\t_0)=f_0$. Integrating both sides of~\eqref{eq:sepint} gives
    $$\arctan(f(\t)/\sqrt{C+4})=
    \begin{cases}\frac{1}{2}\arctan(\frac{1}{2}\sqrt{C+4}\tan(\t/2))+D-\pi/4&\t\in(0,\pi)\\
    \frac{1}{2}\arctan(\frac{1}{2}\sqrt{C+4}\tan(\t/2))+D+\pi/4&\t\in(\pi,2\pi)\end{cases}$$
    for some $D$. Note that the left-hand side takes values in $(-\pi/2,\pi/2)$ and that this is true for the right-hand side only if $D\in[-\pi/4,\pi/4]$. Thus the only global solutions of~\eqref{eq:sep} in the case $C+4>0$ are the ones given in~\eqref{eq:tansol}.\par
    \noindent\textbf{Case 2:} Suppose $C+4=0$. Picard-Lindelöf guarantees existence and uniqueness of local solutions to the initial value problem away from $\t=\pi$. The solutions on $(0,\pi)$ are of the form $f\equiv 0$ or 
    $$f(\t)=\frac{1}{D-\frac{1}{4}\tan(\t/2)}$$
    for some $D\in\m R$. In order for the latter to be defined on all of $(0,\pi)$ we require $D\leq 0$. Similarly, the solutions on $(\pi,2\pi)$ are of the form $f\equiv 0$ or 
    $$f(\t)=\frac{1}{E-\frac{1}{4}\tan(\t/2)}$$
    for some $E\geq 0$. To find solutions on $(0,2\pi)$ we patch the solutions above. For this we note that the left and right derivatives of $f(\t)=(D-\frac{1}{4}\tan(\t/2))^{-1}$ at $\pi$ are both $2$, regardless of $D$.\par
    \noindent\textbf{Case 3:} Finally, fix $C+4<0$ and let $\theta_0$ be the unique solution to $8+2C\sin^2(\t/2)=0$ on $(0,\pi)$. Then Picard-Lindelöf guarantees existence and uniqueness of local solutions to the initial value problem around each $\theta\in(0,2\pi)\setminus\{\t_0,2\pi-\t_0\}$. For the initial values $\pm\sqrt{-C-4}$ we have constant solutions. For the other initial values, we integrate~\eqref{eq:sepint} and find solutions of the form
    $$\log\frac{|f(\theta)-\sqrt{-C-4}|}{|f(\theta)+\sqrt{-C-4}|}=\frac{1}{2}\log\frac{|\sqrt{-C-4}\tan(\t/2)+2|}{|\sqrt{-C-4}\tan(\t/2)-2|}+D,$$
    $D\in\m R$. More concretely, if $\t_1\in(0,2\pi)\setminus\{\t_0,2\pi-\t_0\}$ and $f(\t_1)\in\m R\setminus\{\pm\sqrt{-C-4}\}$ then 
    $$f(\theta)=T\bigg(\frac{1}{2}\log\frac{|\sqrt{-C-4}\tan(\t/2)+2|}{|\sqrt{-C-4}\tan(\t/2)-2|}+D\bigg)$$
    for $\t$ in the connected component of $(0,2\pi)\setminus\{\t_0,2\pi-\t_0\}$ containing $\t_1$, where $T$ is the inverse of $\log\frac{|x-\sqrt{-C-4}|}{|x+\sqrt{-C-4}|}$ on the connected component of $\m R\setminus\{\pm\sqrt{-C-4}\}$ containing $f(\t_1)$. An explicit computation shows that all such solutions satisfy $|f'(\t)|\to\infty$ as $\t\to\t_0$ or $\t\to 2\pi-\t_0$ and therefore they cannot be extended to $C^1$ solutions on $(0,2\pi).$ 
\end{proof}
\textbf{Acknowledgments. }
\begin{itemize}
    \item We thank Mo Chen, Yu Feng, Chongzhi Huang, Eveliina Peltola, Fredrik Viklund, and Lu Yang for helpful discussions.
    We thank Minjae Park for the simulation in Figure~\ref{fig::simulation}. We thank the anonymous referees for their constructive comments on a previous version of the paper.
    \item E.K. is supported by a grant from the Knut and Alice Wallenberg Foundation. E.K. acknowledges the hospitality of L'Institut des Hautes Études Scientifiques where part of this work was carried out.
    \item Y.W. is funded by the Swiss State Secretariat for Education, Research and Innovation [MB25.00004] and the Simons Foundation International [SFI-MPS-PP-00012621-19].
    \item H.W. is supported by New Cornerstone Investigator Program 100001127. Part of this work was carried out during H.W.'s visit at L'Institut des Hautes Études Scientifiques with support from ``K.C. Wong Education Foundation''. H.W. is partly affiliated with Yanqi Lake Beijing Institute of Mathematical Sciences and Applications, Beijing, China.  
\end{itemize}

\textbf{Conflict of interest.} The authors have no competing interests to declare that are relevant to the content of this article.

\textbf{AI declaration.} No use of AI in the preparation of the article.

\end{document}